\documentclass[11pt,reqno,a4paper,psamsfonts]{amsart}
\usepackage[utf8]{inputenc}
\usepackage[T1]{fontenc}
\usepackage[english]{babel}
\usepackage{amssymb,amsmath,amscd,enumerate,amsthm}
\usepackage[psamsfonts]{amsfonts}
\usepackage{latexsym}
\usepackage{amsbsy}
\usepackage{ulem}
\usepackage[left=3.cm,right=3.cm,top=3.cm,bottom=3.cm]{geometry}
\usepackage[hidelinks]{hyperref}
\usepackage{todonotes}
\usepackage{frcursive}

\usepackage{pdfsync}
\usepackage{xcolor,hyperref}
\usepackage{relsize,exscale}
\hypersetup{backref,colorlinks=true,linkcolor=blue,citecolor=blue}
\usepackage{pdfsync}
\usepackage{enumerate}
\usepackage{verbatim}
\usepackage{color}
\definecolor{green}{rgb}{0,0.5,0.0}
\definecolor{red}{rgb}{1,0,0}
\definecolor{marron}{rgb}{0.5,0.2,0.2}
\usepackage[all,color]{xy}
\usepackage{graphicx}

\newcommand{\diam}{\diamond}
\newcommand{\C}{{\mathbb{C}}}

\newcommand{\F}{{\mathcal{F}}}
\newcommand{\R}{{\mathbb{R}}}%
\newcommand{\Z}{{\mathbb{Z}}}

\let\CAL=\mathcal%
\def\mathcal#1{{\CAL#1}}%
\newcommand{\N}{\mathbb{N}}

\newcommand{\A}{{\mathsf{A}}}

\newcommand{\fcon}{\looparrowright}
\newcommand{\Fcon}{\stackrel{\F^\sharp}{\fcon}}
\newcommand{\Fconl}{\stackrel{\F^\sharp}{\fconl}}
\newcommand{\Gcon}{\stackrel{\G}{\fcon}}
\newcommand{\Gconl}{\stackrel{\G}{\fconl}}
\newcommand{\fconl}{\looparrowleft}

\newcommand{\G}{\mathcal{G}}

\renewcommand{\emph}[1]{{\textbf{#1}}}

\newcommand{\mf}[1]{{\mathfrak{#1}}}
\newcommand{\mr}[1]{{{\mathrm{#1}}}}
\newcommand{\mc}[1]{{\mathcal{#1}}}
\newcommand{\mb}[1]{{\mathbb{#1}}}

\newcommand{\mbf}[1]{\mathbf{#1}}
\let\msf\msl

\newcommand{\limproj}{\underleftarrow\lim}
\newcommand{\limind}{\underrightarrow\lim}

\newcommand{\iso}{{\overset{\sim}{\longrightarrow}}}

\newcommand{\inte}[1]{\overset{\circ}{{#1}}}
\newtheorem{teo}{Theorem}[section]

\newtheorem{lema}[teo]{Lemma}

\newtheorem{prop}[teo]{Proposition}
\newtheorem{defin}[teo]{Definition}
\newtheorem{cor}[teo]{Corollary}

\newtheorem{obs2}[teo]{Remark}
\newtheorem{recap2}[teo]{Recapitulation}
\newtheorem{ex2}[teo]{Example}
\newtheorem{question}[teo]{Question}

\newtheorem{idproof2}[teo]{\textit{Idea of the proof}}
\newenvironment{obs}{\begin{obs2}\rm}{\hfill\qed\end{obs2}}

\def\bibartp#1#2#3#4#5#6#7#8
{\bibitem{#1} {\sc #2}, {\it #3}, {#4}, {\bf t. #5}, pages { #6} to
{ #7}, {({#8})}}
\def\bibart#1#2#3#4#5#6
{\bibitem{#1} {\sc #2}, {\it #3}, {#4}, {\bf #5}, {({#6})}}
\def\bibliv#1#2#3#4#5
{\bibitem{#1} {\sc #2}, {\it #3}, {#4}, {({#5})}}
\def\bibaart#1#2#3#4
{\bibitem{#1} {\sc #2}, {\it #3}, {#4}}
\definecolor{vert}{rgb}{0,0.46,0}

\usepackage{amsfonts,mathrsfs}
\newcommand{\ms}{\mathscr}

\title[Topology of singular foliation germs in $\C^2$]{Topology of  singular foliation germs in $\C^2$}
\date{\today}
\author{David Mar\'{\i}n, Jean-Fran\c{c}ois Mattei and {E}liane Salem}

\thanks{D. Mar\'{\i}n acknowledges financial support from the Spanish Ministry of Science, Innovation and Universities, through grant 
PID2021-125625NB-I00 and
 by the Agency for Management of University and Research Grants of Catalonia through the grant 2021SGR01015.
 This work is also supported by the Spanish State Research Agency, through the Severo Ochoa and Mar\'ia de Maeztu Program for Centers and Units of Excellence in R\&D (CEX2020-001084-M). E. Salem acknowledges the support of CNRS through a delegation.}
\address{Departament de Matem\`{a}tiques\\ Universitat Aut\`{o}noma de Barcelona \\ E-08193 Cerdanyola del Vall\`es (Barcelona)\\ Spain\\ \newline \indent Centre de Recerca Matem\`atica, Campus de Bellaterra, E-08193 Cerdanyola del Vall\`es, Spain} \email{David.Marin@uab.cat}
\address{Institut de Math\'{e}matiques de Toulouse\\ Universit\'{e} Paul Sabatier\\ 118, Route de Narbonne\\ F-31062 Toulouse Cedex 9, France} \email{jean-francois.mattei@math.univ-toulouse.fr}
\address{Sorbonne Universit\'e, Universit\'e de  Paris, CNRS,  Institut de Math\'ematiques de Jussieu - Paris Rive Gauche, F-75005 Paris, France}
\email{eliane.salem@imj-prg.fr}

\subjclass[2010]{Primary 37F75; Secondary 32M25, 32S50, 32S65, 34M} 

\keywords{Complex dynamical systems, complex ordinary differential equations, holomorphic vector fields, singular foliations, holonomy}

\usepackage{tikz}
\usetikzlibrary{calc}
\tikzset{elliparc/.style args={#1:#2:#3}{%
insert path={++(#1:#3) arc (#1:#2:#3)}}}

\begin{document}

\maketitle

\tableofcontents

\section*{Introduction}

The object of this survey is to give an overview 
 on the topology of singularities of holomorphic foliation germs on $(\C^2,0)$. Such a foliation germ $\F$ can be defined by a differential form $\omega=a(x,y)dx + b(x,y)dy$ or equivalently by a vector field $X=b(x,y)\partial_x - a(x,y)\partial_y$, with $a(x,y)$ and $b(x,y)$  elements of the ring $\mc O_{\C^2,0}$ of  germs of holomorphic functions vanishing at the origin of $\C^2$, without common divisor. A leaf of the foliation $\F_U$ that represents the germ $\F$  on an open neighborhood $U$ of $0$ on which $a(x,y)$ and $b(x,y)$ are defined and holomorphic, is a maximal immersed Riemann surface in $U\setminus\{0\}$ that is solution of the equation $\omega=0$, or equivalently that is tangent to $X$. Two foliation germs $\F$ and $\G$ are topologically equivalent (or $\mc C^0$-conjugated) if there is a homeomorphism between two open neighborhoods $U$ and $V$ of $0$ that sends any leaf of $\F_U$ to a leaf of $\G_V$. This notion is an equivalence relation on the set of all foliation germs, and the final 
objective of a topological classification would be to obtain a list of foliation germs containing an element  of each topological class, will minimal redundancy.

Special leaves $L$ of $\F_U$ are those such that $L\cup \{0\}$ is a closed  analytic curve in $U$; they are topologically characterized by the property of being closed in $U\setminus\{0\}$.
However for the germ $\F$ this notion takes a rigorous sense only as  germ of curve at $0$, called separatrix of $\F$.
Indeed there are germs $\F$ for which any leaf of $\F_U$ contains a  separatrix (as punctured curve germ) and there are leaves of $\F_U$ containing an infinity of separatrices, see Remark~\ref{inf-sep} and Figure~\ref{fig-inf-sep}.\\

We begin with a historical approach to the topological study of singularities of foliation germs.
The first result on the topology of separatrices is the  famous paper \cite{Briot-Bouquet}  (1856) of C. Briot and J.C. Bouquet  where the authors, after introducing the notion of Newton's polygon of a foliation germ, establish  an algorithm giving the Puiseux pairs of any separatrix, thus determining their topological type. Later A. Seidenberg  establishes in \cite{Seidenberg} (1968)   that for a given  foliation germ,  the number of topological types of separatrices is finite. 
For this he proves that after a finite number of blowing-ups over the origin, the foliation is locally defined, at each singular point on the exceptional divisor, by a vector field  whose  linear part is diagonalizable with a non-zero eigenvalue. 
After supplementary blowing-ups, each singular point becomes reduced, i.e. the quotient of eigenvalues is  not  a positive rational number.  A simple proof of this result is obtained by A. van den Essen \cite{vdE}, \cite[Appendice~I]{MatMou} (1979) and F. Cano \cite{FCano} (1986). The existence of a separatrix is established in 1982 by C. Camacho and P. Sad in \cite{CS}. Alternative proofs of this result are given later by J. Cano \cite{JoseCano} (1997), M. Toma \cite{Toma} (1999).\\

The simplest reduced foliation germs are given by $1$-differential  forms $\omega$ with linear part $\omega_L=\lambda_1 y dx +\lambda_2 xdy$, with $\lambda_1,\lambda_2\in\C^\ast$ and $\lambda_1/\lambda_2\notin\R$. They are studied by H. Poincar\'e and H. Dulac \cite{Dulac} (1904) proving
the existence of local holomorphic coordinates for which $\omega=\omega_L$, i.e. $\F$ is linearizable.  All these foliation germs belong to the same topological class, see \S\ref{PoincareType}.  The topological classification of reduced foliation germs given by a $1$-form $\omega$ with linear part $pydx+qxdy$ with $p,q\in\N^\ast$, called resonant saddles, is made by C. Camacho and P. Sad in \cite{Camacho-Sad} (1982), who obtain a countable set of topological classes and show that the topological type is determined by a finite jet of $\omega$. The remaining reduced foliation germs are the saddle-nodes, that are given by a differential form $\omega$ with  $\omega_L=ydx$,   whose  topological classification  is made by P.M. Elizarov  in  \cite{Elizarov} (1990), and the non-linearizable foliation germs given by $\omega$ with $\omega_L=\lambda_1 y dx +\lambda_2 xdy$ and $\lambda_1/\lambda_2\in\R\setminus\mb Q$. In this last case the foliation germs which are linearizable form a topological class, and the classification problem  of those which are  non-linearizable is equivalent to that of the topological classification of Cremer biholomorphisms of one variable.\\

For non-reduced foliation germs, C. Camacho, A. L. Neto and P. Sad highlight in  \cite{CLNS} (1984) an important class of foliation germs: the  so called generalized curves for which no saddle-node singularity appear after reduction.
They prove the topological invariance of the Milnor number of any foliation germ (defined as the complex codimension  of the ideal of $\mc O_{\C^2,0}$ generated by the coefficients of the $1$-form with isolated singularity defining  the foliation germ). Using this invariant they characterize topologically generalized curves. They also prove that generalized curves have same reduction as their separatrices, but this topological property does not characterize them \cite{MS} (2004). Few topological results are known for non-generalized curves, see \S\ref{Loic}.\\

After these results and in the 1990's the study of the topology of non-reduced foliation germs was ``local'' in the sense of a search for  invariants under ``small $\mc C^0$-trivial deformations'' defined by  families of germs of differential forms parametrized by a germ of space. 
An important step in this study  is the topological rigidity property of non solvable groups of biholomorphism germs in one variable proved by 
A.A. Shcherbakov \cite{Shcherbakov} (1984) and completed by I. Nakai \cite{nakai} (1994). 
This one makes it possible  
to put new (generic) assumptions on the differential form defining a foliation germ (the non-solvability of some holonomy groups), that induce transverse holomorphy properties on any homeomorphism conjugating such foliations germs, see \S \ref{bigdyncomp}. 
Under rigidity assumptions D. Cerveau and P. Sad prove in \cite{CerveauSad} (1986) that  the holonomy of the exceptional divisor is   invariant under $\mc C^0$-trivial deformation for generic foliation  germs reduced by a single blowing-up; they also conjecture that this holds for any pair of topologically conjugated  generalized curves not necessarily contained in a $\mc C^0$-trivial deformation. The classification of generic logarithmic foliations established by E. Paul in \cite{Paul} (1989) shows that the rigidity assumption is necessary, see Theorem \ref{thmpaul}. 
In fact, it is not sufficient to consider the holonomies 
to obtain a complete topological classification.
Other continuous topological invariants under $\mc C^0$-trivial deformations exist and are completely specified by two of us in  \cite{MS2} (1997) for generalized curves satisfying some additional assumptions.\\

The object of this paper is to describe, giving ideas of the proofs, results obtained by the authors  on the topology of leaves,  the structure of leaves space and criteria of conjugacy for any two foliation germs not necessarily contained in a $\mc C^0$-trivial deformation. They are based on the construction of the monodromy, a new topological invariant of geometric and dynamical nature.
All these results  lead to prove, for two  generalized curve  satisfying some generic assumptions,
that the existence of a conjugating homeomorphism between neighborhoods of  the origin of $\C^2$ is equivalent to the existence of a conjugating homeomorphism 
between neighborhoods of the exceptional divisors   after reduction of singularities.
 Clearly this property (Excellence Theorem~\ref{exTheorem})  gives a positive answer to the Cerveau-Sad conjecture. It is also the starting point of a series of three works \cite{MMS, MMS2, MMS3} (2020-2022) that enables us  to carry out the topological classification of generic generalized curves, which we summarize in \S\ref{moduli}.\\

\noindent This paper is structured as follows:
\begin{enumerate}[\hphantom{\S\S}\S]
\item[\S 1] We study the simplest invariant sets of foliation germs: separatrices and separators. Separators appear when there are nodal singularities or non-invariant irreducible components of the exceptional divisor of the reduction of the foliation germ. They divide any small neighborhood of $0\in\C^2$ into disjoint invariant sets that we call dynamical components. 
 We compare it with the decomposition into conical geometric blocks of the complementary of the separatrices obtained using  3-manifold theory~\cite{LMW}.
\item[\S 2]  The key topological property of the leaves of a generalized curve is their incompressibility in the complementary of a suitable union of separatrices called appropriate curve (Incompressibility Theorem \ref{IncomprThm}). The introduction of the foliated connectedness notion and a foliated Van Kampen Theorem allow to prove incompressibility thanks to a foliated block decomposition of the space guided by the decomposition into geometric blocks.
\item[\S 3] We discuss an example that shows the necessity of non-isolated separatrices in an appropriate curve. 
\item[\S 4] We first give examples of the use of the ends of leaves space in the analytical study of a reduced foliation. We then highlight the complex structure of the leaf space of the foliation induced by $\F_U$ on the universal covering of $U\setminus S$, for $U$ in a suitable system of neighborhoods of an appropriate curve $S$. This allows to introduce the notion of monodromy of a foliation germ as an action on the leaf spaces system of the fundamental group of the complementary of $S$. Together with the Camacho-Sad indices, the monodromy is a complete topological invariant, see Classification Theorem~\ref{ClassThm}.
\item[\S 5] We describe different types of dynamical components that permit to state optimal assumptions for topological invariance of the Camacho-Sad indices at the singular points in the exceptional divisor of the reduction.
\item[\S 6] We explain the proof of  Excellence Theorem~\ref{exTheorem} and we 
state the topological classification of generalized curves.
\end{enumerate}

\section{Separatrices and separators}\label{Topologie des feuilles}
A consequence of the long exact sequence associated to the Milnor fibration \cite{Milnor} of a germ of function with isolated singularity  is that the fundamental group of a regular level set injects into the fundamental group of the complement of the singular fiber (\emph{incompressibility property}). In the more general setting  of generalized curves,
we will see that incompressibility property of the leaves
in the complementary of a suitable invariant curve (containing the isolated separatrices)
also holds.

\subsection{Graph decomposition of the complement of a germ of curve}\label{decompCompCurve}
Let $(S,0)$ be a germ of analytic curve $S$ at the origin of $\C^2$ with reduced equation $f=0$. It is well known that 
the \emph{full Milnor tube of $S$}
\begin{equation*}\label{tubeMilnor}
\mb T^\ast_{\varepsilon,\eta}:=\mb T_{\varepsilon,\eta}\setminus S\,,\quad \hbox{with}\quad
\mb T_{\varepsilon,\eta}:=\{\vert f(z)\vert \leq \eta, \Vert z\rVert \leq \varepsilon\}\subset\C^2\,,\quad 0<\eta\ll \varepsilon\ll 1\,,
\end{equation*}
is a retract by deformation of the complementary of the curve in the ball $\{ \Vert z\rVert \leq \varepsilon\}$ and also that it is a cone on the \emph{empty Milnor tube}  
\[\delta\mb T_{\varepsilon,\eta}:=\{\vert f(z)\vert =\eta, \Vert z\rVert \leq \varepsilon\},\] 
i.e. there is a homeomorphism  $\psi:\mb T^\ast_{\varepsilon,\eta}\iso\, \delta\mb T_{\varepsilon,\eta} \times (0,1]$ such that $\psi(z)=(z,1)$ if $z\in \delta \mb T_{\varepsilon,\eta}$. The structure of $\delta\mb T_{\varepsilon,\eta}$ as graphed 3-manifold is described by various authors \cite{Waldhausen,JacoShalen,Johannson} and we adopt here the presentation of  \cite[Theorems 1.2.3 and 1.2.6]{LMW}, see also \cite[\textsection 3.2]{MM2}:
\begin{itemize}
\item \it There is a finite collection $\mathscr T$ of 2-tori  embedded in $\delta\mb T_{\varepsilon,\eta}$, unique up to isotopy, that decomposes $\delta\mb T_{\varepsilon,\eta}$ into incompressible compact submanifolds   with boundary $K_ \alpha$, each of them being endowed with a (unique) Seifert fibration $\mathscr S_\alpha$, such that if $T\in\mathscr T$ is the intersection $K_\alpha\cap K_\beta$, then \begin{enumerate}
\item $T$ is incompressible in $K_\alpha$, in $K_\beta$ and it is not homotopic to a connected component of the boundary of $\delta\mb T_{\varepsilon,\eta}$,
\item in $T$ the general fibers of $\mathscr S_\alpha$ and of $\mathscr S_\beta$ are not homotopic.
\end{enumerate}
\end{itemize}
We will describe a way to obtain the collection $\mathscr T$ using the desingularization map $E_S: M_S\to \C^2$  of $S$.
We recall that the \emph{dual graph $\A_C$ of a curve $C$ with normal crossings}, is a graph whose vertices  are the irreducible components (compact or not) of this curve, two vertices being joined by an edge if and only if the corresponding irreducible components of the curve intersect. 
For simplicity we will denote by $\A_S$ the dual graph $\A_{\mc D_S}$ of the \emph{total transform} 
\[ \mc D_S:=E_S^{-1}(S)\,.\]  Clearly $\A_S$ is a tree. We equip it with the metric for which the edges have length equal to 1. The \emph{valence} of a vertex is the number of edges attached to it.  A \emph{chain} in $\A_S$ is a geodesic of $\A_S $ joining two vertices of valence at least three called \emph{end vertices of the chain}, all the other vertices being of valence $2$.
A \emph{branch of $\A_S$} is  a geodesic that joins a vertex of valence at least three  called \emph{attaching vertex}
 to a vertex of valence one, called \emph{end vertex}, the remaining vertices being of valence $2$. It is called \emph{dead branch} if its end vertex corresponds to a compact irreducible component of $\mc D_S$ and \emph{simple branch} otherwise.
 The union of the irreducible components corresponding to the vertices of a dead or simple branch of $\A_S$ is called a \emph{dead} or \emph{simple branch of $\mc D_S$.}
 Similarly, a \emph{chain of $\mc D_S$} is the union of the irreducible components corresponding to the vertices of valence $2$ of a chain of $\A_S$  or the unique intersection point $D\cap D'$ of a chain reduced to $\stackrel{D}{\bullet}-\stackrel{D'}{\bullet}$.
 
 Now, for  each chain $\mc C$ in $\A_S$ let us choose an edge $\msf e_{\mc C}$. It corresponds to the intersection point of two irreducible components of $\mc D_S$. Let us also choose local holomorphic coordinates  $(u_{\mc C},v_{\mc C})$ at this point    such that $u_{\mc C}v_{\mc C}=0$ is a local equation of $\mc D_S$. 
The real analytic hypersurface $\mb \{|u_{\mc C}|=1\}$  intersects
transversally  the  ``Milnor tube'' $E_S^{-1}( \delta\mb T_{\varepsilon,\eta})$  along a 2-torus $T_{\mc C}$,  with  $0<\eta\ll \varepsilon\ll 1$. After identifying $E_S^{-1}( \delta\mb T_{\varepsilon,\eta})$ with  $\delta\mb T_{\varepsilon,\eta}$  via $E_S$,  we  set 
\[ 
\mathscr T =\{T_{\mc C}\}_{\mc C\in \mathrm{chain}(\A_S)}
 \]
where $\mathrm{chain}(\A_S)$ is the set of all chains in $\A_S$. 
Similarly, the collection of real hypersurfaces 
$(\{|u_{\mc C}|=1\})_{\mc C\in\mathrm{chain}(\A_S)}$ induces  a \emph{geometric block decomposition}
of $E_S^{-1}(\mb T^\ast_{\varepsilon,\eta})\simeq \mb T^\ast_{\varepsilon,\eta}$, 
each \emph{geometric block} being a  real four-dimensional submanifold with boundary and corners  and we have:
\begin{enumerate}[(B1)]\it
\item\label{BB1} each block  is incompressible in  $\mb T^\ast_{\varepsilon,\eta}$,
\item any intersection of two blocks is a   boundary component of each block,  is incompressible in each block and is homeomorphic to a coreless full torus 
$\mb D^*\times \mb S^1$,
\item\label{BB2} the closure in $E_S^{-1}(\mb T_{\varepsilon,\eta})$ of each block $\mc B$ contains a unique irreducible component $D$ of $\mc D_S$ with  at least three singular points of $\mc D_S$ and we denote $\mc B$ by $\mc B_D$;  conversely for any vertex of $\A_S$ with valence at least three, there is a block $\mc B$ such that $\mc B=\mc B_D$,
\item\label{BB4}  given a block $\mc B_D$, over the open set $D^\star\subset D$ such that  $D\setminus D^\star$ is a union of    disjoint   discs centered at  singular points of $\mc D_S$ in $D$,  
the Seifert fibration of $E_S^{-1}(\delta\mb T_{\varepsilon,\eta})\cap \mc B$  coincides with the Hopf fibration of $D$; moreover the  exceptional  fibers consist in one Hopf fiber for each irreducible component of $\mc D_S$ that corresponds to an end vertex of a  dead branch of $\A_S$ with attaching vertex $D$.
\end{enumerate}

\subsection{Separatrices} Now, and in all the sequel we fix a germ of singular foliation $\F$ at the origin of $\C^2$ defined by a germ of holomorphic differential form $\omega$ with $0$ as isolated singular point.

A \emph{separatrix}  of $\F$ is a germ of  \emph{invariant} irreducible holomorphic curve, i.e. any parametrization of it  $\gamma:(\C,0)\to \C^2$  satisfies $\gamma^\ast\omega\equiv 0$.  
Classically,
a germ of a curve $C$ is a separatrix if and only if
there is an  open neighborhood $U$  of $0$  and a  closed leaf $L$ of the regular foliation defined by $\omega$ on $U\setminus \{0\}$, such that 
$C\cap (U\setminus \{0\})=L$.  According to Camacho-Sad existence  theorem \cite{CS}, the number of separatrices  of a foliation is strictly positive. It  may be infinite, but the number of their topological classes is finite by Seidenberg's reduction theorem \cite{Seidenberg}. This one states the existence of   a unique  proper map called \emph{reduction map of $\F$}
\begin{equation*}\label{reduction}
E_\F:M_\F\to \C^2
\end{equation*}
obtained by a minimal sequence of blowing-up  maps  with  finite centers, such that the germ of the foliation $\F^\sharp:=E_\F^{-1}\F$ at any point  of the \emph{exceptional divisor} 
\begin{equation*}
\mc E_\F:=E_\F^{-1}(0)
\end{equation*} is either regular or
defined in suitable  local coordinates $(u, v)$ by  a differential form 
\begin{equation}\label{reducedform}
( \lambda v+\cdots)du-( u+\cdots)dv\quad  \hbox{with}   \quad\lambda\in \C\setminus \mathbb Q_{>0}\,, 
\end{equation}
moreover, at any point of a \emph{dicritical component} (i.e. non $\F^\sharp$-invariant) $D$ of $\mc E_\F$, the foliation $\F^\sharp$ is regular and transversal to $D$. 

We recall that a foliation germ $\F$ is called \emph{generalized curve} if there is no \emph{saddle-node} singularity of $\F^\sharp$, i.e. given by a differential form (\ref{reducedform}) with $\lambda=0$.  For this type of foliations $E_\F$ coincides with the minimal composition of blowing-ups that desingularizes all separatrices of $\F$, see~\cite{CLNS}.

A separatrix $C$ will be called \emph{dicritical separatrix} if its \emph{strict transform} $\overline{E_\F^{-1}(C\setminus\{0\})}$ through  $E_\F$ meets a dicritical component of $\mc E_\F$, otherwise $C$ is called \emph{isolated separatrix}. The \emph{isolated separatrices curve} is the germ of  curve obtained as the union of all
 the isolated separatrices. We denote it, resp. its total transform by the reduction map of~$\F$, by
\begin{equation}\label{compsepcurve}
S_\F\subset (\C^2,0),\quad \hbox{resp.}\quad S_\F^\sharp:=E_\F^{-1}(S_\F)\,.
\end{equation}

These notions can be topologically characterized. Indeed, 
a separatrix is dicritical if and only if it is contained in a non-constant equireducible family of invariant analytic curve germs. On the other hand, classically for  families of germs of curves in $(\C^2,0)$,  equireducibility and topological triviality are equivalent notions. 

The topological information contained in the dicritical and isolated separatrices may be organized into a  global combinatorial object which will also be a topological invariant of the foliation.

\begin{defin}\label{defAF}
The \emph{dual graph $\A_\F$ of $\F$} is the dual graph of the normal crossings curve~$S_\F^\sharp$. 
\end{defin}

Notice that the desingularization map $E_{S_\F}$ of $S_\F$ can be different from the reduction of singularities of $\F$ and in this case $\A_\F$ is not the dual graph $ \A_{S_\F}$ of $E_{S_\F}^{-1}(S_\F)$, see Figure~\ref{EFES}.

\begin{prop}\label{InducedGraphMorphism}
Any homeomorphism germ $\phi:(\C^2,0)\to(\C^2,0)$ that conjugates two generalized curves  $\F$ and $\G$,  $\phi(\F)=\G$, defines an isomorphism
\[ \A_\phi:\A_\F\to\A_\G\]
 between the dual graphs of $\F$ and $\G$,
such that if $D$ is a vertex corresponding to the strict transform of an irreducible component $C$ of $S_\F$, then $\A_\phi(D)$ is the vertex of $\A_\G$ corresponding to the strict transform of $\phi(C)$.
 This isomorphism sends dicritical components into dicritical components and is compatible with the intersection forms, i.e. we have 
 \begin{equation*}\label{Aphi-dot}
 \A_\phi(D)\cdot\A_\phi(D')=D\cdot D'
  \end{equation*}
  for any vertices $D$, $D'$ of  $\A_\F$  
 considered as irreducible components of $S_\F^\sharp$. Moreover the correspondence $(\F,\phi)\mapsto (\A_\F,\A_\phi)$ is a functor from the category of generalized curves  and topological conjugacies, to the category of graphs and isomorphisms of graphs. 
\end{prop}
\noindent Let us give an idea of the proof. 
A consequence of the results in  \cite{CLNS} is that when $\F$ is a generalized curve, the  reduction map  $E_\F$  coincides with the desingularization map  of a suitable  curve germ $C_\F$, obtained by adding to $S_\F$  a pair of non-isolated separatrices for each dicritical component. Moreover 
the reduction map of $\G=\phi(\F)$  also coincides with the desingularization map of $\phi(C_\F)$.  
A classical Zariski theorem \cite{Z} associates to $\phi$ 
a morphism \[\phi_\bullet:\A_{C_\F}\to\A_{\phi(C_\F)}\] between the dual trees of the total transforms of $C_\F$ and  $\phi(C_\F)$ by the reduction map $E_{C_\F}=E_\F$. 
More specifically $\phi_\bullet$ can be obtained by applying the following theorem to $X=C_\F$, $Y=\phi(C_\F)$ and $\varphi=\phi$:
\begin{teo} \cite[Theorem A]{MM2}\label{curves}
 Let $\varphi:(\C^2,0)\to(\C^2,0)$ be a germ of homeomorphism that sends a germ of curve $X$ onto a germ of curve $Y$. Then there exists a germ of homeomorphism $\psi : (\C^2,0)\to(\C^2,0)$ and a germ of homeomorphism along the exceptional divisors of the desingularizations of $X$ and $Y$                        
\[ 
\Psi:(M_X,\mc E_X)\to(M_Y,\mc E_Y)
 \]
such that
\begin{enumerate}
\item $\Psi$ lifts $\psi$, i.e. $E_Y\circ \Psi=\psi\circ E_X$;
\item for any irreducible component $C$ of $X$ we have $\varphi(C)=\psi(C)$;
\item  for small enough Milnor balls $B$ for $X$ and $B'\supset \psi(B)\cup \varphi(B)$ for $Y$, the group morphisms $\pi_1(B\setminus X)\to \pi_1(B'\setminus Y)$ induced by $\varphi$ and $\psi$ are equal up to composition by an inner automorphism. 
\end{enumerate}
\end{teo}

\noindent  The homeomorphism germ $\Psi$ induces a graph morphism $\A_{C_\F}\to\A_{\phi(C_\F)}$ between the dual trees of the total transforms of $C_\F$ and $\phi(C_\F)$ by their desingularization maps, sending the vertex corresponding to an irreducible component $D$ into the vertex corresponding to the irreducible component $\Psi(D)$ and satisfying $\Psi(D)\cdot\Psi(D')=D\cdot D'$ for any pair of vertices $D,D'$ of $\A_{C_\F}$. In fact, it results from the proof of this theorem that this graph-morphism  does not depend on the choice of  $\Psi$, hence we can denote it by $\phi_\bullet$.  
Since $E_\F=E_{C_\F}$, we have the inclusion 
\[S_\F^\sharp\subset E_{C_\F}^{-1}(C_\F)\] and $\A_\F$ is a subgraph of $\A_{C_\F}$, similarly $\A_\G$ is a subgraph of $\A_{\phi(C_\F)}$. These subgraphs are obtained by deleting the vertices corresponding to dicritical separatrices.
Using the fact that $\phi$ sends dicritical separatrices of $\F$  to dicritical separatrices of $\G$,  we see that $\phi_\bullet$ sends  $\A_\F$ onto  $\A_\G$,  inducing the  isomorphism $\A_\phi$  between these graphs stated in Proposition~\ref{InducedGraphMorphism}. \\

As we will see in \S\ref{Exdic}, there are dicritical foliations $\F$  for which  the incompressibility property  of the leaves in the complementary of  $S_\F$ in any neighborhood, does not hold. In other words, the fundamental group $\Gamma$ of the complementary of $S_\F$ in a Milnor ball may be too small to contain the fundamental group of any leaf. For this reason  we will increase $\Gamma$ by adding to $S_\F$ some dicritical separatrices.

\begin{defin}\label{Fappropcurve}
A germ of an invariant curve  containing all the isolated separatrices of $\F$ and whose strict transform by the reduction of $\F$ meets any  dicritical component $D$ with  $\mr{card}(D\cap \mr{Sing}(\mc E_\mc  \F))=1$, will  be called \emph{$\F$-appropriate curve germ}.  
\end{defin}

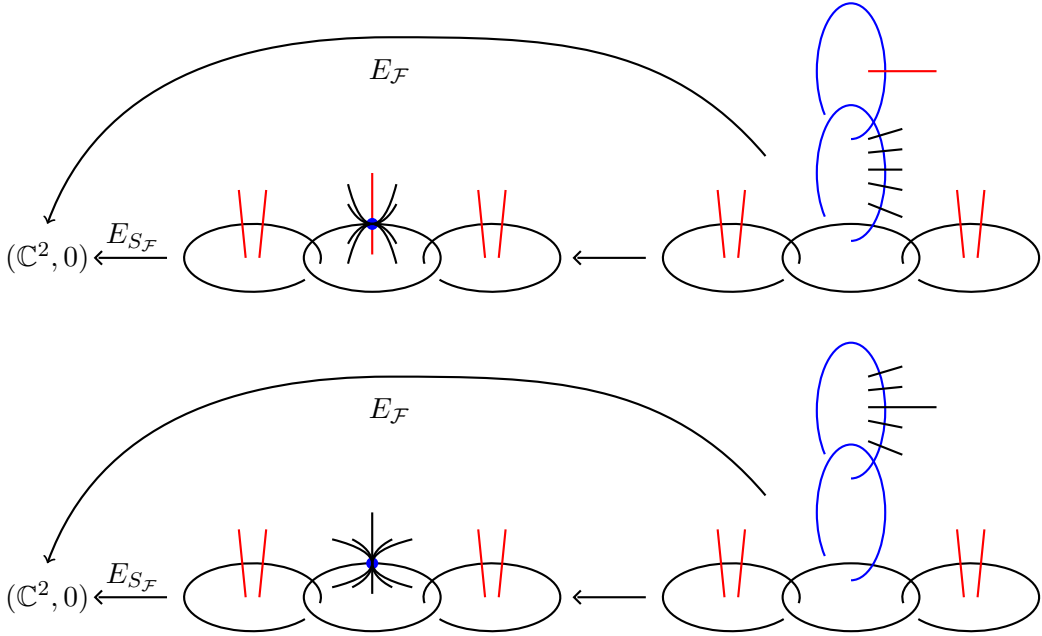
\begin{figure}[ht]
\begin{tikzpicture}[thick,scale=0.9]
\node at (-5,4) {$(\C^2,0)$};
\draw (-4.25,4) to (-3.25,4);
\draw (-4.2,4.1) to (-4.3,4) to (-4.2,3.9);
\node at (-3.75,4.3) {$E_{S_\F}$};
\draw (2.75,4) to (3.75,4);
\draw (2.8,4.1) to (2.7,4) to (2.8, 3.9);
\begin{scope}[shift={(-1,0)}]
\draw (-1,4) [elliparc=-10:320:1cm and .5cm]; 
\draw (.75,4) ellipse (1cm and .5cm);
\draw (2.5,4) [elliparc=220:550:1cm and .5cm]; 
\draw[red] (.75,4.05) to (.75,5.25);
\fill [blue] (.75,4.5) circle [radius=2.5pt];
\draw (.75,5.5) [elliparc=-45:-135:.5cm and 1cm];
\draw (.75,6.5) [elliparc=-45:-135:.5cm and 2cm];
\draw (.75,3.5) [elliparc=45:135:.5cm and 1cm];
\draw (.75,2.5) [elliparc=45:135:.5cm and 2cm];
\draw[red] (-.9,4) to (-.8,5);
\draw[red] (-1.1,4) to (-1.2,5);
\draw[red] (2.6,4) to (2.7,5);
\draw[red] (2.4,4) to (2.3,5);
\end{scope}
\begin{scope}[shift={(6,0)}]
\draw (-1,4) [elliparc=-10:320:1cm and .5cm]; 
\draw (.75,4) ellipse (1cm and .5cm);
\draw (2.5,4) [elliparc=220:550:1cm and .5cm]; 
\draw[red] (-.9,4) to (-.8,5);
\draw[red] (-1.1,4) to (-1.2,5);
\draw[red] (2.6,4) to (2.7,5);
\draw[red] (2.4,4) to (2.3,5);
\draw[blue]  (.75,5.25) [elliparc=-90:220:0.5cm and 1cm];
\draw[blue]  (.75,6.75) [elliparc=-90:220:0.5cm and 1cm];
\draw[red] (1,6.75) to (2,6.75);
\draw (1,5.75) to (1.5,5.9);
\draw (1,5.55) to (1.5,5.6);
\draw (1,5.3) to (1.5,5.3);
\draw (1,5.1) to (1.5,5.);
\draw (1,4.8) to (1.5,4.6);
\end{scope}
\draw [->] (5.5,5.5) to [out=130,in=0] (0,7.25) to [out=-180, in=70] (-5,4.5);
\node at (0,6.75) {$E_\F$};
\begin{scope}[shift={(0,-5)}]
\node at (-5,4) {$(\C^2,0)$};
\draw (-4.25,4) to (-3.25,4);
\draw (-4.2,4.1) to (-4.3,4) to (-4.2,3.9);
\node at (-3.75,4.3) {$E_{S_\F}$};
\draw (2.75,4) to (3.75,4);
\draw (2.8,4.1) to (2.7,4) to (2.8, 3.9);
\begin{scope}[shift={(-1,0)}]
\draw (-1,4) [elliparc=-10:320:1cm and .5cm]; 
\draw (.75,4) ellipse (1cm and .5cm);
\draw (2.5,4) [elliparc=220:550:1cm and .5cm]; 
\draw (.75,4.05) to (.75,5.25);
\fill [blue] (.75,4.5) circle [radius=2.5pt];
\draw (1.75,4.5) [elliparc=135:225:1cm and 0.5cm];
\draw (-.25,4.5) [elliparc=-45:45:1cm and 0.5cm];
\draw (-1.25,4.5) [elliparc=-45:45:2cm and 0.5cm]; 
\draw (2.75,4.5) [elliparc=135:225:2cm and 0.5cm]; 
\draw[red] (-.9,4) to (-.8,5);
\draw[red] (-1.1,4) to (-1.2,5);
\draw[red] (2.6,4) to (2.7,5);
\draw[red] (2.4,4) to (2.3,5);
\end{scope}
\begin{scope}[shift={(6,0)}]
\draw (-1,4) [elliparc=-10:320:1cm and .5cm]; 
\draw (.75,4) ellipse (1cm and .5cm);
\draw (2.5,4) [elliparc=220:550:1cm and .5cm]; 
\draw[red] (-.9,4) to (-.8,5);
\draw[red] (-1.1,4) to (-1.2,5);
\draw[red] (2.6,4) to (2.7,5);
\draw[red] (2.4,4) to (2.3,5);
\draw[blue]  (.75,5.25) [elliparc=-90:220:0.5cm and 1cm];
\draw[blue]  (.75,6.75) [elliparc=-90:220:0.5cm and 1cm];
\begin{scope}[shift={(0,1.5)}]
\draw (1,5.75) to (1.5,5.9);
\draw (1,5.55) to (1.5,5.6);
\draw (1,5.3) to (2,5.3);
\draw (1,5.1) to (1.5,5.);
\draw (1,4.8) to (1.5,4.6);
\end{scope}
\end{scope}
\draw [->] (5.5,5.5) to [out=130,in=0] (0,7.25) to [out=-180, in=70] (-5,4.5);
\node at (0,6.75) {$E_\F$};
\end{scope}
\end{tikzpicture}

\caption{In both cases $E_{S_\F}\neq E_\F$ due to a non-reduced singularity in the central divisor of $E_{S_\F}$ (in blue). In the first one $S_\F$ (in red) is $\F$-appropriate and in the second one is not.}\label{EFES}
\end{figure}

\begin{prop}\label{conjappr}
Any homeomorphism germ $\phi:(\C^2,0)\to(\C^2,0)$ that conjugates two generalized curves  $\F$ and $\G$, $\phi(\F)=\G$,  necessarily transforms any dicritical, resp. isolated separatrix of $\F$, resp. $\F$-appropriate curve germ, 
into  a dicritical, resp. isolated separatrix of $\G$, resp. $\G$-appropriate curve germ. 
\end{prop}
\begin{proof}
Using Proposition~\ref{InducedGraphMorphism} it suffices to see that
the image by $\phi$ of a $\F$-appropriate curve $S$ is a $\G$-appropriate curve. For this
we may   perform  the previous analysis with $X=S$ and $Y=\phi(S)$. We similarly obtain a graph-isomorphism  $\A_{S^\sharp}\to\A_{\phi(S)^\sharp}$ from the dual graph of  $S^\sharp:=E_\F^{-1}(S)$ to that of $\phi(S)^\sharp:=E_\G^{-1}(\phi(S))$, that  sends dicritical components into dicritical components and sends $\A_\F\subset \A_{S^\sharp}$ into  $\A_\G\subset \A_{\phi(S)^\sharp}$. We easily deduce that $\phi(S)$ is  $\G$-appropriate.  
\end{proof}

\subsection{Separators and dynamical decomposition} \label{subsecSeparator}
One says that the germ  of the reduced foliation  $\F^\sharp$ at a singular point $p$ in the exceptional divisor $\mc E_\F$  is  a \emph{node} when  it is defined in suitable  local coordinates $(u, v)$ by  a differential form (\ref{reducedform}) with $\lambda\in\R_{>0}\setminus\mathbb Q$.
Then there are exactly two invariant curve germs at this point; these are regular, with normal crossings and they contain the germ of the exceptional divisor. 

Other invariant sets can be constructed using the existence of \emph{local linearizing coordinates} $(u,v)$  in which $\F^\sharp$ is given by  the differential form  $\lambda vdu-udv$. 
For any small enough ball $B:=\{|u|^2+|v|^2 <r\}$, any leaf of the foliation $\F^\sharp$ restricted to $B\setminus\{uv=0\}$ is contained and dense in  a  real invariant  hypersurface defined by an equation $ |v|=c |u|^\lambda$,  $c\in\R^\ast$. 
More generally   any  subset %
\begin{equation*}\label{nodalseparator_mu}
S_{c_1,c_2}:=\{c_1 |u|^\lambda\leq |v|\leq c_2 |u|^\lambda\}\subset B\quad\text{with}\quad 0<c_1\le c_2<\infty\,, 
\end{equation*}
is \emph{invariant}, i.e. it is a  union of leaves, and it divides $B$ in two invariant connected components. 
For this reason we call \emph{nodal separator of $\F^\sharp$ at $p$},
any  intersection of $S_{c_1,c_2}$ 
with a neighborhood of  $\mc E_\F$.

Another type of  invariant set can be constructed using a  dicritical irreducible component $D$ of $\mc E_\F$. For this we first choose a small enough compact tubular neighborhood $T$ of $D$ so that  $\F^\sharp$ restricted to $T$ is a locally trivial fibration. Then 
we choose open conformal discs $D_s\subset D$ centered at the singular points $s\in D$  of the exceptional divisor $\mc E_\F$,
such that $\overline{D_s}\cap \overline{D_{s'}}=\emptyset$ for $s\neq s'$. The  union of the leaves of 
the foliation $\F^\sharp$  restricted to $T$ 
that meet no disc $D_s$ is a $\F^\sharp$-invariant  set  which we call \emph{dicritical separator of $\F^\sharp$ at $D$}.
\\

The intersection of a neighborhood of $0\in\C^2$ with the image by $E_\F$ of a nodal or  dicritical separator of $\F^\sharp$ will be called \emph{nodal} or \emph{dicritical separator of $\F$}.
A \emph{nodal separatrix of $\F$} is a separatrix whose strict transform by $E_\F$ passes through a nodal singularity of $\F^\sharp$.
C. Camacho and R. Rosas prove that every invariant set contains a separatrix or a nodal separator, see \cite[Theorem~1]{CamachoRosas}.

\begin{teo}[{\cite[Theorem~1]{Rosas} and \cite[Theorem~1.3]{RosasFourier}}]\label{C0InvSep} The image of a nodal separatrix, resp. nodal separator germ, resp. dicritical separator germ of $\F$ by a germ of homeomorphism that conjugates $\F$ to a germ of foliation $\G$, is a nodal separatrix, resp. nodal separator germ, resp. dicritical separator germ of $\G$.
\end{teo}

Let us consider now a collection of separators of $\F$
\begin{equation*}
\mc S =(\mc S_\alpha)_{\alpha\in \mc A}\,, \qquad \mc A=\mc E_\F^{\mr{dic}}\cup \mc N,
\end{equation*}
where $\mc N$ is the collection of the nodal singular points of $\F^\sharp$, $\mc E_\F^{\mr{dic}}$ is the collection of  the dicritical components of $\mc E_\F$ 
and for each $\alpha\in\mc A$, $\mc S_\alpha$ is the image by $E_\F$, in a closed ball, of a chosen closed nodal or dicritical separator of $\F^\sharp$ at $\alpha$. We will say that $\mc S$ is a \emph{complete system of separators for $\F$}. In any small enough closed  ball $B_r$ of  radius $r\leq R$, every $\mc S_\alpha$ is a cone, as foliated manifold,  over its intersection with the sphere $\partial B_r$. This results from the transversality of the leaves contained in $\mc S_\alpha$ to any sphere $\partial B_r$, $0<r\leq R$. We will say that $B_R$ is a \emph{Milnor ball for $\mc S$}.

 \begin{defin}\label{dyncompdef} Let $B$ be a Milnor ball for a complete system $\mc S$ of  separators for $\F$.
 We call  $\F$-\emph{dynamical component  of $B$} defined by $\mc S$,  any connected component of $B\setminus \cup_{\alpha\in\mc A}\mc S_\alpha$. 
 We call \emph{dynamical component of   $S_\F^\sharp$}, see (\ref{compsepcurve}), the closure in $S_\F^\sharp$ of any connected component of $S_\F^\sharp\setminus (\mc E_\F^{\mr{dic}}\cup \mc N)$, see Figure~\ref{decomposition-dynamique}.
 The dual graph of a dynamical component of   $S_\F^\sharp$ is called  \emph{dynamical component of  $\A_\F$}, see Definition \ref{defAF}.
\end{defin}

\begin{figure}[ht]
\begin{center}
\begin{tikzpicture}[scale=1.1]
\draw[thick,red] (1,2) [elliparc=-20:320:1cm and 0.5cm];
\draw[red,thick] (2.75,2)  [elliparc=-20:320:1cm and 0.5cm]; 
\node at (2.75,2) {$\color{red}D_{11}$};
\draw[blue,thick] (4.5,2)  ellipse (1cm and .5cm);
\node at (4.5,2) {$\color{blue}D_{12}$};
\draw[blue,thick] (4.5,.75) [elliparc=60:400:0.5cm and 1cm]; 
\draw[red,thick] (2.75,.75) [elliparc=60:400:0.5cm and 1cm]; 
\draw[red,thick] (2.75,3.25) [elliparc=-100:-10:0.5cm and 1cm]; 
\draw[red,thick] (2.75,3.25) [elliparc=10:220:0.5cm and 1cm]; 
\draw[green,thick] (1.25,3.5) ellipse (1cm and .5cm);
\node at (1.25,3.5) {$\color{green}D_2$};
\fill [black] (2.25,3.5) circle [radius=2pt];
\node at (2.25,4.5) {node};
\draw[->] (2.25,4.25) to (2.25,3.9);
\draw[thick,dashed] (4,3.5) ellipse (1cm and .5cm);
\node at (4,3.5) {dicritical};
\draw[thick,marron] (5.75,3.5) [elliparc=220:550:1cm and .5cm]; 
\node at (5.75,3.5) {$\color{marron}D_3$};
\draw[thick,marron] (7.5,3.5) [elliparc=220:550:1cm and .5cm]; 
\draw[green,thick] (-0.5,3.5) [elliparc=40:380:1cm and .5cm];
\draw [thick,->,green] (1.3,3.75) to (1.3,4.75);
\draw [thick,->,blue] (5,2) to (6,2);
\draw [thick,->,marron] (5.75,3.75) to (5.75,4.75);
\end{tikzpicture}\end{center}
\caption{In this schematic representation of the divisor $S_\F^\sharp$ we have three dynamical components: the green and black ones and on the other hand the union of the red and blue components. Each color correspond to a geometric block (in the sense of \S\ref{decompCompCurve}) containing a divisor $D_i$ of valence $\geq 3$. The non-compact components of $S_\F^\sharp$ are denoted by arrows.}\label{decomposition-dynamique}
\end{figure}
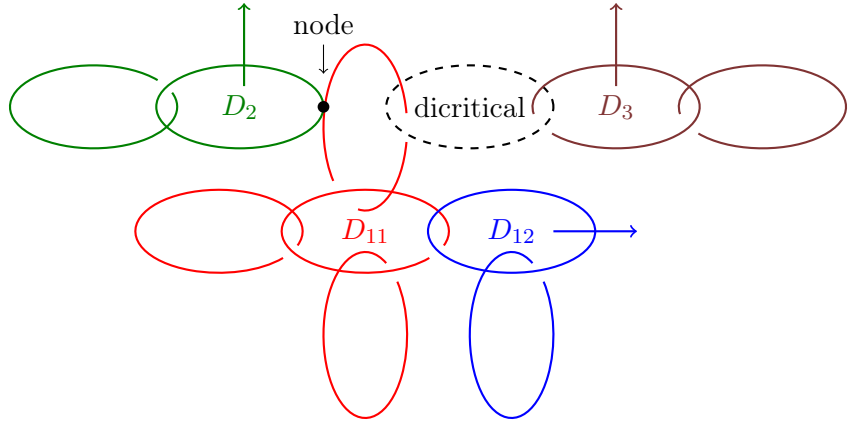

By construction if $\mc D$ is a $\F$-dynamical component of  $B$, then the intersection
\begin{equation*}\label{ED}
\mc E_{\mc D}=\overline{E_\F^{-1}(\mc D)}\cap\overline{S_\F^\sharp\setminus\mc E_\F^\mr{dic}}
\end{equation*}  
of is  a dynamical component  $S_\F^\sharp$. 
 The \emph{dual graph $\A_\mc D$ of $\mc E_D$}  can also be  characterized  as a connected component of the graph obtained by removing from the dual graph $\A_\F$ of $S_\F^{\sharp}$:
\begin{enumerate}[(a)]
\item  the edges $\langle D,D'\rangle$ such that $D\cap D'$ is a nodal singularity of  $\F^\sharp$,  
\item the vertices $D$ corresponding to a dicritical component of $\mc E_\F$ and all the edges containing these vertices.
\end{enumerate}

Notice that a dynamical component $\mc E_{\mc D}$ of $S^\sharp_\F$ can be reduced to a single irreducible component of $S^\sharp_\F$: we have $\mc E_{\mc D}=C$ when   $C$  is a non-compact irreducible component of $S_\F^\sharp$ that intersects $\mc E_\F$ at a nodal singular point of $\F^\sharp$, in other words, when   $C$  is the strict transform by $E_\F$ of a nodal separatrix.

The remarkable property of dynamical components is  the following  refinement of the Camacho-Sad existence theorem, given by L. Ortiz-Bobadilla, E. Rosales-Gonz\'alez and S.M. Voronin \cite[Strong Camacho-Sad Separatrix Theorem]{OBRGV} and completed in \cite[Lem\-ma~1.9]{MM4} for dicritical foliations.
\begin{teo}\label{sepdyncomp}  Any dynamical component of the total transform $S^\sharp_\F$ of the separatrices curve $S_\F$ either contains a non-nodal isolated separatrix, or is reduced to a single nodal separatrix.
\end{teo}

We will highlight now that dynamical components of $S^\sharp_\F$ can be characterized as traces on $S^\sharp_\F$ of ``minimal'' invariant sets.
 
The foliation $\F$ being defined on a ball  $B$ let   $A\subset A'$  be two subsets of $B$, resp. of $E_\F^{-1}(B)$.  We will denote by $\mathrm{sat}(A,A')$  the union of the connected components of $L\cap A'$ for every leaf $L$ of $\F$, resp. of $\F^\sharp$, meeting $A$.
We also denote   by $\overline{\mathrm{sat}}(A,A')$ the closure of ${\mathrm{sat}}(A,A')$ in $B$, resp. in  $E_\F^{-1}(B)$. For any subset $A$ in $\C^2$ let us write 
\[A^\sharp:=E_\F^{-1}(A).\] 
Following the ideas developed in \cite{CamachoRosas} one can prove:
\begin{prop}\label{proprtranscurve}
Let $B$ be a Milnor ball for $S_\F$ and let $(W_i)_i$ be a fundamental system of neighborhoods of $S_\F$ in $B$.
Let $\Delta$ be an embedded disc in $B\setminus\{0\}$ transversal to $\F$, meeting the isolated separatrices $S_\F$ at a single point $p$, and define $\Delta_i=\Delta\cap W_i$. Then
\begin{enumerate}
\item
$\bigcap_i \mr{sat}(\Delta_i^\sharp,B^\sharp)$
is the dynamical component of $S_\F^\sharp$ meeting $\Delta^\sharp$;
\item there exist two complete systems of separators
$\mc S=(\mc S_\alpha)_{\alpha\in \mc A}$ in $B$
and $\mc S'=(\mc S'_\alpha)_{\alpha\in \mc A}$  in a ball $B'\subset B$ such that
\begin{enumerate}[-]
\item $B$ is a Milnor ball for $\mc S$, $B'$ is a Milnor ball for $\mc S'$ and
 $\mc S_\alpha\cap B'\subset \mc S'_\alpha$, 
\item if $W_i$ is small enough then $\mr{sat}(\Delta_i,B)$ is contained in one $\F$-dynamical component of $B$ defined by $\mc S$ and contains one $\F$-dynamical component of $B'$ defined by~$\mc S'$.
\end{enumerate}
\end{enumerate}
\end{prop}

\begin{center}
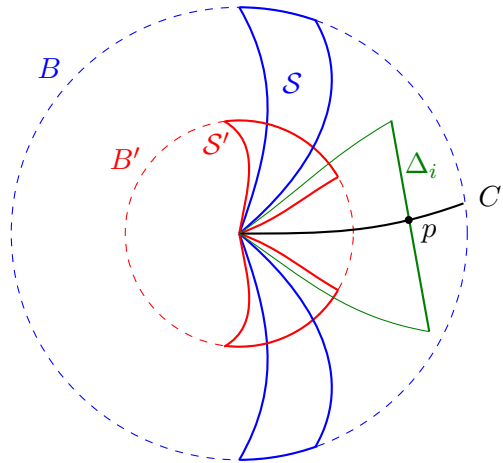
\begin{figure}[ht]
\begin{tikzpicture}
\draw[dashed,blue] (5,2) circle [radius=3];
\draw[dashed,red] (5,2) circle [radius=1.5];
\draw[thick,blue] (5,2) [out=70,in=300] to (5,5);
\draw[thick,blue] (5,2) [out=-70,in=-300] to (5,-1);
\draw [blue,thick] (5,5) arc [radius=3,start angle=90, end angle=70];
\draw [blue,thick] (5,-1) arc [radius=3,start angle=-90, end angle=-70];
\node at (7.5,2) {$p$};
\draw[red,thick](6.3,2.75) arc [radius=1.5,start angle=30, end angle=97];
\draw[red,thick](6.3,1.25) arc [radius=1.5,start angle=-30, end angle=-97];
\draw[thick,blue] (5,2) [out=40,in=300] to (6,4.83);
\draw[thick,blue] (5,2) [out=-40,in=-300] to (6,-.83);
\draw[red,thick] (5,2) [out=20,in=210] to (6.3,2.75);
\draw[red,thick] (5,2) [out=-20,in=-210] to (6.3,1.25);
\draw[red,thick] (5,2) [out=80,in=-30] to (4.8,3.49);
\draw[red,thick] (5,2) [out=-80,in=30] to (4.8,0.5);
\draw [black,thick] (5,2) [out=0,in=200] to (7.95,2.4);
\draw [green,thick] (7,3.5) to (7.5,.7);
\draw [green] (5,2) [out=35,in=210] to (7,3.49);
\draw [green] (5,2) [out=-35,in=-190] to  (7.5,.7);
\node[red] at (3.5,3) {$B'$};
\node[blue] at (2.5,4.2) {$B$};
\node[green] at (7.4,2.9) {$\Delta_i$}; 
\node[blue] at (5.7,4) {$\mc S$};
\node[red] at (4.7,3.2) {$\mc S'$};
\fill  (7.23,2.18) circle [radius=1.5pt];
\node at (8.3,2.5) {$C$};
\end{tikzpicture}
\caption{Illustration of Proposition~\ref{proprtranscurve}, $C$ is an irreducible component of $S_\F$.}
\end{figure}
\end{center}

\begin{obs}\label{inf-sep}
Notice that as a consequence of Proposition~\ref{proprtranscurve},
when $\F^\sharp$ has no nodal singularities and 
$\mc E_\F$ contains dicritical components  then
$
U=E_\F\big(\,\mr{sat}(\mc E_\F^{\mr{dic}},B^\sharp)\cup S_\F^\sharp\,\big)$
is a neighborhood of $0$ in $\C^2$. Thus, for every leaf $L\subset U$ of $\F$ there is a separatrix contained (as germ of curve) in $L\cup\{0\}$. 
If in addition $\mc E_\F$ contains an invariant irreducible component with infinite holonomy group then there are leaves $L\subset U$ of $\F$ containing  infinitely many separatrices, see Figure~\ref{fig-inf-sep}.
\end{obs}

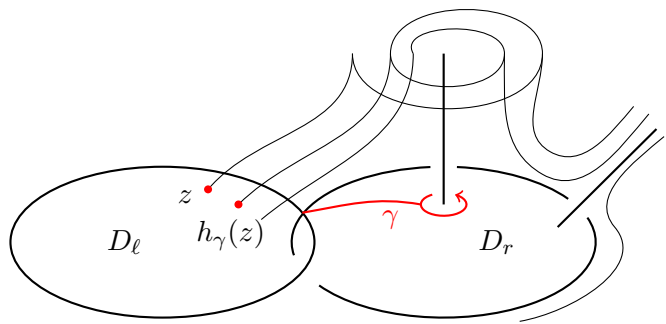
\begin{figure}[ht]
\begin{tikzpicture}
\draw[thick] (2.5,1.5) [elliparc=0:360:2cm and 1cm]; 
\node at (2,1.5) {$D_\ell$};
\draw[thick] (6.2,1.5) [elliparc=215:379:2cm and 1cm];
\draw[thick] (6.2,1.5) [elliparc=35:85:2cm and 1cm];
\draw[thick] (6.2,1.5) [elliparc=95:190:2cm and 1cm];
\node at (6.9,1.5) {$D_r$}; 
\draw[thick] (6.2,2) to (6.2,4);
\draw[thick] (7.7,1.7) to (9,3);
\draw (5,4) [out=270,in=-90] to  (7.5,4)  [out =90,in=90]  to (5.5,4) [out=-90,in=-90] to (7,4) [out=90,in=120] to (5.8,4);
\draw (5.8,4) [out=-90,in=40] to (3.8,1.8);
\draw (5.5,4) [out=-90,in=45] to (3.5,2);
\draw (5,4) [out=-90,in=50] to (3.1,2.2);
\draw [red,thick] (4.34,1.89) [out=10,in=170] to (5.9,2);
\draw [red,thick,->] (6.2,2) [elliparc=120:420:.3cm and .15cm];
\node at (5.5,1.8) {\color{red} $\gamma$};
\fill [red] 
(3.1,2.2) circle [radius=1.5pt];
\node at (2.8,2.1) {$z$};
\fill [red] 
 (3.5,2) circle [radius=1.5pt];
\node at (3.4,1.6) {$h_\gamma(z)$};
\draw (7,4) [out=-90,in=150] to (7.5,2.4) [out=-30,in=230] to (8.9,3.2);
\draw (7.5,4) [out=-90,in=150] to (7.6,2.7) [out=-30,in=230] to (8.7,3.3);
\draw (9.1,2.9) [out=-140,in= 130] to (8.45,1.7) [out=-50,in=10] to (7.2,.45);
\end{tikzpicture}
\caption{The left irreducible component $D_\ell$ of $\mc E_\F$ is dicritical. The right irreducible component $D_r$ is invariant and its projective holonomy group is infinite. The leaf of $\F^\sharp$ passing through the point $z\in D_\ell$ meets $D_\ell$ infinitely many times and consequently its image by $E_\F$ is a leaf of $\F$ containing infinitely many separatrices as germs of curves through $0$.}\label{fig-inf-sep}
\end{figure}

Assertion (1) in Proposition~\ref{proprtranscurve} does not imply that the decomposition of $\A_\F$ into dynamical components is a topological invariant of $\F$. 
\begin{question}\label{question}
If $\phi:(\C^2,0)\to(\C^2,0)$ is a homeomorphism germ conjugating a generalized curve $\F$ to another generalized curve $\G$, does the graph morphism $\A_\phi$ defined in Proposition~\ref{InducedGraphMorphism} send any dynamical component of $\A_\F$ into a dynamical component of $\A_\G$?
\end{question}
It can be easily deduced from assertion (1) in Proposition~\ref{proprtranscurve} that the answer is  positive for dynamical components without nodal singularities:

\begin{cor}\label{compdynsansnoeud}
Let $\mc D$ be the closure of a connected component of $\mc E_\F\setminus\mc E_\F^{\mr{dic}}$ and let $\A_{\mc D}$ be the dual graph of $\mc D$. 
If $\phi:\F\to\G$ is a topological conjugacy between generalized curves then $\A_\phi(\A_{\mc D})$ is the dual graph of the closure of a connected component $\mc D'$ of $\mc E_\G\setminus\mc E_\G^{\mr{dic}}$. Moreover if $\mc D$ does not contain nodal singular points of $\F^\sharp$ then $\mc D'$ does not contain nodal singular points of $\G^\sharp$.
\end{cor}

Moreover, as we will see later, under some Krull-generical assumptions on the differential form defining $\F$, Question~\ref{question} has also a positive answer.

 \section{Incompressibility of leaves} \label{secIncompr}

The goal of this chapter  is to  
bound  for any generalized curve the complexity of the topology of the leaves as immersed Riemann surfaces. This  is given by the following result:

\begin{teo}[Incompressibility Theorem]\label{IncomprThm} Let $\F$ be a generalized curve, $S$ a $\F$-appropriate curve in a Milnor tube $\mb T$ for $S$ and  $C$ a non-necessarily irreducible closed analytic curve in $\mb T$ such that $C\setminus \{0\}$ is smooth and transversal to $\F$. 
Then there exists a fundamental system $(U_n)_{n\in \N}$, $U_{n+1}\subset U_n$,  of open neighborhoods of $S$ in $\mb T$ such that  for each $n\in \N$ we have:
\begin{enumerate}
\item\label{inclpi1} the inclusion map $U_n\hookrightarrow \mb T$ induces an isomorphism between the fundamental groups of $U_n\setminus S$ and $\mb T\setminus S$;
\item\label{incomp} for each leaf $L$ of the foliation induced by $\F$on ${U_n\setminus S}$, the inclusion map  $L\hookrightarrow U_n\setminus S$ induces an injective morphism $\pi_1(L) \hookrightarrow \pi_1(U_n\setminus S)$;
\item\label{interDelta}  any connected component $C'\subset q^{-1}(C)$ of the pull back  of $C\cap U_n$ by the universal covering map $q: \tilde U_n\to U_n\setminus S$ is an embedded open disc such that, for each leaf $L'$ of the foliation    induced by $q^{-1} \F$ on $\tilde U_n$,  we have  $\mr{card}(L'\cap C') \leq1$. 
\end{enumerate}
\end{teo}
\noindent This result was obtained by two of us  in \cite{MM1} for  generalized curves under the additional assumptions of non-dicriticity and the absence,  after reduction, of singularities that are simultaneously non-resonant and non-linearizable. The non-dicriticity assumption was removed in \cite{MM4} by introducing the notion of appropriate curve, cf. Definition~\ref{Fappropcurve}.  Finally in  \cite{Loic} L. Teyssier extended the result  
under the sole generalized curve assumption;
he also proved the necessity of this hypothesis by constructing non-dicritical foliations which are non-generalized curves possessing  compressible leaves in the complementary of  their unique  appropriate curve, cf. \S\ref{Loic}. 

In the following two paragraphs we will describe the steps of the proof of Theorem~\ref{IncomprThm}.

\subsection{Foliated connectedness and a foliated Van Kampen Theorem}\label{folconVK}
To study the incompressibility property of the leaves of $\F$ in the complement of a $\F$-appropriate curve~$S$, the main tool is the notion of $\F$-connectedness introduced in \cite{MM1}. It allows to localize this global property of each leaf,   by transforming it into a local property of the foliations induced by~$\F$ on each block of a  suitable  decomposition of the space. Each block of this decomposition will be contained in a geometric block of the decomposition of a Milnor tube  of $S$ described in~\S\ref{decompCompCurve}.
\begin{defin}\label{defuncopr}  Let $A\subset B$ be subsets of a manifold $M$ endowed with a regular  foliation~$\G$. We say that \emph{$A$ is strictly $\G$-connected in $B$} and we denote $A\Gcon B$ or $B\Gconl A$,  if the following properties are fulfilled:
\begin{enumerate}
\item\label{uncomprAB} $A$ is incompressible\footnote{i.e. the inclusion $A\subset B$ induces monomorphisms $\pi_1(A,a)\hookrightarrow\pi_1(B,a)$ for every $a\in A$.}  in $B$,
 \item for any  leaf $L$ of $\G$,  if a path  $\beta : [0,1] \to L\cap B$ is homotopic
in $B$ to a path $\alpha:[0,1] \to  A$ then $\beta$  is also  homotopic in $L\cap B$ to a path $\gamma:[0,1] \to L\cap A$ and the paths $\alpha$ and $\gamma$ are homotopic in $A$.
\end{enumerate}
\end{defin}
\begin{figure}[ht]
\centerline{\includegraphics[width=5cm]{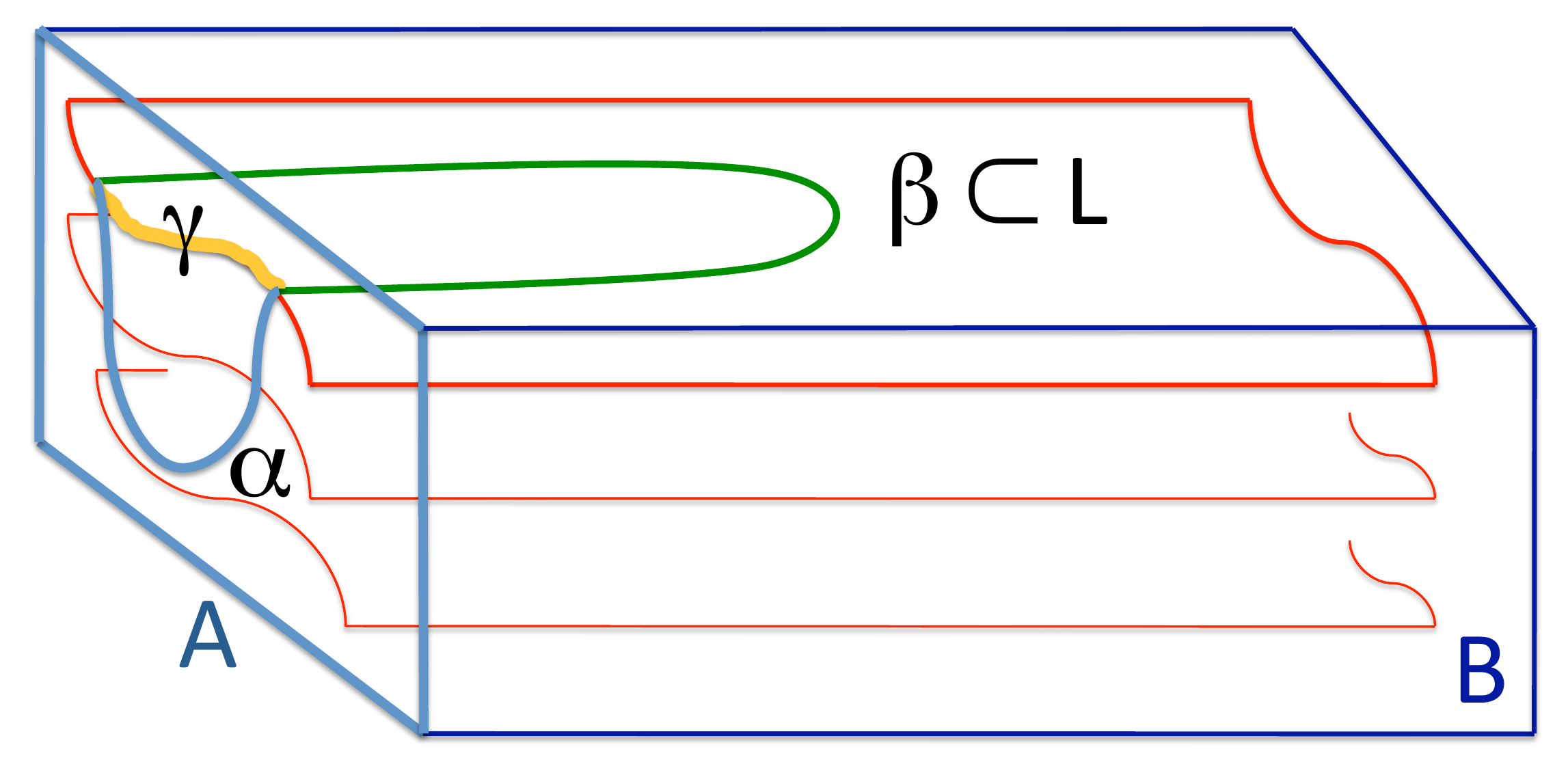}}
\end{figure}

A remarkable property of this notion is that if the set $A$ is  reduced to a single point, $A$ is $\G$-connected in $B$ if and only if the leaf $L$ containing this point is incompressible in $B$. Another trivial but useful property is the transitivity of this relation:
\begin{equation}\label{transitivity}
\Big(A{\Gcon}B\hbox{ and }B{\Gcon}C \Big)\Longrightarrow A{\Gcon}C\,.
\end{equation}
Thus, to prove the incompressibility property of leaves stated in Theorem~\ref{IncomprThm}, it suffices to  construct  a filtration in the Milnor tube ${\mc T}_{\varepsilon,\eta}:=E_\F^{-1}(\mb T_{\varepsilon,\eta})$ of $S^\sharp=E_\F^{-1}(S)$
\begin{equation}\label{exhaustion}
W_0\subset W_1\subset\cdots\subset W_m\;\subset \;{\mc T}_{\varepsilon,\eta}\,,
\end{equation}
satisfying:
\begin{enumerate}
\item[(a)] any leaf of $\F^\sharp$ restricted to $W_0^\ast:=W_0\setminus S^\sharp$ is incompressible in $W_0^\ast$,
\item[(b)] if we denote $W_i^\ast:=W_i\setminus S^\sharp$, then we have $W^\ast_{i-1}\Fcon W^\ast_{i}$,  $i=1,\ldots,m$,
\item[(c)]  $W_m$ is an open set in $\mc T_{\varepsilon,\eta}$, containing a Milnor tube
$\mc T_{\varepsilon,\eta'}$ of $S^\sharp$  and the inclusions $\mc T_{\varepsilon,\eta'}\setminus S^\sharp \subset W_m^\ast\subset \mc T_{\varepsilon,\eta}\setminus S^\sharp$ induce isomorphisms between their fundamental groups.
\end{enumerate}
In order to obtain property (b) the method developed in \cite{MM1} is to construct a sequence $(V_j)_{j=0}^m$ of real four dimensional submanifolds of $\mc T_{\varepsilon,\eta}$ with boundary, that we call \emph{foliated blocks}, such that after setting
 \begin{equation}\label{constrUi}
 W_i:=V_0\cup\cdots\cup V_i\,, 
\quad i=0,\ldots,m\,,
 \end{equation}
 each pair $(W_{i-1}^\ast,V_i^\ast)$, where $V_i^\ast:=V_i\setminus  S^\sharp$, $i=1,\ldots,m$, satisfies the relation
 \[ 
 W_{i-1}^\ast \stackrel{\F^\sharp}{\fconl} W_{i-1}^\ast\cap V_i^\ast\stackrel{\F^\sharp}{\fcon} V_i^\ast\,.
  \]
The following \emph{foliated Van Kampen Theorem}  gives then  the desired  relation 
\[W^*_{i-1}\stackrel{\F^\sharp}{\fcon} W^*_{i-1}\cup V_i^*=W_i^*.\]

\begin{teo}\cite[Th\'eor\`eme~1.3.1]{MM1} Let $(A, B)$ be a pair of  connected subsets of a manifold $M$ endowed with a regular foliation $\G$ of class  $\mc C^1$. Suppose that  
$A\cap B$   is a (non-necessary connected) transversely orientable submanifold of $M$,
that is transversal to~$\G$, and that  $A\cup B$ is a neighborhood of  $A\cap B$. 
Then we have:
\[ \Big(
A\Gconl A\cap B\Gcon B\Big)\Longrightarrow A\Gcon A\cup B\,.
 \]
\end{teo}

\noindent Notice that the above implication is a corollary of the classical Van Kampen Theorem when we only consider the incompressibility property~(\ref{uncomprAB}) in Definition~\ref{defuncopr}.

\subsection{Construction of foliated blocks}\label{subsecConstrElemBlocs}
For each irreducible component $D$ of the total transform
$S^\sharp$ of the $\F$-appropriate curve $S$ in Theorem~\ref{IncomprThm}, we define
$K_D$ as the complement in $D$ of disjoint conformal discs centered at the points of $\mr{Sing}( S^\sharp)\cap D$ that do not belong to any dead branch.
The foliated blocks $V_j$ used in (\ref{constrUi}) to construct the exhaustion $(W_i)_{i=0}^m$  are of three types: 
\begin{itemize}
\item  \emph{Foliated Seifert blocks} $\mathscr B_D$ that
are contained in geometric blocks of the form $\mc B_D$ described in (B\ref{BB2}) of \S\ref{decompCompCurve}, the inclusion map inducing an isomorphism $\pi_1(\mathscr B_D\setminus S^\sharp)\simeq\pi_1(\mc B_D\setminus  S^\sharp)$. 
Moreover, $\mathscr B_D$ is a neighborhood of the union of $K_D$ with all the dead branches of $ S^\sharp$ meeting $D$.
\item  \emph{Foliated collar blocks}  $\mathscr B_C$ that are neighborhoods of either
\begin{enumerate}[(a)]
\item a set $K_D$  with $D$  a non-dicritical irreducible component of valence $2$ in a chain of $ S^\sharp$, or
\item a set  $K_D$ with $D$ the strict transform of an irreducible component of $S$, or
\item a set $K_s= S^\sharp\cap \{|u|,|v|\le 1\}$ where $(u,v)$ are local coordinates centered at a  singular point $s$ of $ S^\sharp$ belonging to a chain and $uv=0$ is a local equation of~$S^\sharp$.
\end{enumerate}
The name foliated collar is justified by the fact that there is a homeomorphism $\varphi: \mb D^*\times\mb S^1\times (0,1)\iso\inte{\mathscr B}_C\setminus S^\sharp$ such that $\varphi(\{(z,\theta)\}\times(0,1))$ is contained in a leaf of $\F^\sharp$, where $\inte{\mathscr B}_C$ denotes the interior of $\mathscr B_C$.
\item  \emph{Gluing blocks} that are neighborhoods of $K_D$ with $D$ a dicritical component of~$S^\sharp$.
\end{itemize}
Moreover, if $\mathscr B$ and $\mathscr B'$ are foliated blocks such that
$\mathscr B\cap\mathscr B'\neq\emptyset$,  then this intersection is a common connected component of $\partial\mathscr B$ and of $\partial\mathscr B'$, 
satisfying
\[\mathscr B\setminus S^\sharp\Fconl (\mathscr B\cap\mathscr B')\setminus S^\sharp\Fcon\mathscr B'\setminus S^\sharp.\]
Each block is built  by an induction process, from the common
boundary component of one of the previously constructed blocks. 

The boundary components of each block
are all  sets of \emph{suspension type}. This means that  they are constructed as follows. First we fix  a submersion $\rho$ defined on a neighborhood $\Omega$ of an invariant irreducible component $D$ of $ S^\sharp$, that is equal to the identity map on $D$, and we also fix a simple  loop $\gamma:[0,1]\to D\setminus \mr{Sing}( S^\sharp)$ and a holomorphic coordinate 
\[z:\rho^{-1}(\gamma(0))\iso \{\vert z\vert\leq 1\}\subset\C\,,\qquad z(\gamma(0))=0\,.
\]
Then we consider a closed conformal disc $\Delta\subset \rho^{-1}(\gamma(0))$  containing $\gamma(0)$  in its interior whose boundary $\partial\Delta$  is parametrized  by a  piecewise real analytic path $\delta$.  We assume $\Delta$ small enough so that  for any point $m\in \Delta$, $\gamma$ lifts through  $\rho$  to the leaf $L_m$ of $\F^\sharp$ containing $m$, defining a (unique)  path $\gamma_m:[0,1]\to L_m\cap\Omega$  with origin $\gamma_m(0)=m$, that satisfies $\rho\circ\gamma_m=\gamma$. 
The set 
\begin{equation*}\label{suspensiontype}
T_\Delta := \{\gamma_m(t)\;|\; m\in \Delta, 0\leq t\leq 1\}
\end{equation*}
is called \emph{set of suspension type relatively to $\Delta$ over $\gamma$}. 

Let us highlight that the set 
$T^\ast_\Delta:=T_\Delta\setminus  S^\sharp$
may have a complex topology because the fundamental group of $\Delta\cup h_\gamma(\Delta)$ may be non-trivial,  $h_\gamma$ denoting the holonomy map $m\mapsto \gamma_m(1)$. To avoid this possibility  we must ensure that $\Delta$ and $h_\gamma(\Delta)$ are star-shaped. To control this property we fix a parametrization $\delta(t)$ of $\partial\Delta$ and we  define the \emph{roughness}  of $\Delta$ as the maximum $e_z(\Delta)$ of the angles between   the tangents  $\delta_{\pm}'(t)$ to $\partial\Delta$ at $\delta(t)$ and the tangent $i\delta(t)$ to the circle passing through this point with center  $z=0$. The star-shaped property is fulfilled when the roughness is $
<\pi/2$. Moreover, as consequence of the smoothness of $h_\gamma$,  the quotient $| e_z(\Delta)/e_z(h_\gamma(\Delta))|$ tends to $1$  when the diameter of $\Delta$ tends to $0$. Hence we have that $T^\ast_\Delta$
has the homotopy type of  a $2$-torus when the roughness and the diameter of $\Delta$ are small enough.

The inductive process to construct  foliated blocks is based on the following existence property (cf. \cite[Th\'eor\`eme 3.2.1]{MM1}).
For every suspension type set $T$ over a loop $\gamma\subset\partial K$, with $K=K_D$ or $K=K_s$, there exists a foliated  block $\mathscr B$ which is a neighborhood of $K$ such that
\begin{itemize}
\item $T$ is a connected component of $\partial\mathscr B$,
\item every connected component $T'$ of $\partial\mathscr B$ is a suspension type set whose roughness is controlled by that of $T$ and it satisfies 
$T'\setminus S^\sharp\Fcon \mathscr B\setminus S^\sharp$,
\item the leaves of the restriction $\F^\sharp_{|\mathscr B\setminus S^\sharp}$ of $\F^\sharp$ to $\mathscr B\setminus S^\sharp$ are incompressible in $\mathscr B\setminus S^\sharp$,
\item for any small enough Milnor tube $\mb T_{\varepsilon,\eta}$ of $S$ we have $\pi_1(\mathscr B\setminus  S^\sharp)\simeq\pi_1(W\cap E_S^{-1}(\mb T^*
_{\varepsilon,\eta}))$, where $W$ is a fixed  
neighborhood of $K$.
\end{itemize}

The order of construction of the blocks is given by first fixing a suitable  exhaustion of~$S^\sharp$ which will correspond to the exhaustion 
\begin{equation*}\label{exhaustionS}
W_0\cap S^\sharp\subset W_1\cap S^\sharp\subset\cdots\subset W_m\cap S^\sharp=S^\sharp.
\end{equation*}
determined by~(\ref{exhaustion}) at the end of the induction process.

\section{Examples}

\subsection{Dicritical cuspidal singularity}\label{Exdic}
We revisit Example~1.1 of \cite{MM4}.
Consider the dicritical foliation $\F$ in $(\C^2,0)$ defined by the rational first integral $f(x,y)=\frac{y^2-x^3}{x^2}$ whose isolated separatrix set is the cusp $S=\{y^2-x^3=0\}$. We will first compute the fundamental group of the complement of $S$ in a Milnor ball $B$, then we will show that not all the leaves of $\F$ are incompressible in $B\setminus S$, but by adding to $S$ a non-isolated separatrix $S'$ of $\F$ the incompressibility property  will be fulfilled for all the leaves in $B\setminus(S\cup S')$.\\

\subsubsection{}\label{271} We consider the composition $E:M\to B$ of the three blowing-ups defining the minimal desingularization of $S$, that also coincides with the reduction map of the foliation. The exceptional divisor $\mc E=E^{-1}(0)$ has three irreducible components $D_1$, $D_2$, $D_3$ which we numerate according to the order that they appear in the blowing-up process.
Thus, the corresponding self-intersections are 
\[
D_1\cdot D_1=-3\,,\quad D_2\cdot D_2=-2\,,\quad D_3\cdot D_3=-1\,.
\]
The strict transform $\mc S$ of $S$ only meets $D_3$ and consequently $D_1$ and $D_2$ are dead branches attached to $D_3$. In addition, $D_1$ is a dicritical component for $\F$ (in fact, it is totally transverse to $\F^\sharp=E^{-1}\F$). Let us illustrate the computation of the fundamental group of $B\setminus S\simeq M\setminus (\mc E\cup\mc S)$. For each $i=1,2,3$ we consider a tubular neighborhood $U_i$ of $D_i$ which is a disc fibration over $D_i\simeq\mb P^1$ of Chern class $k=D_i\cdot D_i$. For $i=1,2$, $U_i\cap(\mc E\cup \mc S)$ consists in $D_i$ and $r=1$ fiber and for $i=3$, $U_i\cap(\mc E\cup \mc S)$ consists in $D_i$ and $r=3$ fibers. We can cover each 
\[
 U=U_i=\varphi_x(\mb D\times\mb D^*)\cup\varphi_y(\mb D\times\mb D^*) \]
 by two trivializing charts $\varphi_x(t,x)$, $\varphi_y(s,y)$ with transition functions $ts=1$, $y=t^{-k}x$. We can assume that 
\[U^*:=U\setminus(\mc E\cup \mc S)=\varphi_x((\mb D\setminus\{t_1,\ldots,t_r\})\times\mb D^*)\cup\varphi_y(\mb D\times\mb D^*).\] 
We consider simple loops $\bar \alpha_j(u)$ in $\mb D\setminus\{t_1,\ldots,t_r\}$ of index $1$ around $t_j$ so that the product $\bar \alpha_0=\bar \alpha_1\cdots \bar \alpha_r$ is homotopic to $e^{2i\pi u}$. We define $\alpha_j(u)=\varphi_x(\bar \alpha_j(u),1)$ and $\gamma(u)=\varphi_x(1,e^{2i\pi u})$.
Then the loop 
\[
 \alpha_0\gamma^k\sim \varphi_x(e^{2i\pi u},e^{2i\pi ku})=\varphi_y(e^{-2i\pi u},1)
 \]
is null-homotopic in $\varphi_y(\mb D\times\mb D^*)\subset U^*$. By applying Seifert-Van Kampen's theorem we obtain that 
\[\pi_1(U^*)=\langle \alpha_1,\ldots,\alpha_r,\gamma\,|\, \alpha_1\cdots \alpha_r=\gamma^{-k}, [\alpha_j,\gamma]=1\rangle.\]
Particularizing this computation to each $U_i$ we deduce that
\[\pi_1(U_1^*)=\langle a,c\,|\, c=a^3, [a,c]=1\rangle,\quad\pi_1(U_2^*)=\langle b,c\,|\,c=b^2, [b,c]=1\rangle\]
and
\[\pi_1(U_3^*)=\langle a,b,c,d\,|\, abd=c, [a,c]=[b,c]=[b,d]=1\rangle,\]
where the loops $a,b,c,d$ are depicted in Figure~\ref{cusp}. Notice that the loops $a,b,c,d$ can be also defined as boundaries (meridians) of fibers of disc fibrations over $D_1$, $D_2$, $D_3$ and $\mc S$ respectively.
Since any geometric realization of the dual tree of $\mc E\cup\mc S$ is contractible, a suitable embedding of it in $M\setminus(\mc E\cup\mc S)$ described in \cite[\S3.1]{MM2},  can be chosen  as a common ``base point'' of the loops $a,b,c,d$, cf.  Figure~\ref{cusp}.

\begin{center}
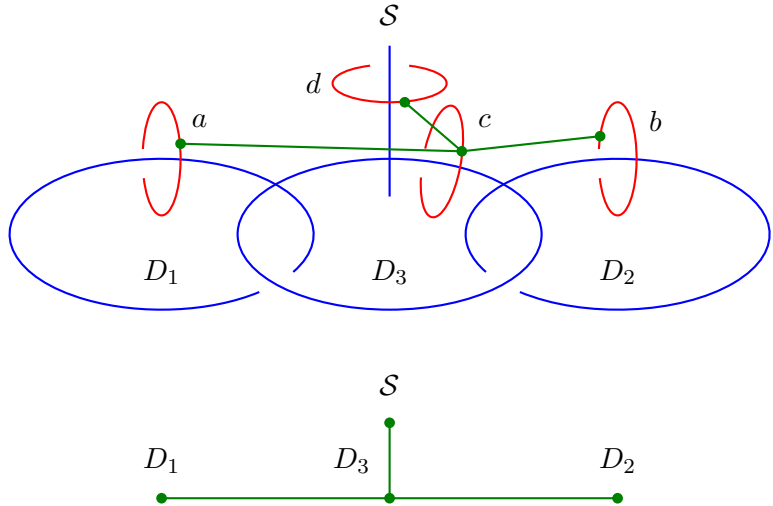
\begin{figure}[ht]
\begin{tikzpicture}
\draw[thick,blue] (2,2.5) [elliparc=-30:310:2cm and 1cm];
\draw[thick,blue] (5,2.5) [elliparc=0:360:2cm and 1cm];
\draw[thick,blue] (8,2.5) [elliparc=-130:210:2cm and 1cm];
\draw[thick,blue] (5,3) to (5,5);
\draw[thick,red] (2,3.5) [elliparc=-160:170:.25cm and .75cm];
\draw[thick,red] (8,3.5) [elliparc=-160:170:.25cm and .75cm];
\draw[thick,red] (5,4.5) [elliparc=-250:70:.75cm and .25cm];
\begin{scope}[rotate=-10]
\draw[thick,red] (5,4.4) [elliparc=-160:170:.25cm and .75cm];
\end{scope}
\fill [green] (2.25,3.7) circle [radius=2pt];
\fill [green] (5.95,3.6) circle [radius=2pt];
\draw[green,thick] (2.25,3.7) to (5.95,3.6) to (5.2,4.25);
\fill [green] (5.2,4.25) circle [radius=2pt];
\fill [green] (7.77,3.8) circle [radius=2pt];
\draw [green,thick] (5.95,3.6) to (7.77,3.8);
\node at (2,2) {$D_1$};
\node at (5,2) {$D_3$};
\node at (8,2) {$D_2$};
\node at (2.5,4) {$a$};
\node at (6.25,4) {$c$};
\node at (8.5,4) {$b$};
\node at (4,4.5) {$d$};
\node at (5,5.4) {$\mc S$};
\fill [green] (2,-1) circle [radius=2pt];
\fill [green] (5,-1) circle [radius=2pt];
\fill [green] (5,0) circle [radius=2pt];
\fill [green] (8,-1) circle [radius=2pt];
\draw [green,thick] (2,-1) to (5,-1) to (5,0);
\draw [green,thick] (5,-1) to (8,-1); 
\node at (2,-0.5) {$D_1$};
\node at (4.5,-0.5) {$D_3$};
\node at (5,0.5) {$\mc S$};
\node at (8,-0.5) {$D_2$};
\end{tikzpicture}
\caption{The strict transform $\mc S$ of of the cusp $S=\{y^2-x^3=0\}$, the total transform $S^\sharp=D_1\cup D_2\cup D_3\cup\mc S$ with its dual tree, and the generators $a,b,c,d$ of the fundamental group of its complement.}\label{cusp}
\end{figure}
\end{center}
By applying again Seifert-Van Kampen's theorem we obtain the following presentation of the  fundamental group of $M^*:=M\setminus(\mc E\cup \mc S)$:
\begin{align*}
\pi_1(M^*)&=\pi_1(U_1^*)\ast_{\pi_1(U_1^*\cap U_3^*)}\pi_1(U_3^*)\ast_{\pi_1(U_3^*\cap U_2^*)}\pi_1(U_2^*)\\
&=\langle a,b,c \,|\, a^3=b^2=c\rangle.
\end{align*}
This method can be applied to compute the fundamental group of the complement of any germ of curve (see \cite[Proposition 3.3]{MM2}).\\

\subsubsection{}Now, we shall show that there exist non-incompressible leaves of $\F$ inside $B\setminus S$. Looking at the situation after the first blow-up $E_1:(M_1,D_1)\to (B,0)$ in the chart $(t,x)$ such that $(x,y)=E_1(t,x)=(x,tx)$ we have that 
\[
f_1(t,x):=f(x,tx)=t^2-x\] 
and
we see that there are two types of leaves of $\F$: those that are near the isolated separatrix~$S$, which are discs minus two points and the others which are diffeomorphic to~$\mb D^*$. If $L$ is a leaf of the first kind then $\pi_1(L)$ is a free group of rank $2$. We claim that we can choose the generators  $\lambda,\mu$ of $\pi_1(L)$ so that the morphism $\imath:\pi_1(L)\to\pi_1(M^*)$ induced by the inclusion is given by $\imath(\lambda)=a$ and $\imath(\mu)=b^{-1}ab$. It will then follow that 
\[
 \imath(\lambda^3\mu^{-3})=a^3b^{-1}a^{-3}b=1\] and consequently $L$ will not be incompressible in $B\setminus S\simeq M^*$. \\

At first,  let $\varepsilon>0$ be small enough so that $V:=f_1^{-1}(\mb D_\varepsilon)$ is a retract by deformation of the image of $U_2\cup U_3$ by the blowing-up $M\to M_1$ and the restriction of $f_1$ to $W:=f_1^{-1}(\mb S^1_\varepsilon)\setminus  D_1$ is a locally trivial $C^\infty$-fibration over the circle $\mb S^1_\varepsilon$ of radius $\varepsilon$, whose fiber over $\varepsilon$ is isomorphic to $F_\varepsilon:=\mb D\setminus\{\pm\sqrt{\varepsilon}\}$. 
The pull-back of $f_1:W\to\mb S^1_\varepsilon$ by the exponential map $[0,1]\to\mb S^1_{\varepsilon}$, $u\mapsto \varepsilon e^{2i\pi u}$, is trivial and there is a trivializing map 
\[
\tau:F_\varepsilon\times [0,1]\to W\,,\qquad (z,u)\mapsto(t,x)=(ze^{i\pi u},(z^2-\varepsilon)e^{2i\pi u})\,.
\]
We consider the path $\beta:[0,1]\to F_\varepsilon\times[0,1]$ given by $\beta(u)=(0,u)$ projecting by $\tau$ onto the loop $(t,x)=(0,-\varepsilon e^{2i\pi u})\in W\subset U_1$ which is a meridian of $D_2$, so that we can choose the generator $b$ of $\pi_1(M^*)$ as the homotopy class of $\tau(\beta)$.\\

Now, let $z(u)$ be a simple loop in $F_\varepsilon$ based in $z=0$ having index $+1$ around $+\sqrt{\varepsilon}$ and index $0$ around $-\sqrt{\varepsilon}$. We define $\alpha_i(u)=(z(u),i)$ for $i=0,1$. It is clear that $\alpha_1$ is homotopic to $\beta\alpha_0\beta^{-1}$ in $F_\varepsilon\times[0,1]$. Hence their respective projections by $\tau$, $\lambda=\tau(\alpha_1)$ and $\tau(\beta\alpha_0\beta^{-1})=b\mu b^{-1}$ are also homotopic in $W$, where $\mu=\tau(\alpha_0)$. Notice that the loops $\lambda,\mu$, which are contained in the leaf $L$ passing through the point $(t,x)=(0,-\varepsilon)$, are meridians around $D_1$ so that we can choose the generator $a\in\pi_1(M^*)$ as the homotopy class of $\lambda$. Moreover, $\lambda$ (resp. $\mu$) turns around the point $(t,x)=(+\sqrt{\varepsilon},0)\in\bar L\cap D_1$  (resp. $(t,x)=(-\sqrt{\varepsilon},0)\in\bar L\cap D_1$). Hence the free fundamental group of $L$ is generated by $\lambda$ and $\mu$ whose images by $\imath:\pi_1(L)\to\pi_1(M^*)$ are $a$ and $b^{-1}ab$ respectively.\\

Finally, if we remove the non-isolated separatrix $S'=\{x=0\}$ then all the leaves of $\F$ are incompressible in $B\setminus(S\cup S')$. Indeed, using the description given in \S\ref{271} we obtain that 
\[  
\pi_1(U_1\setminus (\mathcal{E}\cup\mc S\cup \mc S'))=\langle a,c,e\,|\, ce=a^3,[a,c]=[a,e]=1\rangle\,,
 \]
 where $e$ is a meridian of  the strict transform $\mc S'$ of $S'$; consequently
\begin{align*}
\pi_1(M\setminus(\mc E\cup\mc S\cup\mc S'))&=\langle a,b,c,d,e\,|\, ce=a^3, abd=c=b^2,
[a,e]=[c,a]=[c,b]=[c,d]=1 \rangle\\
&=\langle a,b\, |\, [a,b^2]=1\rangle\,.
\end{align*}
The quotient of this group by the normal subgroup $\langle b^2\rangle$ is the free product $\Z a\ast\Z_2\bar b$. The composition morphism
\[
\langle \lambda,\mu\rangle\to \langle a,b\,|\,[a,b^2]=1\rangle\to\Z a\ast \Z_2\bar b\,,
\]
sending the free generators $\lambda$ and $\mu$ of $\pi_1(L)$  to $a$ and $\bar b a \bar b$ respectively, is clearly injective and so is $\pi_1(L)\hookrightarrow\pi_1(M\setminus(\mc E\cup\mc S\cup\mc S'))$. Alternatively, it can be seen that $W$ is a retract by deformation of $M\setminus(\mc E\cup\mc S\cup\mc S')$ and
$\pi_1(L)=\pi_1(F_\varepsilon)$ injects into $\pi_1(M\setminus(\mc E\cup\mc  S\cup\mc S'))=\pi_1(W)$ thanks to the long exact sequence of the fibration $F_\varepsilon\to W\stackrel{f_1}{\to}\mb S^1_\varepsilon$.

\subsection{Foliations which are not generalized curves}\label{Loic}
Every saddle-node has a convergent separatrix (called strong) and a formal separatrix (called weak) that may be divergent. If all the saddle-nodes that appear in the minimal reduction of a foliation have the strong separatrix transverse to the exceptional divisor, the foliation is called \emph{strongly presentable} in \cite[D\'efinition 1.2]{Loic}. The general result of L. Teyssier \cite[Th\'eor\`eme~C]{Loic} is that
the  properties (\ref{inclpi1}) and (\ref{incomp}) in
 Incompressibility Theorem \ref{IncomprThm} remain true for strongly presentable foliations and not only for generalized curves.
Moreover, property (\ref{interDelta}) also remains true by choosing suitably the curve $C$ in Theorem \ref{IncomprThm}.

On the other hand, L. Teyssier \cite[\S8.2]{Loic} exhibits examples of non-dicritical foliations $\F$ that are not strongly presentable for which the fundamental group of the leaves do not inject into the fundamental group of the complement of the separatrices curve $S_\F$. This is a consequence of the fact that the strong separatrix of some of the saddle-nodes appearing in the singularity reduction of $\F$ is contained in the exceptional divisor and the corresponding weak separatrix is divergent so that the fundamental group  $\Z\oplus\Z$ of the complement of $S_\F$ is too small to contain the  fundamental group (which is non-abelian free group) of a generic leaf (which is a non-compact Riemann surface).

Surprisingly, he also presents \cite[\S8.3]{Loic} an example of a foliation reduced after one blowing-up in which all saddle-nodes have two convergent separatrices admitting compressible leaves in the complement of the separatrices. The construction is based on the existence of transverse curves (discovered by E. Paul) for which foliated connectedness property (\ref{interDelta}) of Theorem  \ref{IncomprThm} fails.

\section{Monodromy}
The notion of  monodromy of a foliation reflects the action of the fundamental group of the space on the ``ends'' of leaves or, more specifically, on the leaf  spaces inverse system. Before we introduce this notion in \textsection \ref{subsecmonodromy}, we first examine in \textsection \ref{EndLeavesRedFol} how for reduced foliations the structure of the leaf spaces contain all the useful informations to classify these foliations. In paragraph \textsection \ref{subsecleavespaces} we see that Incompressibility Theorem \ref{IncomprThm} allows to endow any leaf space with a (in general non-Haussdorff) complex manifold structure. In \textsection \ref{subsecmonodromy}, after defining the  notion of monodromy, we make explicit  its relation to  the holonomies of the invariant irreducible components of the exceptional divisor of the reduction, and to the  extended holonomy defined in \S\ref{ext-hol}. That one corresponds in the case of  foliations that can be  reduced by  a single blowing-up, to the notion  introduced  by L. Ortiz-bobadilla, E. Rosales-Gonz\'alez and S.M. Voronin in  \cite{OBRGV3}. Finally in \S\ref{TCT} we give the statement and an idea of the proof of  Classification Theorem~\ref{ClassThm}  in terms of monodromy and Camacho-Sad  indices.

\subsection{Ends of leaves space of reduced foliations}\label{EndLeavesRedFol}

The analytic classification of reduced foliation singularities can be reformulated in the framework of ends of leaves space.
Let us first recall the case of a saddle $\F$ defined by a $1$-form 
\[\omega=(1+\cdots)ydx-(\lambda+\cdots)xdy,\quad\text{with}\quad \lambda\notin\mb Q_{\geq 0}.\]
The axes $\{x=0\}$ and $\{y=0\}$ are the only separatrices and $S:=\{xy=0\}$ is a $\F$-appropriate curve.
Assume that $\omega$ is holomorphic on the polydisc $P=\{|x|\le 1,\ |y|\le 1\}$ and let us fix a fundamental system of open neighborhoods $\mc U$ of $S\cap P$ in $P$, ordered by the inclusion. For any $U\in\mc U$ we denote by $Q_U$ the leaf space of the foliation $\F_{|U\setminus S}$ induced by $\F$ on $U\setminus S$. They form \emph{the inverse system of leaf spaces} together with the continuous maps
 \begin{equation}\label{systespfeuilles}
 \rho_{VU} : Q_U\longrightarrow Q_V\,,\quad V \supset U\,,\quad U, V \in \mc U\,,
 \end{equation}
 that send a leaf $L_U$ of $\F_{|U\setminus S}$ to the leaf of $\F_{|V\setminus S}$ that contains $L_U$.
The \emph{space of ends of leaves of $\F$ along $S\cap P$} is the inverse limit
 \begin{equation*}\label{espBouts}
 \limproj_{U\in \mc U}\, Q_U := \left\{ (L_U)_{U\in \mc U}\;;\; L_U\in  Q_U\;,  V\supset U \Longrightarrow L_U\subset L_V \right\} 
 \subset \prod_{U\in \mc U} Q_U\,,
 \end{equation*}
which does not depend on the choice of the fundamental system $\mc U$. We distinguish three cases:

\subsubsection{Poincar\'e type singularity $\lambda\notin\R$}\label{PoincareType}
We can construct $\mc U=(U_r)_{0<r<1}$ such that the intersections $U_r\cap\{|x|=1\}$ are suspension type subsets relatively to the  transverse discs $\Delta_r:=\{x=1,\, |y|<r\}$ over the loop $\gamma(s)=(e^{2i\pi s},0)$, cf. \S\ref{subsecConstrElemBlocs}, and such that their saturations by $\F_{|U_r}$ are the whole $U_r$.
Thus the space $Q_{U_r}$ can be identified with the space of orbits of the holonomy $h:\Delta_r\iso h(\Delta_r)$ of the foliation along $\gamma$, which is a linearizable contraction (or a dilatation). Hence $Q_{U_r}$ is the elliptic curve 
\[
\C/(\Z2i\pi+\Z2i\pi\lambda)\simeq\C/(\Z+\Z\lambda)
\]
and the projections $\underleftarrow{\lim}_{r}\, Q_{U_r}\to Q_{U_r}$ are isomorphisms. Classically $\lambda$ determines the analytic type of $\underleftarrow{\lim}_{r}\,Q_{U_r}$ 
and $\lambda$ is the only analytic invariant of $h$ and hence of $\F$ because it is linearizable.

Notice that all Poincar\'e type foliations are topologically conjugated, see Lemma~\ref{Poincare}. Indeed, the leaves are transverse to sufficiently small spheres so they are cones over the intersections with a sphere. The induced foliation on the sphere is a dimension one real foliation of Morse-Smale type with two closed leaves.

\subsubsection{Non-linearizable resonant saddles}\label{322}

The topological classification of these foliations is given by three natural numbers $p,q,k\in\N$, $\gcd(p,q)=1$. Indeed, by \cite{Camacho-Sad}
any non-linearizable resonant saddle $\F$ is 
topologically conjugated to the foliation defined by
\[\omega_{p,q,k,\nu}=py(1+(\nu-1)(x^py^q)^k)dx+qx(1+\nu(x^py^q)^k)dy\]
for any $\nu\in\C$, which is an analytic (and formal) invariant  of $\F$ but not a topological one.

As in the previous case we can use a fundamental system $(U_r)_r$ of neighborhoods  of the axes using  transverse discs $\Delta_r$ which allows to identify again the leaf space $Q_{U_r}$ to the space of orbits of the holonomy of the loop $\gamma$ realized on $\Delta_r$. The leaf space $Q_{U_r}$  is again a Riemann surface, but in contrast to the previous case, is not Hausdorff. More precisely, $Q_{U_r}$ has a structure of ``chapelet de sph\`eres'' (string of spheres), that is, the quotient of the disjoint union 
$\bigsqcup_{j\in \Z/k\Z} (S_j^- \cup S_j^+)$  of punctured spheres
 $S_j^+=S_j^-=\overline{\C}\setminus \{0,\infty\}$, 
 under the equivalence relation $m\sim\varphi_j^\star(m)$ induced by ``gluing'' biholomorphisms 
\[ \varphi_j^0\;\,:\;\, S_j^+\supset (V^+_j\setminus \{0\}) \;\stackrel{\sim}{\longrightarrow}\; (V^-_j\setminus \{0\})\subset S_j^- 
\]
\[ 
\varphi_j^\infty\;\,:\;\, S_j^-\supset (W^-_j\setminus \{\infty\}) \;\stackrel{\sim}{\longrightarrow}\; (W^+_{j+1}\setminus \{\infty\})\subset S_{j+1}^+\,,
 \]
where $V_j^\pm$, resp. $W_j^\pm$, are open neighborhoods of  $0$, resp. of $\infty$,  in $S_j^\pm$ and $j\in \Z/k\Z$, see for instance \cite{MR} or \cite[p. 116]{Loray}.

The space of ends of leaves $\limproj_r\, Q_{U_r}$ is the Hausdorff non-connected Riemann surface  
$\bigsqcup_{j\in \Z/k\Z} (S_j^- \cup S_j^+)$
where  $k$ is the number of petals of the dynamics of the holonomy. To obtain a full analytic invariant of the foliation one needs to collect all the gluing data and consider the whole inverse system of leaf spaces $(Q_{U_r})_r$ introduced in (\ref{systespfeuilles}):
\begin{teo}[Birkhoff, \'Ecalle, see {\cite[Th\'eor\`eme 2.7.1]{Loray}}]\label{BE}
Two germs of non-linearizable resonant saddles $\F$ and $\F'$ with the same topological invariants $p,q,k$ and the same formal invariant $\nu$ are analytically conjugated if and only if there is an isomorphism between their inverse systems of leaf spaces.
\end{teo}

In the above statement the notion of \emph{isomorphism of inverse systems} is  as follows, see \cite[\S2.8]{Douady}: 
\begin{defin}\label{ISTop} Let  $\mr{Top}$ be the category of topological spaces and continuous maps.  
The category $\underleftarrow{\mr{Top}}$ is the category whose objects are inverse systems of topological spaces
\[ 
\mc A:=\big( (A_\alpha)_{\alpha}\,,\quad (\rho_{\alpha',\alpha} :  A_\alpha \to A_{\alpha'})_{\alpha \preccurlyeq \alpha' } \big)
 \]
and whose morphisms, that we call here \emph{continuous germs}\footnote{In \cite{MM3} they are called {pro-germs}.}, are the elements of
 \begin{equation}\label{Hom}
\underleftarrow{\mr{Hom}}\big(\mc A,\mc B\big) := {\limproj}_{\beta}\,{\limind}_{\alpha}
 \mc C^0(A_\alpha,B_\beta),
\end{equation}
where $\mc C^0=\mr{Hom_\mr{Top}}$. Similarly $\underleftarrow{\mr{\C\text{-}Man}}\subset \underleftarrow{\mr{Top}}$ is the subcategory of inverse systems of complex manifolds, its morphisms, that we call \emph{holomorphic germs}, being the elements of ${\limproj}_{\beta}\,{\limind}_{\alpha}
 \mc O(A_\alpha,B_\beta),
$ with $\mc O=\mr{Hom_{\C\text{-}Man}}$.
\end{defin}

\noindent 
Notice that for any $\beta$ the system $(
\mc O(A_\alpha,B_\beta))_\alpha$ is direct, while the system 
$\left({\limind}_{\alpha}
\mc O(A_\alpha,B_\beta)\right)_\beta$ is inverse. This explain the  limits used in the definition of $\underleftarrow{\mathrm{Hom}}$.
\\

Going back to Theorem~\ref{BE}, if $\mc Q=(Q_{U_r})_r$ and $\mc Q'=(Q_{U'_{r'}})_{r'}$ are systems of leaf spaces of resonant saddles $\F$ and $\F'$ then  the invertible elements $\underleftarrow{\mr{Iso}}(\mc Q,\mc Q')\subset \underleftarrow{\mr{Hom}}(\mc Q,\mc Q')$ correspond to biholomorphisms between the ends of leaves spaces (Hausdorff ``chapelets de sph\`eres'')
\[
\Phi:\underleftarrow{\lim}\,\mc Q\iso\underleftarrow{\lim}\,\mc Q'
\]
that commute with the gluing maps $(\varphi_j^0,\varphi_j^\infty)$. Therefore the set of automorphisms of the ends of leaves space $\underleftarrow{\lim}\,\mc Q$ strictly contains the automorphisms of the inverse system $\mc Q$.

\subsubsection{Real saddles $\lambda\in\R_{<0}$.}
 
For this type of singularities the space of ends of leaves $\underleftarrow{\lim}\,\mc Q$ provides a criterion of linearizability.

\begin{teo}
A  real saddle is linearizable if and only if its ends of leaves space is empty.
\end{teo}

\begin{proof}
If $\lambda=-q/p\in\mb Q_{<0}$ and $\omega$ is conjugated to $pydx+qxdy$ then has $x^py^q$ as first integral, no leaf accumulates on $S$ and the space of ends of leaves $\underleftarrow{\lim}\,\mc Q$ is empty.
If $\lambda\in\mb Q_{<0}$ and $\omega$ is not linearizable then the previous description in \S\ref{322} trivially shows that the ends of leaves space
 $\underleftarrow{\lim}\,\mc Q$ is non-empty.

If $\lambda\in\R_{<0}\setminus\mb Q$ and the singularity is linear in the coordinates $(x,y)$ then the leaves are contained in the closed real hypersurfaces $|x||y|^\lambda=$constant. In the non-linearizable case
P\'erez-Marco \cite{Perez-Marco-cercle} proves the existence of orbits of the holonomy accumulating to $0$. Consequently there are
leaves accumulating to the separatrices and  $\underleftarrow{\lim}\,\mc Q\neq\emptyset$. 
\end{proof}

\subsubsection{Non-reduced logarithmic singularities}
Let $\F$ be a logarithmic foliation germ defined by the meromorphic form 
\[\lambda_1\frac{df_1}{f_1}+\cdots+\lambda_r\frac{df_r}{f_r}\]
with $\lambda_i/\lambda_j\notin\R$ if $i\neq j$. By  \cite[Corollaire 2.5]{Paul2} there exists a fundamental system of neighborhoods $\mc U$ of the origin such that  leaf spaces $Q_U$, $U\in\mc U$, are all identified with 
\[\C/(\Z 2i\pi\lambda_1+\cdots+\Z 2i\pi\lambda_r)\simeq\C/(\Z\lambda_1+\cdots+\Z\lambda_r),\]
 which for $r\geq 3$ is not a topological manifold. However, the leaf space of the pull-back foliation $\tilde\F$ by the universal covering map $q:\tilde U\to U\setminus S$, $S=\{f_1\cdots f_r=0\}$, can be identified with $\C$.
This example shows the interest  for non-reduced singularities to rather consider the foliation induced on the universal covering space $\tilde U$.

 \subsection{Complex structure on leaf spaces}\label{subsecleavespaces} Let us suppose again that $\F$ is a generalized curve  and let us fix 
  a Milnor tube $\mb T$ of a $\F$-appropriate curve $S$. We respectively denote  by
\begin{equation*}\label{covmap}
q:\tilde{\mb T}\to 
\mb T\setminus S\,,\quad \Gamma:=\mr{Aut}_q(\tilde{\mb T})\,,\quad \tilde \F := q^{-1}\F\,,
\end{equation*}
the universal covering of $\mb T\setminus S$, the deck transformation group
and  the pull-back of the foliation $\F$ on $\tilde{\mb T}$. We fix a fundamental system $\mc U=(U_n)_{n\in\mb N}$ of neighborhoods of $S$ satisfying the properties stated in Incompressibility Theorem~\ref{IncomprThm} and  we consider the \emph{leaf spaces inverse system of $\tilde\F$}, i.e.
the inverse system of topological spaces
\begin{equation}\label{inversesystem}
\mc Q:=\{(Q_{n})_{n\in\mb N},\  (Q_n\leftarrow Q_m)_{m<n}\},\qquad Q_{n}:=Q^{\tilde\F}_{\tilde U_n}
\end{equation}
where $Q^{\tilde\F}_{\tilde U_n}$ denotes the leaf space of the foliation $\tilde\F_{|\tilde U_n}
$ induced by $\tilde\F$ on  
\begin{equation*}\label{tildeU}
\tilde U_n:=q^{-1}(U_n\setminus S)
\end{equation*}

We associate to any subset $A\subset q^{-1}(\mb T\setminus S)$
 the following  inverse system of topological spaces
\[(A,\infty):=
\left\{\left( A\cap \tilde U_n\right)_{n\in\mb N},\ \left(A\cap\tilde U_m\hookleftarrow A\cap \tilde U_n\right)_{m<n}\right\}\,.
\]
The 
maps \[\tau_{m,n}:A\cap \tilde U_n\to Q_{m},\quad  m<n,\] 
sending $p\in A\cap \tilde U_n$ to the leaf of $\tilde\F_{|\tilde U_m}$ passing through $p$, 
define a morphism
\begin{equation}\label{tautologicalMorphism}
\tau_{{}_{A}}:(A,\infty)\to\mc Q
\end{equation}
 in the category $\underleftarrow{\mr{Top}}$, that we call
\emph{tautological morphism associated to $A$}.\\

Assertion (\ref{interDelta}) in Theorem~\ref{IncomprThm} is not sufficient to cover the whole leaf space with coordinate charts in the dicritical case. For this reason  we give a refinement of this property:

\begin{lema}\label{transverse-curves} Let $\mb T$ a Milnor tube of an appropriate curve $S$ of a generalized curve $\F$. Then there exists a fundamental system $(U_n)_{n\in\N}$ of open neighborhoods of $S$ in $\mb T$ satisfying assertions (\ref{inclpi1}) and (\ref{incomp}) of Theorem \ref{IncomprThm} and there exists in each $U_n$ a closed curve $\Delta_n$ such that 
\begin{enumerate}[(i)]
 \item any connected component $\tilde \Delta_n^\alpha\subset q^{-1}(\Delta_n\setminus S)$ of the preimage of $\Delta_n\setminus S$ by the universal covering map $q: \tilde U_n\to U_n\setminus S$ is an embedded open disc satisfying: for each leaf $L$ of the foliation    induced by $q^{-1} \F$ on $\tilde U_n$,  we have  $\mr{card}(L\cap \tilde \Delta_n^\alpha) \leq1$;
 \item $\mr{sat}(\Delta_n\setminus S,\,U_n\setminus S)=U_n\setminus S$.
\end{enumerate}
\end{lema}  
 \begin{proof}[Idea of the proof] 
 We consider the construction process of each $U_n$ as a union of foliated blocks $V_j$ described in \S\ref{folconVK} and \S\ref{subsecConstrElemBlocs}. It follows from the construction of each $V_j$ in \cite{MM1} that we can  choose a union $\check\Delta_{j,n}\subset V_j$  of disjoint embedded discs  such that
\[ 
(\check\Delta_{j,n} \setminus S^\sharp)\Fcon (V_j\setminus S^\sharp)\quad \hbox{ and }\quad \mr{sat}(\check\Delta_{j,n} \setminus S^\sharp,V_j \setminus S^\sharp) =V_j\setminus S^\sharp\,.
 \]
  By transitivity of the foliated connectedness, see (\ref{transitivity}), the curve $\Delta_{j,n}:=E_\F(\check\Delta_{j,n}) $ satisfies  $(\Delta_{j,n}\setminus S)\Fcon (U_n\setminus S)$. This relation implies that in the universal covering of $U_n\setminus S$ each connected component of $q^{-1}(\Delta_{j,n})$ is an embedded disc that meets each leaf in at most one point. To end the proof we set  $\Delta_n:=\cup_{j}\Delta_{j,n}$.
 \end{proof}

A direct consequence of Lemma~\ref{transverse-curves} is that each  leaf space 
${Q}_{n}$ of the foliation induced by $\tilde\F=q^{-1}\F$ on $\tilde U_n$ is endowed with an atlas\footnote{In fact, it is an ``atlas'' with values in $q^{-1}(\Delta_n)$.}, called here \emph{distinguished atlas induced by~$\Delta_n$}. 
The charts of this atlas are the inverses of  the injective open
maps
\begin{equation*}\label{taualphan}
\tilde \Delta_n^\alpha\hookrightarrow Q_{n}
\end{equation*}
sending $p\in \tilde \Delta^\alpha_n$ to the leaf  of $\tilde \F_{|\tilde U_n}$ containing $p$. We  thus obtain  a structure of (not necessarily Hausdorff) Riemann surface on ${Q}_{n}$. We can see that this structure does not depend on the choice of $\Delta_n$.

\subsection{Extended Holonomy}\label{ext-hol}
After fixing a  $\F$-appropriate curve $S$, let us now consider an invariant irreducible component $D$ of the total transform $S^\sharp$ of $S$,  that is not contained in a dead branch of $S^\sharp$. Let us also consider a smooth retraction defined on a tubular neighborhood $\Omega_D$ of $D$,
\[ \rho_D : \Omega_D\to D, \]
 that is a locally trivial disc fibration, holomorphic on a neighborhood of each singular point $s_1,\ldots,s_v$ of $\F^\sharp$ in $D$ and such that the discs $\rho_D^{-1}(s_j)$ are in $S^\sharp$.  Up to  permutation, we suppose that $s_1,\ldots,s_r$, $1\leq r\leq v$, are the singular points $s_j$
 which do not belong to any dead branch of $S^\sharp$. We choose small enough disjoint closed discs
 \[
 D_{s_j}\subset D\,,\quad j=1,\ldots,v\,,
 \]
  centered at $s_j$  so that $\rho_D$ is holomorphic and trivial along $D_{s_j}$. We also suppose the Milnor tube $\mb T$ of $S$  small enough so that the real hypersurfaces $\rho_D^{-1}(\partial D_{s_j})$ intersect transversally the boundary of $E_\F^{-1}(\mb T)$ along a $2$-torus contained in $E_\F^{-1}(\delta\mb T)$. Then we consider 
the closure $\mc B_D$ of the connected component of the complementary in $E_{\F}^{-1}(\mb T)$ of $\cup_{j=1}^r\rho_D^{-1}(D_{s_j})$ that contains the complementary of $ \bigcup_{j=1}^r D_{s_j}$ in $D$, i.e.
\begin{equation*}
K_D:=D\setminus \bigcup_{j=1}^r \inte{D}_{s_j}\;\subset \; \mc B_D\;\subset\;E_{\F}^{-1}(\mb T)\setminus \bigcup_{j=1}^r\rho_D^{-1}(\inte{D}_{s_j})\,.
\end{equation*}
We fix  a connected component $\tilde{\mc B}_D^\xi$ of $\tilde{\mc B}_D:=q^{-1}(E_\F({\mc B}_D)\setminus S)$ and
we will denote by 
\[ 
\Gamma_{\xi}\subset \Gamma 
 \]
 the subgroup of deck transformations of the covering $q$ leaving invariant $\tilde{\mc B}^\xi_D$.
Using the particular presentation of the fundamental group of $\mc B_D\setminus S^\sharp_\F$ given in
 \cite[Proposition 3.3]{MM2} we can prove that
\begin{itemize}\it
\item $\mc B_D\setminus S^\sharp$ is incompressible in $E_\F^{-1}(\mb T)\setminus S^\sharp$, consequently 
the restriction of $q$ to $\tilde{\mc B}_D^\xi$ is the universal covering of $E_\F(\mc B_D\setminus S^\sharp)$ and
$\Gamma_\xi$ can be canonically identified with the deck transformation group of this covering;
\item $\mc B_D\cap S^\sharp$ contains all the dead branches of $S^\sharp$ meeting $D$;
\item the restriction of $\rho_D$ to $(\mc B_D\setminus S^\sharp)\cap \Omega_D$ extends to a smooth surjective map 
\begin{equation}\label{fibrSeifertBD}
\pi_D: \mc B_D\setminus S^\sharp\to K_D
\end{equation}
 that is a  fibration in punctured discs, locally trivial, except when $r<v$  at the points $s_{r+1},\ldots,s_v$;
in this case the closure of $\pi_D^{-1}(s_j)$ is a closed disc transverse to the end component of the dead branch of $S^\sharp$ containing $s_j$, $j=r+1,\ldots,v$;
 \item   there is a retraction by deformation along the fibers of $\pi_D$ of $\mc B_D\setminus S^\sharp$ to  $\delta\mc B_D:=\mc B_D\cap E_\F^{-1}(\delta \mb T)$;
 \item the restriction of $\pi_D$ to $\delta\mc B_D$ is a Seifert fibration whose exceptional fibers are the  transverse intersections  $\delta\mc B_D\cap\pi_D^{-1}(s_j)$,   $j= r+1,\ldots v$, when $r<v$. 
\end{itemize}

 Let us now fix a connected regular curve $\Sigma\subset \mc B_D$ that meets
$D\subset S^\sharp$ at a regular point $\sigma$  and is transversal to $\F^\sharp$. 
We denote by $\tilde\Sigma^j$, $j\in\mc J$, the connected components of $q^{-1}(E_\F(\Sigma))\cap\tilde{\mc B}_D^\xi$.
 We  assume that  the   fundamental system $\mc U=(U_n)_{n\in\N}$  satisfies Assertion~(\ref{interDelta}) in Theorem~ \ref{IncomprThm} taking   $C=E_\F(\Sigma)$. Then each
 $\tilde\Sigma^j_n:=\tilde\Sigma^j\cap \tilde U_n$ 
   is an embedded disc  and $\mr{card}(L\cap \tilde\Sigma_n^j)\le 1$  for each   leaf $L$ of $\tilde\F_{|\tilde{\mc B}_D^\xi}$. 
\begin{obs}\label{monotau}
The holomorphic maps 
$\tau_n^j :\tilde\Sigma_n^j\hookrightarrow Q_{n}$
sending  $p\in \tilde\Sigma_n^j$ into the leaf of $\tilde\F_{|\tilde U_n}$ passing through $p$, define the tautological morphism 
\[\tau_{\tilde\Sigma^j}:(\tilde\Sigma^j,\infty)\to\mc Q\] introduced in (\ref{tautologicalMorphism}).
As $\tau^j_n$ are injective, $\tau_{\tilde\Sigma^j}$ 
 is a monomorphism in the category
$\underleftarrow{\C\text{-Man}}$.
\end{obs}

\begin{prop}[{\cite[Proposition 4.2.1]{MM3}}]\label{mij} For any $i,j\in\mc J$ and large enough $n\in\mb N$, the images of the maps $\tau_n^i$ and $\tau_n^j$ have a non-empty intersection $W_n^{ij}$, and 
for $p\gg n\gg 1$ we have the following  inclusions \[ 
\tilde\Sigma_n^{k}\cap q^{-1}(U_p)\subset  (\tau_n^{k})^{-1}(W_n^{ij})\,,\qquad k=i, j\,.
 \]
The family of holomorphic maps   
 \[ 
  m^{ij}_n=(\tau^i_n)^{-1}\circ \tau^j_n:(\tau_n^{j}){}^{-1}(W_n^{ij})
  \to
  (\tau_n^{i}){}^{\,-1}(W_n^{ij})\,,\quad n\in\mb N\,.
     \]
defines  a morphism  in the category $\underleftarrow{\C\text{-}\mr{Man}}$ 
\[ m^{ij}\,\in\,\mr{Hom}_{\underleftarrow{\C\text{-}\mr{Man}}}\big((\tilde\Sigma^j,\infty),\,(\tilde\Sigma^i,\infty)\big)
 \]
 such that  $\tau_{\tilde\Sigma^i}\circ m^{ij}=\tau_{\tilde\Sigma^j}$
 and for any $i,j,k\in\mc J$ and any deck automorphism $\varphi\in\Gamma_\xi$ we have
\begin{equation}\label{relatChange_Charts}
m^{ij}\circ m^{jk}=m^{ik}\,,\quad
m^{ij}=(m^{ji})^{-1}\,,\quad 
\varphi\circ m^{ij}\circ\varphi^{-1}=m^{\varphi(i)\varphi(j)}\,,
\end{equation}
where  $\varphi(k)\in\mc J$ is the index such that  
$\varphi(\tilde \Sigma^{k})=\tilde\Sigma^{\varphi(k)}$.
\end{prop}

\noindent 
We easily deduce from  relations  (\ref{relatChange_Charts}) that the   map
\begin{equation*}\label{ext-Holonomy}
\mc H^i :  \Gamma_{\xi}\longrightarrow 
\underleftarrow{\mr{Aut}}(\tilde\Sigma^i,\infty)
\,,\quad \gamma \mapsto \mc H^i(\gamma):= m^{i\gamma(i)}\circ \gamma\,,
\end{equation*}
is well defined  and is a group morphism.  
We call it the \emph{extended holonomy of $\F$  along ${\mc B}_D$ over  $\tilde\Sigma^i$} (see Figure~\ref{fig5}).

\begin{center}
\begin{figure}[ht]
\begin{tikzpicture}
\draw[thick] (1,4) to (1,0);
\draw[thick] (9,4) to (9,0);
\draw[thick] (0,4) to [out=-60,in=150] (1,3) to [out= -30,in=170] (9,1) to  [out=-10,in=175] (10,0.9);
\node at (1,-0.5) {$\tilde\Sigma^i$};
\node at (9,-0.5) {$\gamma(\tilde\Sigma^i)$};
\draw[thick,->] (2.7,-0.5) to (6.8,-0.5);
\node at (4.5,-0.1) {$\gamma$};
\node at (4.5,2) {$L$};
\fill (1,3) circle [radius=2pt];
\fill  (1,1.25) circle [radius=2pt];
\fill  (9,1) circle [radius=2pt];
\node at (0.75,1.25) {$z$};
\node at (9.45,1.25) {$\gamma(z)$};
\node at (0,2.7) {$(\mc H^i(\gamma))(z)$};
\end{tikzpicture}
\caption{The extended holonomy transformation $\mc H^i(\gamma)$: the deck transformation $\gamma$ is represented as an horizontal translation and $L$ is the leaf of $\tilde\F=q^{-1}\F$ containing $\gamma(z)$.}\label{fig5}
\end{figure}
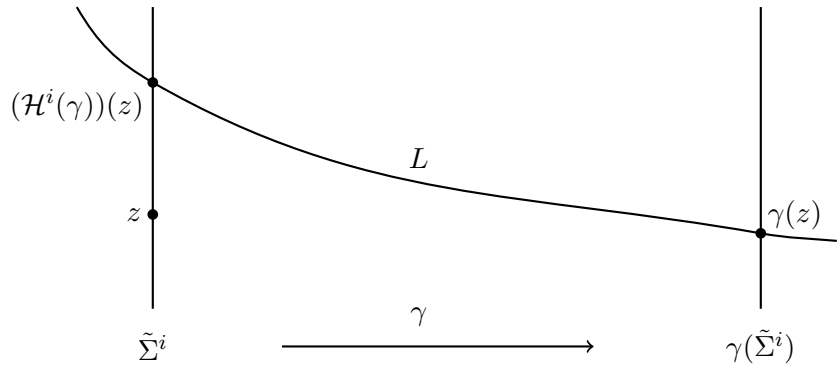
\end{center}

\subsection{Monodromy}\label{subsecmonodromy} Keeping the notations of previous subsection, let us consider now the group 
\begin{equation*}\label{autQ}
\underleftarrow{\mr{Aut}}(\mc Q)\subset\underleftarrow{\mr{Hom}}(\mc Q,\mc Q)
\end{equation*}
of invertible morphisms of $\underleftarrow{\mr{Hom}}(\mc Q,\mc Q)$ in the category $\underleftarrow{\mr{Top}}$,  see (\ref{Hom}). 
Any deck transformation
$\gamma\in \Gamma$ keeps invariant the foliation $\tilde\F$,  it induces an automorphism of each leaf space $Q_{\tilde U_n}$, $n\in\mb N$, and thus  defines an automorphism  $\gamma_\ast\in \underleftarrow{\mr{Aut}}(\mc Q)$.
\begin{defin}\label{definmonodromy}
The map 
\[
\mf m^\F:\Gamma\to
\underleftarrow{\mr{Aut}}(\mc Q)\,,\quad
\gamma\mapsto\gamma_\ast\,,
\] 
is a group morphism which we call  \emph{monodromy representation} of $\F$.
\end{defin}
\noindent In fact the values of  $\mf m^\F$ are in  the group of holomorphic automorphisms of $\mc Q$.\\

The commutative diagram below shows  that   $\mc H^i(\gamma)$ can be interpreted as  the ``expression of $\mf m^\F(\gamma)$ in the chart $\tilde\Sigma^i$''
\begin{equation}
\label{mH}
\xymatrix{
\mc Q \ar[rr]^{\mf{m}^\F(\gamma)}  & & \mc Q\\
(\tilde\Sigma^i,\infty)\ar@{^{(}->}[u]_{\tau_{{}_{\tilde\Sigma^{i}}}} \ar[r]^<<<<{\gamma}
\ar@/_2pc/[rr]^{\mc H^i(\gamma)}
&(\gamma(\tilde\Sigma^{i}),\infty)\ar@{^{(}->}[ur]^{\tau_{{}_{\gamma(\tilde\Sigma^i)}}}\ar[r]^<<<<<{m^{i\gamma(i)}} 
&(\tilde\Sigma^i,\infty)\ar@{^{(}->}[u]_{\tau_{{}_{\tilde\Sigma^i}}}
}
\end{equation}

Now for another foliation $\F'$, let us fix $\mb T'$, $S'$, $\Delta'$, $\check\Delta'$, $\mc U':=(U'_n)_n$, $q':\tilde{\mb T}'\to \mb T'\setminus S'$, $D'$, $\tilde{\mc B}_{D'}^{'\xi'}$, $\Sigma'$, $\sigma'$ and let us denote by $\tilde\F'$,  $\Gamma'$, $\mc Q'$,  $(\tilde\Sigma'{}_n^{j})_{j\in\mc J'}$, $\Gamma'_{\xi'}$, the analogous elements for $\F'$ that we have considered for $\F$. We will also denote more simply by $\mf m$ and  $\mf m'$ the monodromy representations $\mf m^\F$ and $\mf m^{\F'}$ of $\F$ and $\F'$ and by $\mc H^i$ and $\mc H'{}^{i'}$ their extended holonomies along $\mc{B}_D$ and $\mc{B}'_{D'}$,  over $\tilde\Sigma^i$ and $\tilde\Sigma'{}^{i'}$.\\

A \emph{$\mc C^0$-conjugacy\footnote{In \cite{MM3} this corresponds to the notion of geometric conjugacy.} between the monodromies} $\mf m$ and $\mf m'$ is a pair $(\tilde g_\ast,h)$ formed by 
\begin{enumerate}[(i)]
\item\label{conj1} a group morphism $\tilde g_\ast :\Gamma\to\Gamma'$, $\gamma\mapsto \tilde g\circ\gamma\circ \tilde g^{-1}$, defined by a lifting $\tilde g:\tilde{\mb T}\to\tilde{\mb T}'$ of a homeomorphism $g:(\mb T,S)\to(\mb T',S')$,
\item an isomorphism $h:\mc Q\to\mc Q'$ in the category $\underleftarrow{\mr{Top}}$
\end{enumerate}
such that the following diagram is commutative:
\begin{equation}\label{conjmonodr}
\xymatrix{\Gamma\ar[r]^{\mf m\hphantom{aaa}}\ar[d]_{\tilde g_*} & \underleftarrow{\mr{Aut}}(\mc Q)\ar[d]^{h_*}\\
\Gamma' \ar[r]^{\mf m'\hphantom{aaa}} & \underleftarrow{\mr{Aut}}(\mc Q')}
\end{equation}
where $\tilde g_*(\gamma):= \tilde g\circ\gamma\circ\tilde g^{-1}$  and $h_\ast(\psi):= h\circ \psi\circ h^{-1}$. 
According to Theorem~\ref{curves}, without modifying the first element $\tilde g_*$, we always can choose the homeomorphism $g$ defining it so that its lifting through the reduction maps  $E_\F$ and $E_{\F'}$   extends homeomorphically to the exceptional divisor $\mc E_\F$ and this extension is holomorphic at the  singular points of $S^\sharp:=E_\F^{-1}(S)$; we then say that $g$ is  an \emph{excellent conjugacy between $S$ and $S'$}.

\begin{defin}\label{conjmon}
We will say that a conjugacy $(\tilde g_*,h)$ between $\mf m$ and $\mf m'$  is \emph{excellent} if \emph{$h$ is excellent} in the following sense:  there is a union $K$ of nodal and dicritical separators of $\F$ such that $h=\limproj_{n}\limind_{m}h_{m,n}$ and, for large enough $m, n$, the $\mc C^0$-maps   $h_{mn}:Q_{m}\to Q'_{n}$ are  holomorphic at each point $L\in Q_{m}$ that is a leaf not intersecting $K$.
\end{defin}
Similary a \emph{holomorphic conjugacy between the extended holonomies} $\mc H^{i}$ and $\mc H'{}^{i'}$ along $\mc{B}_D$ and $\mc{B}'_{D'}$ over  $\tilde\Sigma^i$ and $\tilde\Sigma'{}^{i'}$ is a pair $(\tilde g_*,\tilde \varphi)$ formed by a group morphism $\tilde g_\ast:\Gamma_\xi\to\Gamma'_{\xi'}$  such that:
\begin{enumerate}
\item $g$ is an excellent conjugacy   between $S$ and $S'$ whose lifting $g^\sharp$ through $E_\F$ and $E_{\F'}$ sends $D$ onto $D'$,
\item $\tilde g$ is a lifting  of $g$ through $q$ and $q'$ such that for large enough  $n\gg p$,  $\tilde g(\tilde\Sigma_n^{i})$ is contained in $\tilde\Sigma'_p{}^{i'}$  and consequently the induced by $g$ group morphism $\tilde g_*:\Gamma\to\Gamma'$ sends $\Gamma_\xi$ onto $\Gamma'_{\xi'}$,
\item  $\tilde\varphi:(\tilde \Sigma^i,\infty)\to (\tilde\Sigma'{}^{i'},\infty)$ is  a lifting of a biholomorphism germ $\varphi:(\Sigma,\sigma)\to (\Sigma',\sigma')$ and for any $\gamma\in\Gamma_\xi$ the following diagram is commutative:
\begin{equation}\label{conjextholo}
\xymatrix{
(\tilde \Sigma^i,\infty)\ar[r]^{\tilde\varphi}\ar[d]^{\mc H^i(\gamma)} & (\tilde\Sigma'{}^{i'},\infty)\ar[d]^{\mc H'{}^{i'}(\tilde g_*(\gamma))}\phantom{\,.}
\\
(\tilde \Sigma^i,\infty)\ar[r]^{\tilde\varphi}  & (\tilde\Sigma'{}^{i'},\infty)
\,.}
\end{equation}\end{enumerate}

With these notations we have:
\begin{prop}\label{conjMonExtHol}
Let us consider  $(\tilde g_\ast,h)$  a conjugacy between the monodromies $\mf m$ and $\mf m'$ of $\F$ and $\F'$ 
such that $g$ is excellent, $g^\sharp(D)=D'$ and  $\tilde{g}(\tilde\Sigma_n^{i})\subset\tilde{\Sigma}'_n{}^{i'}$. Let $\tilde\varphi : (\tilde\Sigma^i,\infty)\to(\tilde\Sigma'{}^{i'},\infty)$
be a lifting 
 of a biholomorphism germ $\varphi:(\Sigma,\sigma)\to (\Sigma',\sigma')$, that is \emph{compatible with $h$} in the sense that the following diagram 
\begin{equation}\label{relcompat}
\xymatrix{
(\tilde\Sigma^i,\infty)\ar[r]^{\tilde\varphi}\ar[d]^{\tau} & (\tilde\Sigma'{}^{i'},\infty)\ar[d]^{\tau'}\\
\mc Q\ar[r]^{h} & \mc Q'
}
\end{equation}
is commutative, with $\tau$, resp. $\tau'$, denoting the tautological morphisms $\tau_{{}_{\tilde\Sigma^{i}}}$, resp.  $\tau_{{}_{\tilde\Sigma'{}^{i'}}}$ defined in (\ref{tautologicalMorphism}). 
Then $(\tilde g_\ast, \tilde\varphi)$ is a holomorphic conjugacy between the extended holonomies of $\F$ and $\F'$ along $\mc B_D$ and $\mc B'_{D'}$ over $\tilde\Sigma^i$ and $\tilde\Sigma'{}^{i'}$. 
\end{prop}

\begin{proof} For any connected component $\tilde\Sigma^i$ of $\tilde \Sigma$, let us consider the following diagram  where $\tilde \Sigma'{}^{{i'}}$ is the image of $\tilde\Sigma^i$ by $\tilde\varphi$, $\gamma\in\Gamma_\xi$ and $\gamma':=\tilde\varphi_\ast(\gamma)\in\Gamma'_{\xi'}$.  
\begin{equation}\label{cubicdiagram}
\xymatrix
{ 
& \mc Q \ar@{->}[rr]^h\ar@{<-^)}'[d][dd]_>>>>>>>{\tau}
&&\mc Q' \ar@{<-^)}[dd]_>>>>>>>{\tau'}\\ 
\mc Q \ar@{->}[ur]^{\mf m(\gamma)}
\ar@{->}[rr]^>>>>>>>>>>>h\ar@{<-^)}[dd]_>>>>>>>>>>>{\tau}
&&\mc Q'\ar@{->}[ur]^{\mf m'(\gamma')}\ar@{<-^)}[dd]_>>>>>>>>{\tau'}\\
    &(\tilde\Sigma^i,\infty) \ar@{->}
    '[r]
   [rr]^<<<<<{\tilde\varphi}
    && (\tilde\Sigma^{'{i'}},\infty)\\
(\tilde\Sigma^i,\infty) \ar@{->}[rr]^{
\tilde{\varphi}
}\ar@{->}[ur]^{\mc H^i({\gamma})} && (\tilde\Sigma^{'{i'}},\infty) \ar@{->}[ur]_<<<<<<<<<<<<{\mc H^{'{i'}}({\gamma'})}
}
\end{equation}
According to the $\tilde\varphi$-compatibility relation (\ref{relcompat}) and the
extended holonomy relations (\ref{mH}), all lateral squares of this cubic diagram are commutative.
We need to prove the commutativity of the bottom square (\ref{conjextholo}). 
For this, first notice that we have the following equivalence:
\begin{equation}\label{equ}
\tilde\varphi\circ\mc H^i(\gamma)=\mc H'{}^{{i'}}\circ\tilde\varphi\;\;
\Longleftrightarrow \;\;
\tau'\circ \tilde\varphi\circ\mc H^i(\gamma)\;=\tau'\circ\mc H'{}^{{i'}}\circ\tilde\varphi
\end{equation}
 because $\tau'$ is a monomorphism, see Remark \ref{monotau}. But the commutativity of the lateral squares gives: 
 \[\tau'\circ
 \tilde\varphi\circ\mc H^i(\gamma)=h\circ\mf m(\gamma)\circ\tau
 \quad
 \hbox{ and }\quad \tau'\circ\mc H'{}^{{i'}}\circ\tilde\varphi=\mf m'(\gamma')\circ h\circ\tau\,.
 \]
 Therefore the two equivalent equalities in (\ref{equ}) are also  equivalent to 
 \[h\circ\mf m(\gamma)\circ\tau=\mf m'(\gamma')\circ h\circ\tau
\,. \]
But that equality is satisfied because the top square in  diagram (\ref{cubicdiagram}) is commutative thanks to
 the conjugacy of monodromies relation (\ref{conjmonodr}). 
\end{proof}

With the same notations we also have:
\begin{prop}\label{conjHolproj} Let  $(\tilde g_\ast, \tilde\varphi)$ be a conjugacy between the extended holonomies $\mc H^{i}$ and $\mc H^{'^{i'}}$ of $\F$ and $\F'$ along $\mc B_D$ and $\mc B_{D'}$  over $\tilde\Sigma^{i}$ an $\Sigma^{'{i'}}$ respectively. Then the holonomy representations $H_{D}^{\F}$ and $H_{D'}^{\F'}$ of the foliations  $\F^\sharp$ and $\F'{}^\sharp$ along $D$ and $D'$ are holomorphically conjugated, more precisely the following diagram is commutative:
\begin{equation}\label{conjhol}
\xymatrix{\pi_1(D^*,\sigma)\ar[r]^{H_D^\F}\ar[d]_{ g^{\sharp}_*}&\mr{Diff}(\Sigma,\sigma)\ar[d]^{\varphi_*}\\ \pi_1(D'^*,\sigma')\ar[r]^{H_{D'}^{\F'}} &\mr{Diff}(\Sigma',\sigma')}
\end{equation}
with $g^\sharp_\ast$ induced by the restriction of $g^\sharp$ to $D^*:=D\setminus \mr{Sing}(\F^\sharp)$ onto $D'{}^*:=D'\setminus \mr{Sing}(\F'{}^\sharp)$ and $\varphi_\ast(f):=\varphi\circ f\circ \varphi^{-1}$.
\end{prop}

\begin{proof}[Idea of the proof.]
We can assume that $\Sigma$ is the (regular) fiber over $\sigma\in D\setminus \mr{Sing}(\F^\sharp)$  of the fibration $\pi_D:\mc B_D\setminus S^\sharp\to K_D$ previously  described, see (\ref{fibrSeifertBD}). For any
$\gamma\in\Gamma_\xi$, thanks to  Proposition~\ref{mij} there is a path $\alpha$ in a leaf 
of $\tilde \F_{|\tilde{\mc B}_D^\xi}$ with endpoints in $\tilde\Sigma^i$ and $\gamma(\tilde\Sigma^i)$; when $r<v$ we can choose $\alpha$ generic enough so that the image of $E_\F^{-1}\circ q\circ\alpha$ meets no  exceptional fiber $\pi_D^{-1}(s_j)$, $j=r+1,\ldots,v$. 
The particular presentation of the fundamental group of $\mc B_D\setminus S^\sharp$ given in
\cite[\S 4.1]{MM3}, allows to check that
\begin{itemize}\it
\item the class $\dot\gamma$ of the loop  $\pi_D\circ E_\F^{-1}\circ q\circ\alpha$  in the quotient group $\,\pi_1(D^\ast, \sigma)/\ker(H_D^\F)\,$ does not depend on the choice of the path $\alpha$ but only on $\gamma$;
\item the map $\Gamma_\xi\to \,\pi_1(D^\ast, \sigma)/\ker(H_D^\F)$, $\,\gamma\mapsto\dot\gamma$, is surjective and $H_D^\F(\dot\gamma)$ is well defined.
\end{itemize}
For  $\gamma'=\tilde g_*(\gamma)\in\Gamma'_{\xi'}$  we similarly define $\dot\gamma'\in\pi_1(D'^*,\sigma')/\ker(H^{\F'}_{D'})$.
Let  $q_\infty$, resp.  $q'_\infty$, be the morphism in the category $\underleftarrow{\C\text{-Man}}$ induced by the restrictions of the covering map $q$, resp.  $q'$, to $q^{-1}(U_n)$, resp.  $q'{}^{-1}(U'_n)$, $n\in \mb N$. It follows from the construction of $\dot\gamma$ and $\dot\gamma'$ that all lateral squares of the following cubic diagram are commutative:
\[\xymatrix
{ 
& (\tilde\Sigma^i,\infty)  \ar@{->}[rr]^{\tilde\varphi}\ar@{->>}'[d][dd]^<<<<{q_\infty}
&&(\tilde\Sigma^{'{i'}},\infty)\ar@{->>}[dd]^{q'_\infty}\\ (\tilde\Sigma^i,\infty)\ar@{->}[ur]^{\mc H^i({\gamma})}
\ar@{->}[rr]^>>>>>>>>>>>{\tilde\varphi}\ar@{->>}[dd]^{q_\infty}
&&(\tilde\Sigma^{'{i'}},\infty)\ar@{->}[ur]^{\mc H^{'{i'}}({\gamma'})}\ar@{->>}[dd]^<<<<<<<{q_\infty'}\\
    &(\Sigma,\sigma) \ar@{->}'[r][rr]^<<<<<<<{\varphi}
    && (\Sigma',\sigma')\\
(\Sigma,\sigma) \ar@{->}[rr]^{
{\varphi}
}\ar@{->}[ur]^{H^\F_D(\dot{\gamma})} && (\Sigma',\sigma') \ar@{->}[ur]^<<<<<<{H^{\F'}_{D'}(\dot{\gamma}')}
}\]
The top square is commutative by hypothesis. We deduce that the bottom square is also commutative by proceeding as in the proof of Proposition~\ref{conjMonExtHol} after noting that $q_\infty$ and $q'_\infty$ are epimorphisms in the category $\underleftarrow{\C\text{-Man}}$.
\end{proof}

Finally, as a consequence of  Propositions~\ref{conjMonExtHol} and~\ref{conjHolproj} we obtain:
\begin{teo}\label{monhol}\cite[Theorem 4.3.1]{MM3} Let $(\tilde g_\ast,h)$ be a $\mc C^\mr{ex}$-conjugacy between the monodromies of $\F$ and $\F'$, let $\Sigma$ and $\Sigma':=g^\sharp(\Sigma)$ be transverse sections to invariant irreducible components $D\subset \mc E_\F$ and $D':=g^\sharp (D)\subset\mc E_{\F'}$. Let $\tilde\varphi:\tilde\Sigma\iso\tilde\Sigma'$ be a lifting of a biholomorphism $\varphi:\Sigma\to\Sigma'$, that is compatible to~$h$. Then the holonomies $H^\F_D$ and $H^{\F'}_{D'}$ of $\F^\sharp$ and $\F'{}^\sharp$ along $D$ and $D'$ are holomorphically conjugated by $\varphi$, i.e.  diagram (\ref{conjhol}) is commutative.
\end{teo}

\subsection{Classification Theorem}\label{TCT}  We  consider  two germs of foliations $\F$ and $\F'$ at $0\in\C^2$ that are generalized curves. We keep the previous notations: in a Milnor tube $\mb T$, resp. $\mb T'$, of an  appropriate curve $S$ for $\F$, resp. $S'$ for $\F'$; we again denote by  $q:\tilde {\mb T}\to \mb T\setminus S$ and $q':\tilde{\mb T}'\to \mb T'\setminus S'$ the universal covering maps and by $\Gamma$ and $\Gamma'$ the  deck transformation groups of these coverings;  $\mc Q$, resp.  $\mc Q'$, denotes  the leaf spaces inverse system for $\F$, resp.  $\F'$, defined as in (\ref{inversesystem}) by a fundamental system $\mc U=(U_n)_n$, resp. $\mc U'=(U'_n)_n$, of open neighborhoods of $S$, resp. of $S'$,  that satisfies the properties stated in Theorem \ref{IncomprThm}.

\begin{defin}
Given two subsets $V\subset\mb T$ and $V'\subset \mb T'$ whose closures meet $S$ and $S'$ respectively, we call \emph{realization on the germs $(V,S)$  and $(V',S')$ of a  conjugacy $(\tilde g_\ast, h)$} between the monodromies $\mf m$ and $\mf m'$ of $\F$ and $\F'$, the data $(\psi,\tilde \psi)$ of a germ of homeomorphism $\psi: (V,S)\to(V',S')$ and  a $\underleftarrow{\mr{Top}}$-morphism 
 \begin{equation*}\label{liftpsi}
 \tilde\psi :(\tilde V,\infty)\longrightarrow (\tilde V',\infty)\,,\qquad
 (\tilde V,\infty):=(q^{-1}(V\cap U_n))_n\,,\quad
(\tilde V',\infty):=(q'{}^{-1}(V'\cap U'_n))_n\,,
 \end{equation*}
that lifts $\psi$ i.e. $q'\circ\tilde\psi=\psi\circ q$, and satisfy the commutativity of the following diagrams
\begin{equation*}
\xymatrix{
(\tilde V,\infty)\ar[r]^{\tilde\psi}\ar[d]^{\tau} & (\tilde V',\infty)\ar[d]^{\tau'}\\
\mc Q\ar[r]^{h} & \mc Q'
}
\qquad
\qquad
\xymatrix{
\Gamma_{\tilde V}\ar[r]^{\tilde\psi_*} &\Gamma'_{\tilde V'}\\
\Gamma\ar[r]^{\tilde g^*}\ar[u]^\iota & \Gamma'\ar[u]^{\iota'}\\
}
\end{equation*}
where $ \Gamma_{\tilde V}$ and  $ \Gamma'_{\tilde V'}$ are the deck transformation groups of the (not necessarily connected) coverings $q^{-1}(V)\to V$ and $q'{}^{-1}(V')\to V'$ induced by $q$ and $q'$.  
\end{defin}
When $\psi$ lifts through the reduction of singularities maps as a homeomorphism germ between the strict transforms of $V$ and $V'$:
\[ 
\psi^\sharp:(\hat V,S)\longrightarrow (\hat V'{},S^{\sharp})\,,\quad
E_{\F'}\circ\psi^\sharp=\psi\circ E_\F\,,
\qquad
\hat{V}':=\overline{E_\F^{-1}(V')\setminus S'^\sharp}\,,
 \]
we will also say that $(\psi^\sharp,\tilde\psi)$ is a \emph{realization of $(\tilde g_*,h)$ on the germs  $(\hat V,S)$ and $(\hat{V}',S')$}.

 \begin{obs}\label{conjreal}
 If $V$ contains an open set $U$ and $(\psi,\tilde\psi)$ realizes on $(V,S)$ and $(V',S')$ a conjugacy between the monodromies of $\F$ and $\F'$, then the restriction of $\psi$ to $U$ is a conjugacy between the restriction of $\F$ to $U$ and $\F'$ to $\psi(U)$.
 \end{obs}

\begin{defin}\label{CompTransCurve}
We call \emph{complete $\F$-transversal curve} in a Milnor ball $B$ for $S_\F$, any finite union $\Delta$ of disjoint connected embedded discs in $B\setminus \{0\}$ that are transversal to $\F$ and such that  any dynamical component of $S_\F^\sharp$ meets  $E_\F^{-1}(\Delta)$ at a single point (necessarily belonging to the strict transform of an isolated separatrix of $\F$). 
\end{defin}

Theorem~\ref{sepdyncomp} assures the existence of a complete $\F$-transversal curve by choosing a non-nodal isolated separatrix in each dynamical component of $S_\F^\sharp$ and a small embedded disc transverse to their image by $E_\F$.

\begin{teo}\cite[Classification Theorem]{MM2}\label{ClassThm} 
Let us fix a complete $\F$-transversal curve $\Delta$, resp. complete $\F'$-transversal curve $\Delta'$. Assume that there exists an excellent  $\mc C^0$-conjugacy  $(\tilde g_\ast, h)$ between the monodromy representations $\mf m$ and $\mf m'$ of $\F$ and $\F'$ and let 
\[ \tilde\varphi:\tilde\Delta\to \tilde\Delta' \,,\quad \tilde\Delta:=q^{-1}(\Delta)\,,\quad
 \tilde\Delta':=q'^{-1}(\Delta')\,,
 \]
be a lifting through  the covering maps $q$ and $q'$ of a germ of biholomorphism $\varphi:(\Delta,S_\F)\to (\Delta',S_{\F'})$. Let us assume that $(\varphi,\tilde\varphi)$ is a realization of the conjugacy $(\tilde g,h)$ on the germs $(\Delta,S_\F)$ and $(\Delta',S_{\F'})$ and that 
\begin{enumerate}
\item the image by $\varphi$ of any connected component of $\Delta$ meeting an isolated separatix $C$ of $\F$ is the connected component of $\Delta'$ meeting $g(C)$,
\item\label{eqCSisosep} the Camacho-Sad index of $\F^\sharp$ along the strict transform  $\mc C$ of any isolated separatrix $C$ of $\F$, is equal to that of $\F'{}^\sharp$ along the strict transform $\mc C'$ of $g(C)$, i.e.
 \[ 
\mr{CS}(\F^\sharp, \mc C,s)=\mr{CS}(\F'{}^{\sharp}, \mc C', s')\,,\qquad \{s\}:=\mc C\cap\mc E_\F\,,\quad \{s'\}:=\mc C'\cap\mc E_{\F'}\,.
 \]
\end{enumerate}
Then there exists a homeomorphism germ  $\Phi:(\C^2,0)\to(\C^2,0)$  that conjugates  $\F$ to $\F'$ and lifts through the reduction maps,  to a homeomorphism germ: 
\[ 
\Phi^\sharp:(M_\F, S_\F^\sharp)\to(M_{\F'},S_{\F'}^\sharp)\,,\quad E_{\F'}\circ\Phi^\sharp=\Phi\circ E_\F\,,\quad \Phi^\sharp(\F^\sharp)=\F'{}^{\sharp}\,.
 \]
Moreover $\Phi^\sharp$  is holomorphic at each non-nodal singular point of $\F^\sharp$ and is transversally holomorphic at each point of $S_{\F}^\sharp$ that is  regular for $\F^\sharp$ and does not belong to a dicritical component of $\mc E_\F$.
\end{teo}

\begin{defin}
A homeomorphism $\Phi$ that satisfies the conclusions of this theorem will be called \emph{excellent conjugacy} between $\F$ and $\F'$ and we will say that $\F$ and $\F'$ are $\mc C^{\mr{ex}}$-conjugated.
\end{defin}
\begin{proof}[Idea of the proof] With  some additional assumptions this theorem is proved in \cite{MM3} for foliations with  a single dynamical component.
A proof for the general statement  is given in \cite[Proof of Theorem 11.4, Steps (iii)-(v)]{MMS}. 
The conjugacy is independently constructed along each dynamical component of $S^\sharp_\F$,   their gluing on separators being made in a final step. 

We divide any dynamical component $\mc D$ of $S^\sharp_\F$  in \emph{elementary pieces} that are either (small pieces) connected components of the intersection of $S^\sharp_\F$ with  small  closed polydiscs centered at the singular points of $\F^\sharp$, or  (big pieces) the closure of the connected components of the complementary in $\mc D$ of  the union of all small pieces. Over each big piece $K$ we fix a locally trivial  smooth retraction-fibration  $\rho_K:\Omega_K\to K$ defined on an open neighborhood $\Omega_K$ of $K$  whose fibers are discs transversal to $\F^\sharp$. We can suppose that $\rho_K$ is holomorphic at each point of the boundary $\partial K$ of $K$. We can construct an exhaustion of~$\mc D$ by connected closed sets,
\begin{equation*}\label{filtr-elem-pieces}
Z_0\subsetneq Z_1\subsetneq \cdots \subsetneq Z_k=\mc D\,,\qquad Z_j=Z_{j-1}\cup K_j\,,
\end{equation*}
where $Z_0$ is the  big piece whose image by $E_\F$ is contained in the (unique) isolated separatrix in $\mc D$ meeting $\Delta$, and the closure $K_j$ of $Z_j\setminus Z_{j-1}$ is an elementary piece intersecting $Z_{j-1}$.  
Without loss of generality we suppose that $g$ is excellent and that $g^\sharp$ is holomorphic on a neighborhood of any small elementary piece. As we will see in \S\ref{62} 
the Camacho-Sad equalities on the isolated separatrices  in hypothesis (\ref{eqCSisosep}) induce the equality 
 \begin{equation}\label{CSequdem}
 \mr{CS}(\F^\sharp,D,s)=\mr{CS}(\F'^\sharp, g^\sharp(D), g^\sharp(s))
 \end{equation}
for any irreducible component $D$ of $S^\sharp_\F$ and any singular point $s\in D$ of $\F^\sharp$. Therefore $g^\sharp$ sends any dynamical component of $S^\sharp_\F$ onto a dynamical component of $S^\sharp_{\F'}$. 

We also have the  exhaustion 
\[ Z'_0\subsetneq Z'_1\subsetneq\ldots\subsetneq Z'_k=\mc D'\,,\qquad Z'_i:=g^\sharp(Z_i)\,,
 \]
 of the dynamical component $\mc D':=g^\sharp(\mc D)$ of $S^\sharp_{\F'}$. 
 Up to performing isotopies \cite[Theorem~2.9 and Definition~2.5]{MM2}  we can also suppose  
 that for any big piece $K\subset \mc D$ the map  $g^\sharp$ sends the fiber  of $\rho_K$ over any $p\in K$ to the fiber over $g^\sharp(p)$ of a   locally trivial smooth fibration $\rho_{g^\sharp(K)}'{}:\Omega'\to g^\sharp(K)$ that we have previously fixed, i.e.
 \begin{equation*}\label{comrhog}
 \rho_{g^\sharp(K)}'\circ g^\sharp=g^\sharp\circ \rho_{K}\,.
 \end{equation*}

We will successively construct  realizations $(\psi_{j}, \tilde\psi_j)$ of the monodromy conjugacy $(\tilde g_*,h)$, where $\psi_j$ is a homeomorphism  from a neighborhood of $Z_j$ to a neighborhood of $Z'_j$, that extends  $\psi_{j-1}$.  
Moreover, we will have 
\begin{equation}\label{comphij}
\rho'_{g^\sharp(K)}\circ\psi_j=\psi_j\circ\rho_K
\end{equation} 
on the boundary of each elementary big piece $K$ intersecting $Z_j$.

The proof ends by applying  Remark \ref{conjreal} to the  last homeomorphism $\Phi:=\psi_k$; this one conjugates $\F^\sharp$ to $\F'{}^\sharp$ on neighborhoods of  $\mc D$ and $\mc D'$ and the holomorphic transversality  property and the holomorphy property at the non-nodal singular points  will result from the construction.

Assume that we have already constructed $\psi_{j-1}$ and $\tilde\psi_{j-1}$. Let us fix a point $p_j\in K_j\cap Z_{j-1}$ and denote by $K$ the big piece containing $p_j$.
Consider the transversal discs
\[ 
\Sigma_j:=\rho_{K}^{-1}(p_j)\,,\qquad
\Sigma'_j:=\rho_{g^\sharp(K)}'^{-1}(p'_j)=g^\sharp(\Sigma_j)\,,\qquad
p'_j:= g^\sharp(p_j)\,,
\]
and the liftings by their corresponding universal covering maps
\[ 
 \tilde\Sigma_j:=q^{-1}(E_\F(\Sigma_j))\,,\quad \tilde{\Sigma}_j':=q'{}^{-1}(E_{\F'}(\Sigma'_j))\,.
 \]
The key of the induction process  consists in  applying the  following Extension Lemma taking for $\phi$ the restriction of $\psi_{j-1}$ to $\Sigma_{j}$ and for $\tilde \phi_j$ the restriction of $\tilde\psi_{j-1}$ to $\tilde\Sigma_{j}$.
\begin{lema}[{\cite[Lemma 8.3.2]{MM3}}]\label{extlemma}
Let 
\[ 
\phi: (\Sigma_j,p_j)\to(\Sigma_j',p_j')\,,\quad 
\tilde\phi : (\tilde{\Sigma}_j,\infty)\longrightarrow (\tilde{\Sigma}_j',\infty)\,,
 \]
be a biholomorphism germ and a 
$\underleftarrow{\C\text{-}\mr{Man}}$ 
isomorphism  defining a realization  of  $(\tilde g_\ast,h)$ on the germs $(\Sigma_j,S^\sharp_\F)$ and $(\Sigma'_j,S^\sharp_{\F'})$. If $\tilde\phi$ and $\tilde g$   induce the same canonical morphism 
\[ 
\tilde \phi_\star=\tilde g_\star : \pi_0(\tilde \Sigma_j)\to\pi_0(\tilde{\Sigma}_j')\,,
 \]
then there exists a homeomorphism $\Phi$ that extends $\phi$ along a neighborhood $\Omega_j$ of $K_j$ and a lifting $\tilde\Phi$ of $\Phi$ through $q\circ E_\F^{-1}$ and $q'\circ E_{\F'}^{-1}$, that extends $\tilde\phi$ such that 
\begin{enumerate}
\item $(\Phi,\tilde\Phi)$ is a realization of $(\tilde g_\ast,h)$ on the germs $(\Omega_j,S^\sharp_\F)$ and $(\Phi(\Omega_j),S^\sharp_{\F'})$,
\item\label{comrho} on a neighborhood of each connected component $\mathscr C$ of the boundary of $K_j$ the equality $\rho'\circ\Phi=\Phi\circ\rho$ is satisfied, where $\rho$ (resp. $\rho'$) is the disc fibration over the big piece containing $\mathscr C$ (resp. $\Phi(\mathscr C)$),
\item $\tilde\Phi_{|\tilde\Sigma_{j+1}}$ and $\tilde g$ induce the same canonical morphism from $\pi_0(\tilde\Sigma_{j+1})$ to $\pi_0(\tilde\Sigma'_{j+1})$.
\end{enumerate} 
 \end{lema}
Property (\ref{comrho}) above implies the uniqueness of the extension of $\phi$ on a neighborhood  $\Omega$ of $Z_{j-1}\cap K_j$.  Hence according to relation (\ref{comphij}) with the index $j-1$,  $\Phi$ and $\psi_{j-1}$ coincide on $\Omega$, and can be glued to define the homeomorphism $\psi_{j}$.

The proof of Lemma \ref{extlemma}, which is technical, is based on the fact that by Theorem~\ref{monhol} the conjugacy of monodromies realized by $\phi$
implies that $\phi$ also conjugates the holonomies of $\F^\sharp$ and $\F'{}^\sharp$ along the whole piece $K_{j }$ when $K_j$ is big and along the irreducible component of $K_j$ containing $p_j$ when $K_j$ is small. The extension $\Phi$ as a conjugacy of $\F^\sharp$ to $\F'{}^\sharp$ along $K_j$  is then given by a classical theorem on regular foliations when $K_j$ is big. When $K_j$ is small, using  Camacho-Sad indices equalities (\ref{CSequdem})  the extension is given by  a classical theorem on  saddle singularities \cite{MatMou}. 
\end{proof}

\section{Topological Invariance of Camacho-Sad indices} \label{CS0}

\subsection{Camacho-Sad index} We recall that  at a singular point $s$ of a foliation $\F$ on an invariant regular curve $C$, the Camacho-Sad index $\mr{CS}(\F,S,s)$ gives information about the winding of the leaves of $\F$ around $C$ in a neighborhood of $s$. More precisely, if $\F$ is defined by the $1$-form $vA(u,v)du+B(u,v)dv$ and $C=\{v=0\}\ni (0,0)=s$, then
\begin{equation}\label{CS-int}
\mr{CS}(\F,C,s):=\mr{Res}_{u=0}\left(-\frac{B(u,0)}{A(u,0)}\right)=\lim\limits_{z\to 0}\frac{1}{2i\pi}\int_{\Gamma_z}\frac{dv}{v}\,,
\end{equation}
where $\Gamma_z=(\gamma(t),\mu_z(t))$ is a path obtained by lifting into the leaves the loop $\gamma(t)=(e^{2i\pi t},0)$ inside $C$ with $\mu_z(0)=z$.

We will determine a large class of foliation germs  $\F$  for which the Camacho-Sad indices along the exceptional divisor $\mc E_\F$, at the singular points of $\F^\sharp$, \emph{are topological invariants} in the sense that 
if $\phi$ is a topological conjugacy from $\F$ to $\G$ then
 we have the equality
\begin{equation}\label{equalityCS}
\mr{CS}(\F^\sharp,D,s)=\mr{CS}(\G^\sharp,\A_\phi(D),\A_\phi(s))
\end{equation}
for any vertex $D$ adjacent to an edge $s$ in $\A_\F$. Here  the vertices and the edges of the dual trees $\A_\F$ and $\A_\G$ are  identified respectively to the irreducible components and the singular points of the total transforms 
$S_\F^\sharp$ and $S_\G^\sharp$ and $\A_\phi:\A_\F\to\A_\G$ is the graph morphism introduced in Proposition~\ref{InducedGraphMorphism}.

We will  further introduce two conditions (NFC) and (TR) on $\F$  making it possible to obtain the topological invariance of Camacho-Sad indices. But first we will examine the possible obstructions to this.

\subsection{Different types  of dynamical components}\label{62}
A straightforward induction process \cite[\S7.3]{MM3} based on the following Camacho-Sad Index Formula \cite[Appendix]{CS} for any invariant  irreducible component $D$ of $\mc E_\F$
\begin{equation}\label{index-formula}
\sum_{s\in\mr{Sing}(\F^\sharp)\cap D}\mr{CS}(\F^\sharp,D,s)=D\cdot D
\end{equation}
gives that 
\begin{itemize}
\item[($\star$)]
\it the equality (\ref{equalityCS}) holds for any $D$ and $s$ in $\A_\F$  if it is satisfied when $D$ is the strict transform of any isolated separatrix. 
\end{itemize}
For a fixed isolated separatrix $C$ we will examine different cases  according to the nature of the dynamical component $\mc D$ of $S_\F^\sharp$ containing the strict transform $\mc C$ of $C$. Let us notice that Theorem \ref{sepdyncomp}, giving the existence of a non-nodal isolated separatrix meeting the dynamical component $\mc D$, does not ensure that there is   an irreducible component in $\mc D$  with at least 3 singular points of $\F^\sharp$. However using  this theorem and Index Formula~(\ref{index-formula}) we can see (cf. \cite[Proof of Lemma~11.6]{MMS}) that there are only five possibilities for $\mc D$ and~$\mc C$:
 \begin{enumerate}
 \item\label{nodalsep} $\mc C$ is a nodal separatrix and $\mc D$ is reduced to $\mc C$;
 \item\label{smallcomp} 
The non-empty intersection $\mc D':=\mc D\cap\mc E_\F$ is a union $\mc D'=D_1\cup\cdots \cup D_\ell$ of invariant irreducible components of $\mc E_\F$,
 where $D_i$ meets $D_{i+1}$ at a (non-nodal) single point $s_i$ when $\ell>1$, and  $D_i\cap D_j$ is empty  if $|i-j|\neq 0, 1$; moreover $\mc C$ meets $D_1$ at a non-nodal singular point $s$ and one of the following configurations holds:
\begin{enumerate}[(a)]
\item\label{A} $\mc D=\mc C\cup \mc D'$, when $\ell> 1$ 
$\mr{Sing}(\F^\sharp)\cap \mc D$ is equal to 
$\{s,s_1,\ldots,s_{\ell-1}\}$, 
{or to $\{s\}$ when $\ell=1$}, cf. Figure~\ref{ci};
\item\label{B} $\mc D=\mc C\cup \mc D'$, 
$\mr{Sing}(\F^\sharp)\cap \mc D
=\{s,s_1,\ldots,,s_\ell\}$ and $s_\ell\in D_\ell\setminus\{s_{\ell-1}\}$ is a nodal singularity of $\F^\sharp$, cf. Figure~\ref{cii};
\item\label{freecomp} 
$\mc D=\mc C\cup \mc D'\cup\mc C'$, 
$\mc C'$ is an isolated non-nodal separatrix meeting  $D_\ell$ at singular points $s'$, and
$\mr{Sing}(\F^\sharp)\cap \mc D$ is equal to 
$\{s,s_1,\ldots,s_{\ell-1}, s'\}$ when $\ell>1 $ and to $\{s,s'\}$ when $\ell=1$,  cf. Figure~\ref{ciii}. 
\end{enumerate}
In all these cases any irreducible component of $\mc E_\F$ meeting $\mc D$ is either contained in $\mc D$ or dicritical.
 \item\label{bigcomp} There is an irreducible component of  $\mc D$ containing at least 3 singular points of $\F^\sharp$.
\end{enumerate}
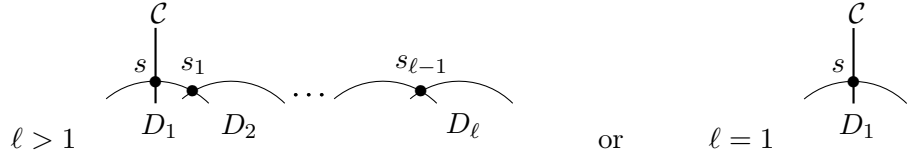
\begin{figure}[ht]
$\ell>1$\quad
\begin{tikzpicture}
\draw [thick] (6.3,0) to (6.3,1);
\node at (8.35,0.1) {$\cdots$};
\draw  (7,0) arc [radius=1, start angle=45, end angle= 130];
\draw  (8,0) arc [radius=1, start angle=45, end angle= 130];
\draw  (10,0) arc [radius=1, start angle=45, end angle= 130];
\draw  (11,0) arc [radius=1, start angle=45, end angle= 130];
\node at (6.1,.5) {$s$};
\node at (6.8,.5) {$s_1$};
\node at (9.8,.5) {$s_{\ell-1}$};
\node at (6.35,-0.3) {$D_1$};
\node at (6.35,1.2) {$\mc C$};
\node at (7.4,-0.3) {$D_2$};
\node at (10.35,-0.3) {$D_\ell$};
\draw (6.3,0.28) node {$\bullet$};
\node at (8.35,0.1) {$\cdots$};
\draw (6.785,0.16) node {$\bullet$};
\draw (9.79,0.16) node {$\bullet$};
\node at (8.35,0.1) {$\cdots$};
\end{tikzpicture}
\phantom{AAA} 
or 
\phantom{AAA} 
$\ell=1$\quad
\begin{tikzpicture}
\draw [thick] (6.3,0) to (6.3,1);
\draw  (7,0) arc [radius=1, start angle=45, end angle= 130];
\node at (6.1,.5) {$s$};
\node at (6.35,-0.3) {$D_1$};
\node at (6.35,1.2) {$\mc C$};
\draw (6.3,0.28) node {$\bullet$};
\end{tikzpicture}

\caption{Situation~(\ref{A}): there is no nodal singular point on $\mc D$ and  some dicritical component must meet $\mc D$.
}\label{ci}
\end{figure}

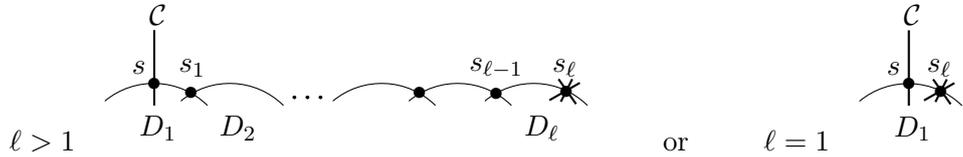
\begin{figure}[ht]
$\ell>1$\quad
\begin{tikzpicture}
\draw [thick] (6.3,0) to (6.3,1);
\node at (6.1,.5) {$s$};
\node at (6.8,.5) {$s_1$};
\node at (8.35,0.1) {$\cdots$};
\node at (6.35,1.2) {$\mc C$};
\node at (6.35,-0.3) {$D_1$};
\node at (7.4,-0.3) {$D_2$};
\node at (11.4,-0.3) {$D_\ell$};
\draw [thick] (11.65,0.05) to (11.8,.35);
\draw [thick] (11.5,0.07) to (11.9,.3);
\draw [thick] (11.85,0.02) to (11.62,.35);
\node at (10.8,.5) {$s_{\ell-1}$};
\node at (11.7,.5) {$s_\ell$};
\draw  (7,0) arc [radius=1, start angle=45, end angle= 130];
\draw  (8,0) arc [radius=1, start angle=45, end angle= 130];
\draw  (10,0) arc [radius=1, start angle=45, end angle= 130];
\draw  (11,0) arc [radius=1, start angle=45, end angle= 130];
\draw  (12,0) arc [radius=1, start angle=45, end angle= 130];
\draw (6.3,0.28) node {$\bullet$};
\node at (8.35,0.1) {$\cdots$};
\draw (6.785,0.16) node {$\bullet$};
\draw (9.79,0.16) node {$\bullet$};
\node at (8.35,0.1) {$\cdots$};
\node at (10.8,0.15) {$\bullet$};
\node at (11.72,0.18) {$\bullet$};
\end{tikzpicture}
\phantom{AAA}or\phantom{AAA}
$\ell=1$\quad
\begin{tikzpicture}
\draw [thick] (6.3,0) to (6.3,1);
\node at (6.1,.5) {$s$};
\node at (6.35,1.2) {$\mc C$};
\node at (6.35,-0.3) {$D_1$};
\draw [thick] (6.65,0.05) to (6.8,.35);
\draw [thick] (6.5,0.07) to (6.9,.3);
\draw [thick] (6.85,0.02) to (6.62,.35);
\node at (6.7,.5) {$s_\ell$};
\draw  (7,0) arc [radius=1, start angle=45, end angle= 130];
\draw (6.3,0.28) node {$\bullet$};
\node at (6.72,0.18) {$\bullet$};
\end{tikzpicture}
\caption{Situation~(\ref{B}): $s_\ell$ is the unique nodal singular point on $\mc D$ and some dicritical components may meet $\mc D$.}\label{cii}
\end{figure}

\begin{figure}[ht]
$\ell>1$
\begin{tikzpicture}
\draw [thick] (6.3,0) to (6.3,1);
\node at (6.1,.5) {$s$};
\node at (6.8,.5) {$s_1$};
\node at (8.35,0.1) {$\cdots$};
\node at (6.35,1.2) {$\mc C$};
\node at (6.35,-0.3) {$D_1$};
\node at (7.4,-0.3) {$D_2$};
\node at (11.4,-0.3) {$D_\ell$};
\draw [thick] (11.62,0.0) to (12,.9);
\node at (12.2,1.1) {$\mc C'$};
\node at (10.8,.5) {$s_{\ell-1}$};
\node at (11.6,.5) {$s'$};
\draw  (7,0) arc [radius=1, start angle=45, end angle= 130];
\draw  (8,0) arc [radius=1, start angle=45, end angle= 130];
\draw  (10,0) arc [radius=1, start angle=45, end angle= 130];
\draw  (11,0) arc [radius=1, start angle=45, end angle= 130];
\draw  (12,0) arc [radius=1, start angle=45, end angle= 130];
\draw (6.3,0.28) node {$\bullet$};
\node at (8.35,0.1) {$\cdots$};
\draw (6.785,0.16) node {$\bullet$};
\draw (9.79,0.16) node {$\bullet$};
\node at (8.35,0.1) {$\cdots$};
\node at (10.8,0.15) {$\bullet$};
\node at (11.72,0.18) {$\bullet$};
\end{tikzpicture}
\phantom{A}
or 
\phantom{A}
$\ell=1$\quad
\begin{tikzpicture}
\draw [thick] (6.3,0) to (6.3,1);
\node at (6.1,.5) {$s$};
\node at (6.95,.5) {$s'$};
\node at (6.35,1.2) {$\mc C$};
\node at (6.35,-0.3) {$D_1$};
\draw [thick] (6.8,0.0) to (7.5,.9);
\node at (7.7,1.1) {$\mc C'$};
\draw  (7,0) arc [radius=1, start angle=45, end angle= 130];
\draw (6.3,0.28) node {$\bullet$};
\draw (6.89,0.1) node {$\bullet$};
\end{tikzpicture}
\caption{Situation~(\ref{freecomp}): 
there is no nodal singular point on $\mc D$ and some dicritical component must meet $\mc D$.}\label{ciii}
\end{figure}
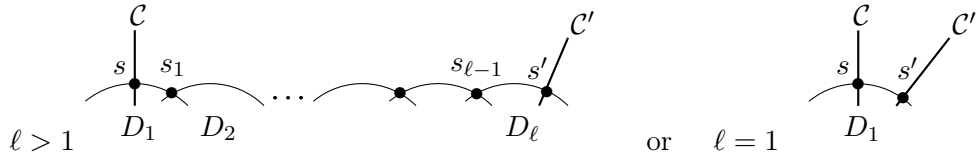

\noindent We say that the dynamical componant $\mc D$  is  \emph{small} in cases (\ref{nodalsep}) and (\ref{smallcomp}) and that  it is  \emph{big} in case (\ref{bigcomp}). 
We will now go through each case.

\subsection{Small dynamical components}\label{smalldyncomp}
We denote again by $\mc C$ the strict transform of an isolated separatrix of $\F$ and by $\phi:\F\to\G$ a topological conjugacy between $\F$ and $\G$.

In the case~(\ref{nodalsep})   the equality 
\[\mr{CS}(\F^\sharp,\mc C,s)=\mr{CS}(\G^\sharp,\A_\phi(\mc C),\A_\phi(s))\] is proven by R. Rosas under  weak hypothesis  \cite[Proposition~13]{Rosas}, see also \cite[Theorem~1.12]{MM4}. 

Index Formula~(\ref{index-formula}) implies that:
\begin{itemize}
\item
 in case (\ref{A}) the Camacho-Sad index $\mr{CS}(\F^\sharp,\mc C,s)$ is a rational number that  can be expressed as a continued fraction in terms of the self-intersections $D_i\cdot D_i$ of the irreducible components of $\mc D\cap \mc E_\F$; as these ones are  invariant by  topological conjugacies, cf. \cite{Z} or \cite{MM2}, the   equality of Camacho-Sad indices (\ref{equalityCS}) is  satisfied;
\item  in case (\ref{B}) we have $\mr{CS}(\F^\sharp,D_\ell,s_\ell)\in\R\setminus\mb Q$ and consequently  $\mr{CS}(\F^\sharp,\mc C,s)\in\R\setminus\mb Q$; in fact  $s$ must be a linearizable saddle of $\F^\sharp$ (recall that $s$ is not a node). Then the proof of~\cite[Theorem~1.12]{MM4} gives the equality~(\ref{equalityCS}).
\end{itemize}

In case (\ref{freecomp})   any reduced singularity may appear at the  point~$s$ and for this reason we will say that $\mc D$ is a  \emph{free (dynamical) component}.
All the free components of \emph{Poincar\'e type}, i.e. with $\mr{CS}(\F^\sharp,\mc C,s)\in \C\setminus \R$,  are topologically conjugated.
Indeed, this claim uses the following fact:
\begin{lema}\label{Poincare}
 \noindent Let $(\lambda_1,\lambda_2)$ and $(\mu_1,\mu_2)$ be two pairs of $\R$-independent complex numbers and 
 $g:\C\to\C$ be a $\R$-linear map such that $g( i\lambda_j):= i\mu_j$. Then the maps  $\psi_j : \C\to\C$, $j=1,2$,   defined by
 \begin{equation*}
 \left\{
  \begin{aligned}
  \psi_j(z)&:=\exp\left(\mu_j^{-1}g(\lambda_j \log z)\right) \hbox{ if } z\neq 0\,,\\
 \psi_j(0)&:=0\,,  \end{aligned}
 \right.
 \end{equation*}
 are well defined homeomorphisms and the map 
\[
\Psi:\C^2\to\C^2\,,\quad\Psi(z_1,z_2)=\big(\psi_1(z_1),\psi_2(z_2)\big)
 \]
is a homeomorphism that conjugates the foliation defined by the meromorphic form $\lambda_1\frac{dz_1}{z_1}+\lambda_2\frac{dz_2}{z_2}$ to that defined by $\mu_1\frac{dz_1}{z_1}+\mu_2\frac{dz_2}{z_2}$.
\end{lema}
\begin{proof}
The map
\[(\mu_1\log z_1+\mu_2\log z_2)\circ\Psi
=g(\lambda_1\log z_1+\lambda_2\log z_2)\]
is constant along the leaves of $\lambda_1\frac{dz_1}{z_1}+\lambda_2\frac{dz_2}{z_2}$.
\end{proof}

\noindent Thus, in case (\ref{freecomp}) when $\F^\sharp$  is of Poincar\'e type at $s$, then it is  also of Poincar\'e type at each singular point of the free component $\mc D$. 
Using similar methods as those developed by E. Paul in \cite{Paul} one can see that the local conjugacies given by Lemma~\ref{Poincare} at the singular points of the free component can be glued into a global $\mc C^0$-conjugacy along the whole free component. 
Finally, for  two topologically conjugated foliations whose reductions contain a Poincar\'e type free  component, the equality (\ref{equalityCS}) may not be satisfied  at the singular points of the free component. 
This phenomenon forces us to exclude this case (\ref{freecomp}) with the following assumption that there is no free dynamical component:

\begin{enumerate}
\item[(NFC)] \it There is no dynamical component of $S_\F^\sharp$ without nodal singular points of $\F^\sharp$ whose compact irreducible components contain exactly two singular points of $\F^\sharp$.
\end{enumerate}

Notice that if $\phi:\F\to\G$ is a topological conjugacy between generalized curves and $\F$ satisfies (NFC) then $\G$ also satisfies (NFC) thanks to Corollary~\ref{compdynsansnoeud}.

\subsection{Big dynamical components}\label{bigdyncomp}
 To obtain equality (\ref{equalityCS}) in case (\ref{bigcomp}) we will ask $\F$ to fulfill the following \emph{transverse rigidity} assumption:
 \begin{enumerate}
 \item[(TR)] \it If a dynamical component of $S_\F^\sharp$ 
contains an irreducible component with at least $3$ singular points of $\F^\sharp$,  it also contains an irreducible component whose holonomy group is topologically rigid, for instance unsolvable.
 \end{enumerate}
 We recall that a group $G$ of germs of $\mr{Diff}(\C,0)$ is \emph{topologically rigid} if any orientation preserving homeomorphism germ that conjugates this group with another subgroup of $\mr{Diff}(\C,0)$ is necessarily holomorphic. D. Cerveau and P. Sad showed that non abelian subgroups with generic linear part are topologically rigid \cite{CerveauSad} and by classical results of Y. Ilyashenko,  A.A. Shcherbakov and I. Nakai  \cite{Ilyashenko, Shcherbakov, nakai}  any unsolvable subgroup is topologically rigid. Condition (TR) above is Krull generic \cite{LeFloch}. We recall that a property $\ms P$ on the space $\Omega$ of germs  of holomorphic differential $1$-forms on $(\C^2,0)$ is \emph{Krull generic} if it is open dense for the Krull topology on $\Omega$, i.e. 
 \begin{enumerate}[(a)]
 \item for any $\omega\in\Omega$ and for any $p\in\N$ there is  $\eta\in\Omega$ satisfying $\ms P$ with same $p$-jet at $0$ as $\omega$,
 \item if $\omega$ satisfies $\ms P$ then there is $q\in\N$ such that every $\xi\in\Omega$ with same $q$-jet as $\omega$ also satisfies $\ms P$.
 \end{enumerate}
The interest of Condition (TR) is that we can apply the main theorem of J. Rebelo  in   \cite{Rebelo}, whose proof remains  valid for the extended statement below:

 \begin{teo}[Transverse Rigidity Theorem]\label{transrigthm} 
Let $\F$ be a generalized curve  and let $\mc D$ be a dynamical component  of $S_\F^\sharp$
 containing an irreducible component with topologically rigid holonomy group. Then $\F$ is \emph{transversally rigid along $\mc D$} in the following sense: for any homeomorphism $\phi:(\C^2,0)\to(\C^2,0)$  that conjugates $\F$ to  another germ of foliation preserving the orientations of the ambient space and of the leaves,
 there is an open neighbourhood $W$ of $\mc D\setminus \mc N$ such that 
$\phi$ is transversally holomorphic at each point of $E_\F(W)\setminus\{0\}$, where $\mc N\subset \mc E_\F$ is the set of nodal singular points of~$\F^\sharp$.
\end{teo}

\noindent We recall that a germ of homeomorphism $\phi$ which  conjugates two  foliations $\F$ and  $\G=\phi(\F)$, is \emph{transversally holomorphic at a point $p$}, if for any germs of submersion first integrals $u:(\C^2,p)\to(\C,0)$ and $v:(\C^2,\phi(p))\to(\C,0)$ there exists a germ of biholomorphism 
$\phi^\flat:(\C,0)\to(\C,0)$ such that $\phi^\flat\circ  u=v\circ \phi$. 

Notice that for a $\mc C^0$-conjugacy defined  only  on a neighborhood of a punctured separatrix $C\setminus\{0\}$, transversal holomorphy property is not suficient to obtain the invariance of the Camacho-Sad index  along the strict transform of $C$. A trivial counter-example with $C$ regular is given  by performing  the blowing-up $E$ with center $\{0\}$ and by taking for $\G$ the germ of  $E^{-1}\F$ at
the meeting point $p$ of the strict transform $C'$ of $C$ by $E$ and the exceptional divisor $E^{-1}(0)$. Restricted to   a neighborhood $V$ of $C'\setminus \{p\}$, $E$ is a biholomorphism that conjugates $\G_{|V}$ to $\F_{|E(V)}$. However, taking $\phi=(E_{|V})^{-1}$  we have 
\[ \mr{CS}(\F,C,0)=\mr{CS}(\phi(\F),C',p)+1\,,\]see  \cite[Proposition~2.1]{CS}.
We thus must impose as additional
hypothesis the compatibility of the conjugacy with the peripheral structure of $C$ in the fundamental group of the complementary of the $\F$-appropriate curve $S$. 
\subsection{Peripheral structure and Index Invariance Theorem}
We fix local coordinates $(u,v)$ at the meeting point $s$ of the strict transform $\mc C$ of $C$ with the exceptional divisor $\mc E_\F$, such that near $s$  the equation $v=0$ defines   $\mc C$ and $u=0$ defines $\mc E_\F$;  we consider the loop   $m$, resp. $p$, with common origin $a:=(\varepsilon_1,\varepsilon_2)$, defined by $u(t)=\varepsilon_1$, $v(t)=\varepsilon_2e^{2\pi it}$, resp. by  $u(t)=\varepsilon_1e^{2\pi it}$, $v(t)=\varepsilon_2$, $t\in[0,1]$; we choose $\varepsilon_1$ and $\varepsilon_2$ small enough so that the loops $E_\F\circ m$ and $E_\F\circ p$ are contained in a Milnor ball $B$ for the foliation $\F$; let us fix  $b\in B\setminus S$ and a path $\alpha$ in $B\setminus S$ joining $a$ to $b$. We call \emph{peripheral structure of $C$ in $\pi_1(B\setminus S,b)$} the data of the two classes $\mbf m$, $\mbf p\in \pi_1(B\setminus S,b)$, respectively  called \emph{meridian} and \emph{parallel}, defined by  the path 
\[ \alpha^{-1}{\mathsmaller\vee}(E_\F\circ m) {\mathsmaller\vee} \alpha\quad \hbox{ and }\quad \alpha^{-1}{\mathsmaller\vee}(E_\F\circ p) {\mathsmaller\vee} \alpha\,.\] A consequence of the fact \cite{Heil} that the  subgroup $\langle \mbf m,\mbf p\rangle\subset \pi_1(B\setminus S,b)$ is equal to its normalizer in $\pi_1(B\setminus S,b)$, 
  is that $\mbf m$ and $\mbf p$ do
not depend on the choice of the path $\alpha$, see \cite[Corollary~6.1.4]{MM3}. 
\begin{teo}\label{CStranshol}
Let $\phi$ be a homeomorphism defined on an open tubular neighborhood $V$ of  $C\setminus\{0\}$ in $\C^2$ that conjugates $\F_{|V}$ to $\G_{|\phi(V)}$. If $\phi$ is transversally holomorphic and sends respectively the meridian and the parallel of the peripheral structure  of $C$ to  the meridian and the parallel of $\phi(C)$, then the Camacho-Sad index of $\F^\sharp$ along the strict transform of $C$ by $E_\F$ is equal to that of $\G^\sharp$ along the strict transform of $\phi(C)$ by $E_{\G}$.
\end{teo}

\begin{proof}
Let $(u,v)$ (resp. $(u',v')$) be local coordinates centered at the intersection point $s$ (resp. $s'$) of the strict transform $\mc C=\{v=0\}$ (resp. $\mc C'=\{v '=0\}$) of $C$ (resp. $\phi(C)$) with the exceptional divisor $\mc E_\F=\{u=0\}$ (resp. $\mc E_\G=\{u'=0\}$).
By the transversal holomorphy of $\phi$ the holonomies of $\mc C$ and $\mc C'$ are analytically conjugated and consequently 
\begin{equation}\label{CS1}
\mr{CS}(\F^\sharp,\mc C,s)-\mr{CS}(\G^\sharp,\mc C',s')\in\Z.
\end{equation}
 Let $\Gamma_n$ be the lift of the loop $t\mapsto (\varepsilon_1 e^{2i\pi nt},0)$ to a leaf of $\F^\sharp$. Formula (\ref{CS-int}) implies that
 \[
 \mr{CS}(\F^\sharp,\mc C,s)=\lim\limits_{n\to\infty}\frac{1}{2i\pi n}\int_{\Gamma_n}\frac{dv}{v}\,.
 \] 
Let $\xi_n$ be a path contained in $\{u=\varepsilon_1, 0<|v|<1/n\}$ joining the points $\Gamma_n(1)$ and $\Gamma_n(0)$ in such a way that the variation of the argument of $v\circ\xi_n$ is bounded, consequently
\[
\lim\limits_{n\to\infty}\mr{Re}\left( \frac{1}{2i\pi n }\int_{\xi_n}\frac{dv}{v}\right)=0
\]
and
\begin{equation}\label{CS2}
\mr{Re}(\mr{CS}(\F^\sharp,\mc C,s))=\lim\limits_{n\to\infty}\frac{I_n}{n},\quad\text{where}\quad I_n:=\frac{1}{2i\pi}\int_{\Gamma_n\vee \xi_n}\frac{dv}{v}\in\Z\,.
\end{equation}
The homotopy class of $E_\F(\Gamma_n\vee\xi_n)$ can be decomposed as $I_n\mbf m_n+n \mbf p_n$ using the peripheral structure $\mbf m_n,\mbf p_n$ of $C$  in $\pi_1(V\setminus C,b_n)\hookrightarrow\pi_1(B\setminus S,b_n)$, 
$b_n=E_\F(\Gamma_n(0))$. 
Up to isotopy we can assume that $\phi^\sharp\circ u=u'\circ\phi^\sharp$ in a neighbourhood of the circle $\{|u|=\varepsilon_1,v=0\}$, where $\phi^\sharp=E_\G^{-1}\circ\phi\circ E_\F$. As the variation of the argument of $v'\circ\phi^\sharp\circ\xi_n$ is bounded, because $\phi$ is transversely holomorphic, we have 
\begin{equation}\label{CS3}
\mr{Re}(\mr{CS}(\G^\sharp,\mc C',s'))=\lim\limits_{n\to\infty}\frac{I_n'}{n},\quad\text{where}\quad I_n':=\frac{1}{2i\pi}\int_{\phi^\sharp(\Gamma_n\vee \xi_n)}\frac{dv'}{v'}\in\Z\,,
\end{equation}
and the homotopy class of $\phi(E_\F(\Gamma_n\vee \xi_n))$ can be decomposed as $I_n'\mbf m'_n+n\mbf p'_n$, where $\mbf m'_n,\mbf p'_n$ is the peripheral structure of $C'$ in $\pi_1(\phi(V\setminus C),\phi(b_n))\hookrightarrow\pi_1(B'\setminus S',\phi(b_n))$. By hypothesis the group morphism
\[
\phi_*:\pi_1(V\setminus C,b_n)\to\pi_1(\phi(V\setminus C),\phi(b_n))
\]
sends $\mbf m_n$ into $\mbf m_n'$ and $\mbf p_n$ into $\mbf p'_n$. Since 
\[
\phi_*(I_n\mbf m_n+n\mbf p_n)=I_n\mbf m_n'+n\mbf p_n'\,,
\]
we obtain that $I_n=I_n'$. Hence relations (\ref{CS1}), (\ref{CS2}) and (\ref{CS3}) give the equality
$\mr{CS}(\F^\sharp,\mc C,s)=\mr{CS}(\G^\sharp,\mc C',s')$.
\end{proof}
\color{black}

In addition global homeomorphisms always fulfill the peripheral structure assumption:

 \begin{prop}\label{invperiphstructure}\cite[Theorem 6.2.1]{MM3}
Let us consider  Milnor balls $B$ and $B'$ for germs of curves $S$ and $S'$ at the origin of $\C^2$ and $b\in B\setminus S$.  
For any homeomorphism $\phi:B\to \phi(B)\subset B'$
such that $\phi(S\cap B)=S'\cap\phi(B)$,
 the map $\phi_\ast:\pi_1(B\setminus S,b)\to\pi_1(B'\setminus S',\phi(b))$ sends respectively  the meridian and the parallel of the peripheral structure in $\pi_1(B\setminus S,b)$ of any irreducible component $C$ of $S$ to  the meridian and the parallel of the peripheral structure of $\phi(C)$ in $\pi_1(B'\setminus S',b)$.
\end{prop}

Theorem~\ref{CStranshol}, Proposition~\ref{invperiphstructure}
 and the  previous analysis of small dynamical components give us:

\begin{teo}[Index Invariance Theorem]\label{invthm}
For generalized curves satisfying assumptions (NFC) and~(TR), the Camacho-Sad indices after reduction of singularities are topological invariants in the sense that equalities (\ref{equalityCS}) are satisfied for any  $\mc C^0$-conjugacy $\phi:\F\to\G$.
\end{teo}

The hypothesis of this theorem are minimal.
Condition (NFC) is necessary as we have already seen in \S\ref{smalldyncomp}.
Counter-examples not fulfilling Condition (TR) are given by families of logarithmic foliations, whose topological classification is obtained by E. Paul:
\begin{teo}[{\cite[Th\'eor\`eme 3.5]{Paul2}}]\label{thmpaul}
Let $\omega_t$ be a family of germs of logarithmic forms
\[f_1\cdots f_p\sum_{j=1}^p\lambda_{j,t}\frac{df_j}{f_j},\quad t\in(\C,0),\]
on $(\C^2,0)$
whose residues $\lambda_{j,t}\in\C$ are without relation with positive integer coefficients. If there is a family germ $\{g_t, t\in(\C,0)\}$ of $\mr{GL}_2^+(\R)$ such that each map $g_t$ sends $2i\pi\lambda_{j,0}$ into $2i\pi\lambda_{j,t}$, then the family germ of foliations 
$\{\omega_t, t\in(\C,0)\}$ is topologically trivial.
\end{teo}

\section{Excellence Theorem and topological moduli space}
\subsection{Excellence Theorem} 
\noindent  
On the set of foliation germs satisfying (NFC) and (TR) the equivalence relations of $\mc C^{0}$-conjugacy and $\mc C^{\mr{ex}}$-conjugacy coincide: 

\begin{teo}[Excellence Theorem]\label{exTheorem} Let $\F$ and $\mc G$ be  two topologically equivalent germs  of foliations at $0\in\C^2$ fulfilling (NFC) and  (TR) properties  stated in \S\ref{smalldyncomp} and \textsection \ref{bigdyncomp}. 
Then the following assertions are equivalent:
\begin{enumerate}
\item there exists a homeomorphism germ $\phi:(\C^2,0)\to(\C^2,0)$ that conjugates $\F$ to $\G$;
\item there exists a  homeomorphism germ  $\Phi : (M_\F,\mc E_\F)\to(M_\G,\mc E_\G)$  that conjugates $\F^\sharp$ to $\G^\sharp$; moreover $\Phi$ is holomorphic at each non-nodal singular point of $\F^\sharp$ and is   transversely holomorphic at each point of $\mc E_\F$  that is regular for $\F^\sharp$ and not contained in a dicritical component.
\end{enumerate}
\end{teo}

\begin{proof} It suffices to use Index Invariance Theorem \ref{invthm}, Classification Theorem \ref{ClassThm} and Transverse Rigidity Theorem \ref{transrigthm} as
indicated in Figure~\ref{dem}.
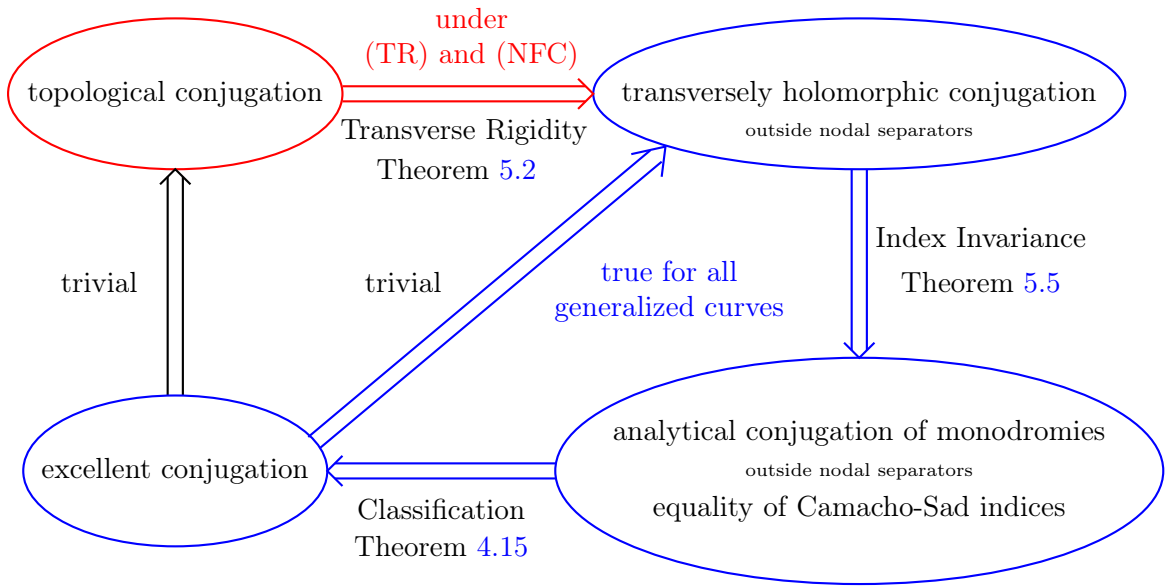
\begin{figure}[ht]
\begin{tikzpicture}[thick]
\node at (0,7) {topological conjugation};
\draw[red,thick] (0,7) ellipse (2.2cm and 1cm);

\node at (9,2.5) {analytical conjugation of monodromies};
\node at (9,2) {\tiny outside nodal separators};
\node at (9,1.5) {equality of Camacho-Sad indices};
\draw[blue,thick] (9,2) ellipse (4cm and 1.5cm);
\draw[blue]  (8.9,3.6) to (8.9,6);
\draw[blue]  (9.1,3.6) to (9.1,6);
\draw[blue]  (8.8,3.7) to (9,3.5) to (9.2,3.7);

\node at (9,7) {transversely holomorphic conjugation};
\node at (9,6.5) {\tiny outside nodal separators};
\draw[blue,thick] (9,7) ellipse (3.5cm and 1cm);

\draw (-0.1,3) to (-0.1,5.9);
\draw( 0.1,3) to (0.1,5.9);
\draw (-0.2,5.8) to (0,6) to (0.2,5.8);

\node at (0,2) {excellent conjugation};
\draw[blue,thick] (0,2) ellipse (2cm and 1cm);

\draw[red] (2.2,7.1) 
to (5.4,7.1);
\draw[red] (2.2,6.9)  
to (5.4,6.9);
\draw[red] (5.3,6.8) to (5.5,7) to (5.3,7.2);
\node at (3.9,8) {\color{red} under};
\node at (3.9,7.5) {\color{red}  (TR) and (NFC)};
\node at (3.8,6.5) {Transverse Rigidity};
\node at (3.7,6) {Theorem~\ref{transrigthm}};

\node at (-1,4.5) {trivial};
\node at (3,4.5) {trivial};

\draw[blue]  (5,2.1) to (2.1,2.1);
\draw[blue]  (5,1.9) to (2.1,1.9);
\draw[blue]  (2.2,2.2) to (2,2) to (2.2,1.8);

\node at (3.5,1.5) {Classification};
\node at (3.5,1) {Theorem~\ref{ClassThm}};

\draw[blue]  (1.9,2.3) to (6.4,6.1);
\draw[blue]  (1.75,2.45) to (6.25,6.25);
\draw[blue]  (6., 6.2) to (6.45, 6.3) to (6.4,5.9);

\node at (10.6,5.1) {Index Invariance};
\node at (10.6,4.5) {Theorem \ref{invthm}};

\node at (6.5,4.65) {\color{blue} true for all};
\node at (6.5,4.15) {\color{blue} generalized curves};
\end{tikzpicture}
\caption{Strategy of the proof of Excellence Theorem~\ref{exTheorem}.}\label{dem}
\end{figure}
\end{proof}

\subsection{Classification Problem: Complete families and Moduli space}\label{moduli} Excellence Theorem reduces the problem of topological classification of germs of foliations at the origin of $\C^2$ to the problem of classification under excellent germs of homeomorphisms.
This one  can be formulated  as follows. 
Let us call \emph{semi-local invariants} of a   foliation germ $\F$  the data of  the  topology of its  separatrices curve,  the Camacho-Sad indices of $\F^\sharp$  and  the holonomy representation morphisms of the invariant irreducible components of $\mc E_\F$.  All these data are $\mc C^{\mr{ex}}$-invariants and therefore topological invariants under the hypothesis of Excellence Theorem \ref{exTheorem}. Let us denote by $\text{SL}(\F)$ the set of all foliation germs with same semi-local invariants as $\F$. 
The goal of  $\mc C^0$-classification is  to construct a natural set of foliation germs
 whose  intersection with the topological class of any foliation in $\mr{SL}(\F)$ is non-empty and totally discontinuous.
For a large class of foliation germs a complete answer to this problem is given in the papers \cite{MMS,MMS2,MMS3} that lead to  a family of foliations in $\text{SL}(\F)$  depending holomorphically on a parameter that varies in a complex manifold of minimal dimension. This family has natural properties of factorization of any other holomorphic family in $\text{SL}(\F)$; moreover the redundancy of the topological types in this family is minimal and can be precisely described. \\

In the spirit of Teichm\"uller theory, it is convenient for our moduli problem to endow each foliation germ with a marking. For this, let us fix a foliation germ $\F_0$ that is a generalized curve and let us simply denote by $\mc E^\diam:=(\mc E,\Sigma)$ the pair $(\mc E_{\F_0},\mr{Sing}(\F_0^\sharp))$.
A \emph{marked by $\mc E^\diam$ foliation germ} is a pair $(\F,f)$ where $\F$ is a foliation germ at $(\C^2,0)$ and $f$ is a homeomorphism from $\mc E$ into $\mc E_\F$ such that 
\[ 
f(\Sigma)=\mr{Sing}(\F^\sharp)\quad\hbox{and}\quad f(D)\cdot f(D')=D\cdot D'
 \]
for any pair $(D,D')$ of irreducible components of $\mc E$. 
Two  $\mc E^\diam$-markings $f$ and $g$ of a same foliation $\F$  are \emph{equivalent} if $g^{-1}\circ f$ is isotopic to the identity map, by an isotopy fixing each point of $\Sigma$.
Two marked by $\mc E^\diam$ foliation germs $(\F,f)$ and $(\G,g)$ are \emph{$\mc C^{\mr{ex}}$-conjugated} if there exists a $\mc C^{\mr{ex}}$-conjugacy germ $\Phi:(M_\F, \mc E_\F)\to\ (M_\G,\mc E_\G)$ such that $g$ and $\Phi\circ f$ are equivalent $\mc E^\diam$-markings  of $\G$.
\begin{defin}
We say that two marked by $\mc E^\diam$ foliation germs $(\F,f)$ and $(\G,g)$ \emph{have same semi-local type}  if
\begin{enumerate}
\item  for any irreducible component $D$ of $\mc E$ and any $s\in \Sigma\cap D$, we have the equality of Camacho-Sad indices  $\mr{CS}(\G^\sharp, g(D), g(s))= \mr{CS}(\F^\sharp, f(D), f(s))$
\item  for any irreducible  component $D$ of $\mc E$ with $D\cap\Sigma\neq\emptyset$, the compositions with the group morphisms induced by the markings
\[  
H_{D}^\F\circ f_*\,\hbox{ and }\;\; H_D^\G\circ g_*\;:\;\pi_1(D\setminus \Sigma,\cdot)\longrightarrow \mr{Diff}(\C,0)
\]
 of the  holonomy morphisms $H_{f(D)}^\F$ and $H_{g(D)}^\G$ of the foliations $\F^\sharp$ and $\G^\sharp$ along $f(D)$ and $g(D)$,
are equal up to composition by an inner automorphism of $\mr{Diff}(\C,0)$.
\end{enumerate}
\end{defin}
To  obtain a ``parametrization'' by a finite dimensional manifold of the set of $\mc C^{\mr{ex}}$-conjugacy classes of marked foliation germs with fixed semi-local type, we need an additional hypothesis called  \emph{finite type}. 
This one was first introduced in \cite{MMS} as a combinatorial property of an appropriate orientation of the dual graph $\A_\F$ and then reinterpreted in \cite{MMS,MMS2} as the finiteness dimension $\tau(\F)$ of a naturally associated cohomological space.
This  assumption is not very restrictive because in the set of $1$-differential forms defining generalized curves, the subset of finite type foliations contains an open dense set for the Krull topology, see \cite[Theorem 6.2.1]{MS} and \cite[Remark 2.5]{MMS}. 
\begin{teo}\label{complThm}\cite[Theorem 3.7]{MMS3} Let $(\F,f)$ be a marked finite type foliation germ which is a generalized curve. Then there exists an equisingular marked family of foliation  germs
\[\left(\F_{z,d},f_{z,d}\right)_{(z,d)\in\C^\tau\times\mbf{D}}\] 
such that
\begin{enumerate}
\item\label{CtauD} $\mbf{D}$ is a quotient of a finite product of totally disconnected subgroups of the circle $\mb U(1)$ and $\tau=\tau(\F)$,
\item\label{comp} 
for any marked foliation  germ
$(\G,g)$ with same semi-local type as $(\F,f)$, there exists $(z,d)\in\C^\tau\times \mathbf D$ such that $(\G,g)$ is $\mc C^{\mr{ex}}$-conjugated to $(\F_{z,d},f_{z,d})$,
\item\label{FacFam} if $(\G_{t},g_t)_{t\in P}$ is an equisingular marked family of   foliation germs with parameter space  a connected manifold $P$ satisfying $H^1(P,\Z)=0$, then for any  $t_0  \in P$ and $( z_0, d_0)\in\C^\tau\times \mathbf D$ such that $(\G_{t_0},g_{t_0})$ is $\mc C^{\mr{ex}}$-conjugated to  $\F_{z_0,d_0}$, there exists a unique holomorphic map $\lambda:P\to \C^\tau\times \mathbf D$ with $\lambda(t_0)=(z_0,d_0)$ such that  for any $t\in P$ the marked foliation germs $(\G_t,g_t)$ and $(\F_{\lambda(t)},f_{\lambda(t))})$ are $\mc C^{\mr{ex}}$-conjugated,
\item\label{improvement} in a neighborhood $W_{t'}$ of any point $t'\in P$, 
the  $\mc C^\mr{ex}$-conjugacies between the marked foliation germs $(\G_t,g_t)$ and $(\F_{\lambda(t)},f_{\lambda(t))})$ in Assertion (\ref{FacFam}) can be realized by excellent homeomorphisms $\Phi_{t}$ depending continuously on $t\in W_{t'}$.
\end{enumerate}
\end{teo}
\noindent The notion of \emph{marked family of foliation germs}, 
introduced in \cite[\S 2.2]{MMS3}, is 
a compatibility property of the markings with the local topological triviality in the parameter of the whole exceptional divisor. That of \emph{equisingularity} is precisely  defined in \cite[\S3]{MMS2}; basically it means that the family is locally equireducible and  all the foliation germs in the family have same semi-local type. 
 Assertion (\ref{improvement}) in Theorem~\ref{complThm} follows from \cite[Theorem~4.4]{MMS3}.\\

We finally describe the redundancy of the $\mc C^{\mr{ex}}$-classes in the family $\left(\F_{z,d},f_{z,d}\right)_{(z,d)\in\C^\tau\times\mbf{D}}$. Let us define the \emph{$\mc C^{\mr{ex}}$-moduli space of $(\F,f)$} as the set
\begin{equation*}
\mr{Mod}(\F,f):=\{\;  [\G,g]
\;|\; \text{$(\G,g)$ and $(\F,f)$ have same semi-local type}
\}
\end{equation*}
of $\mc C^{\mr{ex}}$-conjugacy classes $[\G,g]
$ of marked by $\mc E^\diam$ foliations $(\G,g)$. 
To any equisingular marked family $\G_P:=(\G_t,g_t)_{t\in P}$ of foliation germs with same semi-local type as $(\F,f)$ we associate its \emph{moduli map}
\[ 
\text{mod}_{\G_P}:P\longrightarrow \mr{Mod}(\F,f)\,,\qquad t\mapsto [\G_t,g_t]
\,.
 \]
\begin{teo}\cite[Theorem D]{MMS}\label{modules} Let $(\F,f)$ be a marked finite type foliation germ which is a generalized curve.  There is a natural abelian group structure on $\mr{Mod}(\F,f)$ given by an exact sequence 
\begin{equation*}
0\to \Z^p\xrightarrow{\alpha}\C^\tau\xrightarrow{\Lambda} \mr{Mod}(\F,f) \xrightarrow{\beta}\mathbf D\rightarrow 0
\end{equation*}
and for any section $\zeta$ of $\beta$ there exists an equisingular marked family 
\[
\F_{\C^\tau\times\mathbf D}:= 
(\F_{z,d},f_{z,d})_{(z,d)\in\C^\tau\times\mathbf D}\] 
of foliation germs with same semi-local type as $(\F,f)$
such that
\begin{enumerate}
\item all assertions 
in Theorem \ref{complThm} are satisfied,
\item the moduli map $\mr{mod}_{\F_{\C^\tau\times\mathbf D}}$ is surjective and for any $(z,d)\in\C^\tau\times \mathbf D$ we have
\[ 
\mr{mod}_{\F_{\C^\tau\times\mathbf D}}(z,d)=\Lambda(z)\cdot \zeta(d)
 \]
 where $\cdot$ is the group operation in $\mr{Mod}(\F,f)$.
\end{enumerate}
\end{teo}
\noindent Notice that the surjectivity of $\text{mod}_{\F_{\C^\tau\times\mathbf D}}$ is equivalent to property (\ref{comp}) in Theorem~\ref{complThm}.\\

A direct consequence of Theorem~\ref{modules} 
is the explicit description of the redundancy of the $\mc C^\mr{ex}$-classes 
in the family $\F_{\C^\tau\times\mbf D}$:
\begin{cor}
Two marked foliation germs $(\F_{z,d},f_{z,d})$ and $(\F_{z',d'},f_{z',d'})$ in the marked family $\F_{\C^\tau\times\mathbf D}$ are $\mc C^\mr{ex}$-conjugated if and only if $d=d'$ and there exists $N\in\Z^p$ such that $z'=z+\alpha(N)$.
\end{cor}

\bibliographystyle{plain}

\end{document}